\documentclass[letterpaper,10pt,reqno,english]{amsart}
\usepackage{babel}
\usepackage{amsmath,amsthm,amssymb,amsfonts,indentfirst}
\usepackage{color}
\usepackage[title]{appendix}
\usepackage[colorlinks]{hyperref}
\hypersetup{linkcolor=blue,citecolor=blue,filecolor=black,urlcolor=blue}

\newtheorem{theo}{Theorem}

\newtheorem{lem}{Lemma}
\newtheorem{prop}{Proposition}
\newtheorem{claim}{Claim}
\theoremstyle{remark}

\numberwithin{equation}{section} \numberwithin{theo}{section}
\numberwithin{cor}{section} \numberwithin{lem}{section}
\numberwithin{prop}{section} \numberwithin{claim}{section}
\numberwithin{obs}{section} \numberwithin{dfn}{section}

\newcommand\R{\text{I\!R}}
\newcommand\N{\text{I\!N}} 

\newcommand\Q{\mathbb{Q}}

\newcommand\e{\epsilon}
\newcommand\de{\delta}
\newcommand\be{\beta}
\newcommand\al{\alpha}

\newcommand{\Om}{\Omega}
\newcommand{\oen}{\Omega_{\e_n}}
\newcommand{\fr}{\partial}

\newcommand{\grad}{\nabla}
\newcommand{\ml}{\mathcal}

\newcommand{\sm}{\setminus}
\newcommand{\la}{\lambda}
\newcommand{\st}{such that }
\newcommand{\dem}{\bf Proof:}

\newcommand\lap{\Delta}

\newcommand\ti{\tilde}

\newcommand{\lf}{\left}
\newcommand{\rg}{\right}
\renewcommand\({\left(}
\renewcommand\){\right)}

\newcommand\ds{\displaystyle}

\newcommand{\bebs}{\begin{equation*}\begin{split}}
\newcommand{\ee}{\end{equation*}}
\newcommand{\esp}{\end{split}}
\newcommand\ms{\medskip}
\newcommand\bs{\bigskip}

\DeclareMathAlphabet{\mathpzc}{OT1}{pcz}{m}{it}

\begin{document} 
\title[Sign-changing bubble tower solutions 
on pierced domains
]{Sign-changing bubble tower solutions for sinh-Poisson type equations on pierced domains}
\author[P. Figueroa]{Pablo Figueroa}
\address{Pablo Figueroa
\newline \indent Instituto de Ciencias Físicas y Matemáticas
\newline \indent Facultad de  Ciencias 
\newline \indent Universidad Austral de Chile
\newline \indent Campus Isla Teja s/n, Valdivia, Chile} 
\email{pablo.figueroa@uach.cl}

\date{\today}



\maketitle

\begin{abstract}
\noindent For asymmetric sinh-Poisson type problems with Dirichlet boundary condition arising as a mean field equation of equilibrium turbulence vortices with variable intensities of interest in hydrodynamic turbulence, we address the existence of sign-changing bubble tower solutions on a pierced domain $\Om_\e:=\Om\sm \ds \overline{B(\xi,\e)}$, where $\Om$ is a smooth bounded domain in $\R^2$ and  $B(\xi,\e)$ is a ball centered at $\xi\in \Om$ with radius $\e>0$. Precisely, given a small parameter $\rho>0$ and any integer $m\ge 2$, there exist a radius $\e=\e(\rho)>0$ small enough such that each sinh-Poisson type equation, either in Liouville form or mean field form, has a solution $u_\rho$ with an asymptotic profile as a sign-changing tower of $m$ singular Liouville bubbles centered at the same $\xi$ and with $\e(\rho)\to 0^+$ as $\rho$ approaches to zero.
\end{abstract}

\section{Introduction}
Let $\Om$ be a smooth bounded domain in $\R^2$. Given $\e>0$ and $\xi\in \Om$, define $\Om_\e:=\Om\sm \ds \overline{B(\xi,\e)}$, a pierced domain, where $B(\xi,\e)$ is a ball centered at $\xi$ with radius $\e$. Inspired by results in \cite{AP,EFP,F1,GP,pr2},  we are interested in constructing sign-changing \emph{bubble tower} solutions to sinh-Poisson type equations, either in Liouville form or mean field form, with variable intensities and Dirichlet boundary conditions on a pierced domain $\Om_\e$. Precisely, on one hand, we consider the following problem in Liouville form
\begin{equation}\label{spsh}
\left\{ \begin{array}{ll} 
-\lap u=\rho(V_0(x)e^{u} - \nu V_1(x) e^{-\tau u})&\text{in $\Om_\e$}\\
\ \ u=0 &\text{on $\fr \Om_\e$}
\end{array} \right.,
\end{equation}
and, on the other hand, we also study the problem in mean field form
\begin{equation}\label{mfsh}
\left\{ \begin{array}{ll} 
\ds-\lap u={\la_0V_0(x)e^{u}\over \int_{\Om_\e} V_0e^{u}}  - {\la_1\tau V_1(x) e^{-\tau u}\over \int_{\Om_\e}V_1 e^{-\tau u}} &\text{in $\Om_\e$}\\
\ \ u=0 &\text{on $\fr \Om_\e$}
\end{array} \right.,
\end{equation}
where $\rho>0$ is small, $\la_0,\la_1>0$, $V_0,V_1>0$ are smooth potentials in $\Om$, $\tau>0$, $\e>0$ is a small number and $\nu\ge 0$. Our aim is to construct for each problem a family of solutions $u_\rho$ for a suitable choice of $\e=\e(\rho)$, with an asymptotic profile as a sum of positive and negative singular Liouville bubbles centered at the same point $\xi$ as $\rho\to 0$, on the line of \cite{GP,pr2}.

These equations and related ones have attracted a lot of attention in recent years due to its relevance in the statistical mechanics description of 2D-turbulence, as initiated by Onsager \cite{o}. Precisely, in this context Caglioti, Lions, Marchioro, Pulvirenti \cite{clmp} and Sawada, Suzuki \cite{ss} derive the following equation:
\begin{equation}\label{p1}
\ds -\Delta u=\lambda \int\limits_{[-1,1]}{\alpha e^{\alpha u}\over \int\limits_\Omega e^{\alpha u}dx}d \mathcal P(\alpha) \ \ \hbox{ in } \Omega, \ \ \ \ u=0  \ \hbox{ on } \partial\Omega,
\end{equation}
where $\Omega$ is a bounded domain in $\mathbb R^2,$ $u$ is the stream function of the flow, $\lambda>0$ is a constant related to the inverse temperature and  $\mathcal P$ is a Borel probability measure in $[-1,1]$ describing the point-vortex intensities distribution. We observe that \eqref{p1} is obtained under a \emph{deterministic} assumption on the distribution of the vortex circulations. 

\medskip \noindent Equation \eqref{p1} 
includes several well-known problems depending on a suitable choice of $\ml P$. For instance, if $\mathcal P=\delta_1$ is concentrated at $1$, then \eqref{p1} corresponds to the classical mean field equation
\begin{equation}\label{p2mf}
-\Delta u=\lambda  { e^{ u}\over \int\limits_\Omega e^{  u}dx} \ \ \hbox{ in } \Omega, \ \ \ \ u=0 \ \hbox{ on } \partial\Omega.
 \end{equation}
Since there are plenty of results in the literature concerning \eqref{p2mf}, let us just quote  \cite{BL,clmp1,CL1,CL2,Dja,Mal}.  When $\mathcal P=\sigma \delta_{1}+(1-\sigma)\delta_{-\tau }$ with $\tau \in[-1,1]$ and $\sigma\in[0,1]$, equation \eqref{p1} becomes
\begin{equation}\label{p21}
-\Delta u=\lambda \bigg( \sigma{ e^{ u}\over \int\limits_\Omega e^{  u}dx} -(1-\sigma)\tau { e^{ -\tau  u}\over \int\limits_\Omega e^{ -\tau   u}dx}
\bigg) \ \ \hbox{ in } \Omega, \ \ \ \ u=0 \ \hbox{ on }\ \partial\Omega.
 \end{equation}
Setting $\lambda_0=\lambda\sigma$, $\lambda_1=\lambda(1-\sigma)$ and $V_0=V_1=1$ problem \eqref{p21} can be rewritten as \eqref{mfsh} replacing $\Om_\e$ by $\Om$.  
If $\tau =1$ and $V_0=V_1\equiv 1$ problem \eqref{mfsh} reduces to mean field equation of the equilibrium turbulence, see \cite{bjmr,j,jwy2,os1,r} or its related sinh-Poisson version, see \cite{BaPi,BaPiWe,GP,jwy1,jwyz}, which have received a considerable interest in recent years. Recently, sign-changing solutions have been constructed in $\Om$ for the sinh-Poisson equation with Robin boundary condition in \cite{FIT}.

\medskip \noindent Concerning results for $\tau>0$, 
Pistoia and Ricciardi built in \cite{pr1} sequences of blowing-up solutions   to \eqref{mfsh} (in $\Om$ instead $\Om_\e$) when $\lambda_0,\lambda_1\tau^2$ are close to $8\pi$. Ricciardi and Takahashi in \cite{rt}  provided a complete blow-up picture for solution sequences of \eqref{mfsh} and successively  in \cite{rtzz} Ricciardi et al. constructed min-max solutions  when $\lambda_0 \to 8\pi^+$ and   $\lambda_1 \to 0$ on a multiply connected domain (in this case the nonlinearity $e^{-\tau  u}$ may  be treated as a lower-order term with respect to the main term $e^u$). A blow-up analysis and some existence results are obtained when $\tau>0$ in a compact Riemann surface in \cite{j2,rz}. Bubbling solutions in a compact Riemann surface has been recently constructed in \cite{F2}. 

\medskip
\noindent On the other hand, on pierced domains, Ahmedou and Pistoia in \cite{AP} proved that on $\Omega_\epsilon$, there exists a solution to the classical mean field equation \eqref{p2mf} which blows-up at $\xi$ as $\epsilon \to 0$ for any $\lambda>8\pi$ (extra symmetric conditions are required when $\la \in 8\pi \mathbb{N}$). In \cite{EFP} the authors constructed a family of solutions to  the mean field equation with variable intensities \eqref{mfsh} blowing-up positively and negatively at $\xi_1,\dots,\xi_{m_1}$ and $\xi_{m_1+1},\dots,\xi_m$, respectively, as $\epsilon_1,\dots,\epsilon_m \to 0$  on a pierced domain with several holes ($\Om_\e$ is replaced by in $\Om \setminus \cup_{i=1}^m \overline{B(\xi_i,\e_i)}$ ), in the super-critical regime $\lambda_0>8\pi m_1$ and $\lambda_1 \tau^2>8\pi (m-m_1)$ with $m_1 \in \{0,1,\dots,m\}$. Recently, in the same spirit of \cite{EFP}, the author in \cite{F1} addressed the sinh-Poisson type equation with variable intensities \eqref{spsh} on a pierced domain with several holes $\Om \setminus \cup_{i=1}^m \overline{B(\xi_i,\e_i)}$. Equation \eqref{spsh} is related, but not equivalent, to problem \eqref{mfsh} by using the change
$$\rho={\la_0\over \int_{\Om_\e} V_0 e^{u} } \qquad \text{ and }\qquad \rho\nu ={\la_1\tau \over \int_{\Om_\e} V_1 e^{-\tau u} }.$$

 \noindent To the extent of our knowledge, there are by now just few results concerning non-simple blow-up 
for sinh-Poisson type problems. Precisely, sign-changing solutions with non-simple blow-up has been built in \cite{EW} 
for the Neumann sinh-Poisson equation. Grossi and Pistoia built in \cite{GP} a sign-changing bubble tower solutions for the sinh-Poisson version ($\tau=1$) in a symmetric domain $\Om$ with respect to a fixed point $\xi\in \Om$. After that, Pistoia and Ricciardi in \cite{pr2} extend this construction to a sinh-Poisson type equation with asymmetric exponents ($\tau\ne 1$) under a symmetric assumption on $\Om$ depending on either $\tau\in\Q$ or $\tau\notin\Q$. In both situations \cite{GP,pr2} the number of bubbles can be arbitrary large.

\medskip 
A matter of interest to us is whether do there exist sign-changing bubble tower solutions to \eqref{spsh} for small values of $\rho$ or to \eqref{mfsh} for some values of the parameters $\lambda_0,\lambda_1,\tau>0$ on a pierced domain $\Omega_\e$, with bubbles centered at $\xi$. Our first result in this direction without symmetry assumptions reads as follows.

\begin{theo}\label{main}
Let $m\ge 2$ be a positive integer. There exists $\rho_0>0$ such that for all $0<\rho<\rho_0$ there is $\e=\e(\rho)$ small enough such that problem \eqref{spsh} has a sign-changing solution $u_\rho$ in $\Om_\e$ blowing-up at $\xi$ in the sense that
\begin{equation}\label{aburo}
u_\rho=(-1)^{m+1}{2 \pi\over \tau^{\nu(m)}} (\al_m+2) G(\cdot,\xi) + o(1)
\end{equation}
locally uniformly in $\bar\Om \sm\{\xi\}$ as $\rho\to 0^+$.
\end{theo}

Here, for simplicity, we denote
\begin{equation}\label{nui}
 \nu(i)={1+(-1)^{i}\over 2} = \begin{cases}
 0&\text{ if $i$ is odd}\\
 1&\text{ if $i$ is even}
 \end{cases} 
\qquad\text{and}\qquad  \sigma(i)={1-(-1)^{i}\over 2} = \begin{cases}
1&\text{ if $i$ is odd}\\
0&\text{ if $i$ is even}
 \end{cases},
 \end{equation}
$\al_m$ is given by \eqref{ali} with $i=m$ and  $G(x,y)=-\frac{1}{2\pi}  \log |x-y|+H(x,y)$ is the Green's function of $- \lap$ in $\Om$, where the regular part $H$ is a harmonic function in $\Om$ so that $H(x,y)=\frac{1}{2\pi}  \log |x-y|$ on $\fr\Om$. 
Nevertheless, the latter result may not tell us whether \eqref{mfsh} has sign-changing bubble tower solutions for some values of the parameters $\lambda_0,\lambda_1,\tau>0$. Therefore, we perform directly to problem \eqref{mfsh} a similar procedure. Conversely, Theorem \ref{main2} below may not tell us whether \eqref{spsh} has sign-changing bubble tower solutions for \emph{all} small $\rho>0$. Assume that $\la_0$ and $\la_1$ decompose for some $\al_1>2$ (see \eqref{ali0}), as either
\begin{equation}\label{la0}
\la_0=2\pi m\lf[\al_1+(m-2)\Big(1+{1\over \tau}\Big)\rg],\quad \la_1\tau^2= 2\pi m \lf[\al_1\tau + m(1+\tau)\rg],\quad\text{if } m \text{ even}
\end{equation}
or
\begin{equation}\label{la1}
\la_0=2\pi (m+1)\lf[\al_1+(m-1)\Big(1+{1\over \tau}\Big)\rg], \ \ \la_1\tau^2= 2\pi (m-1) \lf[\al_1\tau + (m-1)(1+\tau)\rg],\ \text{ if }  m \text{ odd}.
\end{equation}
In particular, we choose $\al_1>2$ and $\al_1\notin 2\N$ if $\tau=1$. Our second result (without symmetry assumptions) is the following

\begin{theo}\label{main2}
If either \eqref{la0} or \eqref{la1} holds for a positive integer $m\ge 2$, then there exists a radius $\e>0$ small enough such that problem \eqref{mfsh} has a sign-changing solution $u_\e$ in $\Om_\e$ blowing-up at $\xi$ in the sense of \eqref{aburo}  locally uniformly in $\bar\Om \sm\{\xi\}$ as $\e\to 0^+$.
\end{theo}

\noindent Our solutions correspond to a superposition of highly concentrated vortex configurations of alternating orientation around the hole $B(\xi,\e)$ and they extend some known results \cite{GP,pr2} for symmetric domains. We also point out that a delicate point in the paper concerns the linear theory developed in Section \ref{sec3}. A bit more complicated analysis is necessary in comparison with linear theories developed in  previous works \cite{AP,DeKM,EFP,EGP,F,GP,pr2}. 

\medskip

\noindent Without loss of generality, we shall assume in the rest of the paper that $\xi=0\in\Om$, so that $\Om_\e=\Om\sm B(0,\e)$ where $\e>0$ is small and  $\nu=1$, since we can replace $\nu V_2$ by $V_2$. However, we need the presence of $\nu$ when we compare \eqref{spsh} with equation \eqref{mfsh}. 

\medskip \noindent Finally, we point out some comments about the proofs of the theorems. Following the ideas presented in \cite{GP,pr2} about bubble tower solutions for sinh-Poisson type equations and in \cite{AP,EFP,F} about construction of solutions on pierced domains, we find a solution $u_\rho$ using a perturbative approach, precisely, we look for a solution of \eqref{spsh} as
\begin{equation}\label{solurho}
u_\rho=U +\phi,
\end{equation}
where $U$ is a suitable ansatz built using the projection operator $P_\e$ onto $H^1_0(\Omega_\e)$(see \eqref{ePu})  and $\phi\in H_0^1(\Om_\e)$ is a small remainder term. Letting $w_{\de,\al}(x)=\log\frac{2\alpha^2\delta^{\alpha}}{(\delta^{\alpha}+|x|^{\alpha})^2}$ be a solution to the singular Liouville equation
\begin{equation*}
\Delta u+|x|^{\alpha-2}e^{u}=0 \quad \text{ in $\R^2$},\qquad
\ds\int_{\R^2} |x|^{\alpha-2} e^u <+\infty, 
\end{equation*}
the ansatz $U$ is defined as follows
\begin{equation}\label{ansatz}
U(x)=\sum_{i \text{ odd}} P_\e w_i(x)-\dfrac{1}{\tau} \sum_{i \text{ even}} P_\e w_i(x) = \sum_{i =1}^m  \dfrac{(-1)^{i+1}}{\tau^{\nu(i)}} P_\e w_i(x),
\end{equation}
where $w_j:= w_{\delta_j,\al_j}$ for $j=1,\dots,m$.
A careful choice of the parameters $\delta_j $'s, $\al_i$'s and the radius $\e$, depending on $\rho>0$, is made in section \ref{sec2}  (see \eqref{dei}, \eqref{ali0}-\eqref{ali} 
and \eqref{eps}) in order to make $U$ be a good approximated solution. Indeed, the error term $R$ for \eqref{spsh} given by
\begin{equation}\label{R} 
R=\Delta U+\rho(V_0(x) e^{U} - V_1(x)e^{-\tau U})
\end{equation} 
is small in $L^p$-norm for $p>1$ close to $1$ (see Lemma \ref{estrr0}). A linearization procedure around $U$ leads us to re-formulate \eqref{spsh} in terms of a nonlinear problem for $\phi$ (see equation \eqref{ephi}). We will prove the existence of such a solution $\phi$ to \eqref{ephi} by using a fixed point argument, thanks to some estimates in subsection \ref{sec4} (see \eqref{estlaphi} and \eqref{estnphi}). The corresponding solution $u_\rho$ in \eqref{solurho} behaves as a sign-changing tower of $m$ singular Liouville bubbles thanks to the asymptotic properties of its main order term $U$ (see \eqref{waje} in Lemma \ref{expaU}). In Section \ref{sec3} we will prove the invertibility of the linear operator naturally associated to the problem (see \eqref{ol}) stated in  Proposition \ref{p2}. To conclude Theorem \ref{main2}, the same procedure with the same ansatz is performed to equation \eqref{mfsh}, assuming \eqref{la0}-\eqref{la1} and $\e=\e(\rho)$, where the error term is given by
\begin{equation}\label{Rmf} 
\ml{R}=\Delta U+\la_0{V_0(x) e^{U}\over \int_{\Om_\e} V_0e^{U}} - \la_1\tau {V_1(x)e^{-\tau U}\over \int_{\Om_\e} V_1e^{-\tau U}}.
\end{equation} 


\section{Approximation of the solution}\label{sec2}
\noindent In this section we shall make a choice of the parameters $\al_i$'s, $\de_j$'s and $\e=\e(\rho)$ in order to make $U$ a good approximation.  Introduce the function $P_\e w$ as the unique solution of
\begin{equation}\label{ePu}
\left\{ \begin{array}{ll} 
\Delta P_\e w=\lap  w &\text{in }\Om_\e\\
P_\e w=0,&\text{on }\fr\Om_\e.
\end{array}\right.
\end{equation}
For simplicity, we will denote $h_0=H(0,0)$. We have the following asymptotic expansion of $Pw_{\de,\al}$ as
$\delta \to 0$, as shown in \cite[Lemma 2.1]{AP} (see also \cite[Lemma 2.1]{EFP}):
\begin{lem}\label{ewfxi}
The function $P_\e w_{\delta,\al}$ satisfies
$$P_\e w_{\delta,\al}(x)=w_{\delta,\al}(x)-\log(2\al^2\de^\al)+
4\pi\al H(x,0)-\gamma_{\de,\e}^\al G(x,0)+O\Big(\de^\al+\Big({\e\over \de}\Big)^{\al}+\Big[1+\Big|\frac{\log\de}{\log\e}\Big|\Big]\e\Big),$$
uniformly in $\Om_\e$, where $\gamma_{\de,\e}^\al$ is given by $\gamma_{\de,\e}^\al=\frac{-2\al \log\de +4\pi\al h_0 }{-{1\over2\pi}\log\e + h_0}.$ In
particular, there holds
$$P_\e w_{\delta,\al}(x)=[4\pi\al-\gamma_{\de,\e}^\al] G(x,0)+O\Big(\de^\al+\Big({\e\over \de}\Big)^{\al}+\Big[1+\Big|\frac{\log\de}{\log\e}\Big|\Big]\e\Big)$$
as $\e\to0$ locally uniformly in $\Om \sm\{0\}$.
\end{lem}

\noindent Given $m\in \N$ with $m\ge 2$, consider $\delta_j>0$ and $\al_j>2$, $j=1,\dots,m$, so that our approximating solution $U$ is defined by \eqref{ansatz}, 
parametrized by $\de_j$'s and $\al_i'$s with $\de_j=\de_j(\rho,\al_1)$ (it also depends on $\tau$, $h_0$, $V_0(0)$ and $V_1(0)$), where $\nu(i)$ is defined in \eqref{nui} and $P_\e$ the projection operator defined by \eqref{ePu} for a suitable choice of $\e$. In order to have a good approximation, for any $i=1,\dots,m$ we will assume that
 \begin{equation}\label{dei}
 \de_i^{\al_i}= d_i \rho^{\be_i},
 \end{equation}
 for small $\rho>0$, where $\al_i$'s, $\be_i$'s, and $d_i$'s will be specified below. We choose
 \begin{equation}\label{ali0}
\al_1>2,\quad\text{with}\quad \al_1\in\begin{cases}
\ds\bigcap_{k=0}^{(m-2)/2} \Big(2\N - {4k\over \tau}\Big)^c \bigcap \Big[\bigcap_{k=1}^{m/2} \Big( {2\over \tau}\N - 4k+2\Big)^c\Big]&\text{ if $m$ is even}\\[0.6cm]
\ds\bigcap_{k=0}^{(m-1)/2} \Big(2\N - {4k\over \tau}\Big)^c \bigcap \Big[\bigcap_{k=1}^{(m-1)/2} \Big( {2\over \tau}\N - 4k+2\Big)^c\Big]&\text{ if $m$ is odd}
\end{cases},
 \end{equation}
\begin{equation}\label{ali}
\text{and \ for $\ i\ge 2$}\qquad \al_i=\begin{cases}
 \al_1+2(i-1)+\dfrac{2(i-1)}{\tau}&\text{ if $i$ is odd}\\[0.3cm]
 \al_1\tau+2(i-1)\tau + 2(i-1)&\text{ if $i$ is even}
 \end{cases}.
 \end{equation}
Note that
\begin{equation}\label{ali01}
\al_i=\lf(\al_1+2i-2+{2i-2\over\tau}\rg) \tau^{\nu(i)},\qquad\text{for $i\ge 2$}
\end{equation}
 and $\al_i>0$ for all $i=1,\dots,m$. Furthermore, several identities and properties of $\al_i$'s, $\be_i$'s, $d_i$'s and $\e$ will be proven in order to have a good approximation. From the definition \eqref{ali}, it is readily checked that $\{\al_{2k+1}\}_k$ and $\{\al_{2k}\}_k$ are increasing in $k$. Since $\al_1>2$ and $\al_2=\al_1\tau + 2\tau +2>2$, it follows that $\al_i>2$ for all $i=1,\dots,m$. Notice that all these sets are countable : $2\N - {4k\over \tau}$ for $k=0,\dots,\frac{m-2}{2}$ and $ {2\over \tau}\N - 4k+2$ for $k=1,\dots, \frac m2$ with $m$ even; $ 2\N - {4k\over \tau} $ for $k=0,\dots, \frac{m-1}{2}$ and $  {2\over \tau}\N - 4k+2$ for $ k=1,\dots, \frac{m-1}{2}$ with $m$ odd. Hence, its union is also countable set and the complement of its union is dense in $\R$. Therefore, there exist $\al_1\in\R$ satisfying \eqref{ali0}.

\begin{lem}
If $\al_i$'s are given by \eqref{ali} then
\begin{equation}\label{ali1}
\al_{j+1} =  (\al_j+2)\tau^{(-1)^{j+1} } +2 =\begin{cases}
 (\al_j+2)\tau +2 &\text{ if $j$ is odd}\\[0.3cm]
\dfrac{ \al_j+2}{\tau}+ 2 &\text{ if $j$ is even}
 \end{cases}
\end{equation}
and
\begin{equation}\label{sai}
\sum_{i=1}^{j} {(-1)^{i+1}\over \tau^{\nu(i)}} \al_i=
\begin{cases}
-j-{j\over\tau}&\text{if $j$ is even}\\
\al_1+j-1+{j-1\over\tau}&\text{if $j$ is odd}
\end{cases}
\end{equation}
hold for all $j\ge 1$.  If $\al_1$ satisfies \eqref{ali0} then $\al_i\notin 2\N$  for all $i=1,\dots,m$.
\end{lem}

\begin{proof}[\dem]Assume first that $i$ is odd. Then, $i+1$ is even, from \eqref{ali} for $i$ odd, $ \al_i= \al_1+2i-2+\frac{2i-2}{\tau}$ and direct computations lead us to obtain that
$$(\al_i+2)\tau +2 =\lf(\al_1+2i+{2i-2\over \tau}\rg)\tau +2=(\al_1+2i)\tau + 2i.$$
On the other hand, from \eqref{ali} for $i+1$ even, we find that $\al_{i+1}=\al_1\tau +2i\tau +2i$, so that $\al_{i+1}=(\al_i+2)\tau +2$. Direct computations in case $i$ is even allow us to conclude \eqref{ali1}.

From the choice of $\al_i$'s and \eqref{ali01}, it follows that
\begin{equation*}
\begin{split}
\sum_{i=1}^{j} {(-1)^{i+1}\over \tau^{\nu(i)}} \al_i
&=\sum_{i=1}^{j}  (-1)^{i+1} \bigg[\al_{1} + 2i-2 + {2i - 2\over \tau }  \bigg] \\
&=\al_1\sum_{i=1}^{j} (-1)^{i+1} + 2\sum_{i=1}^{j} (-1)^{i+1} (i-1)  + {2 \over \tau} \sum_{i=1}^{j} (-1)^{i+1} (i-1)
\end{split}
\end{equation*}
and we conclude \eqref{sai}, in view of
$$\sum_{i=1}^{j} (-1)^{i+1} (i-1)=\sum_{i=1}^{j} (-1)^{i+1} i - \sum_{i=1}^{j} (-1)^{i+1} =\begin{cases}
-{j\over 2} +0 &\text{if $j$ is even}\\
{j+1\over 2} - 1&\text{if $j$ is odd}
\end{cases}.$$

Finally, assume that $m$ is even. If $\al_1$ satisfies \eqref{ali0} then we  have that $\al_1\in \big (2\N -{4k\over \tau}\big)^c$, $k=0,1,\dots,{m-2\over 2}$ and $\al_1\in\big({2\over \tau}\N-4k+2\big)^c$,  $k=1,2,\dots,{m\over 2}$. That is to say, $\al_1+{4k\over \tau}\notin 2\N$, $k=0,1\dots,{m-2\over 2}$ and $\tau(\al_1+4k-2)\notin 2\N$, $k=1,\dots,{m\over 2}$. Therefore, $\al_{2k+1}=\al_1+4k+{4k\over \tau}\notin 2\N$, $k=0,\dots,{m-2\over 2}$ and $\al_{2k}=\al_1\tau + (4k-2)\tau + 4k-2\notin 2\N$, $k=1,\dots,{m\over 2}$. Similar argument leads us to conclude that $\al_i\notin 2\N$ for all $i=1,\dots,m$ if $m$ is odd.
\end{proof}

In particular, we have that $\al_1>2$, $\al_2=(\al_1+2)\tau +2$, $\al_3={\al_2+2\over \tau} + 2$, $\al_4=(\al_3+2)\tau +2$, $\al_5={\al_4+2\over\tau} + 2,\ \dots$
Furthermore, we have that 
\begin{equation}\label{4psa}
4\pi \sum_{i=1}^{m} {(-1)^{i+1}\over \tau^{\nu(i)}} \al_i -2\pi(\al_1-2)=(-1)^{m+1}{2\pi\over \tau^{\nu(m)}}(\al_m+2).
\end{equation}

Now, we define $\be_i$ as follows, if $m$ is even then
\begin{equation}\label{beimp}
\be_l=\tau^{\nu(l)} \bigg[ m-l + {m-l+1\over\tau} \bigg]=\begin{cases}
(m - l)\tau +m-l+1 ,&\text{if $l$ is even}\\[0.3cm]
\ds m-l + {m-l+1\over \tau}&\text{if $l$ is odd}
\end{cases} ,
\end{equation}
 for $l=1,2,\dots,m$. 
If $m$ is odd then
 \begin{equation}\label{beimi}
\be_l=\tau^{\nu(l)} \lf[m-l+1 + {m-l\over\tau}\rg]
=\begin{cases}
(m - l+1)\tau +m-l ,&\text{if $l$ is even}\\[0.3cm]
\ds m-l +1 + {m-l\over \tau}&\text{if $l$ is odd}\end{cases}
\end{equation}
 for $l=1,2,\dots,m$. In any case, either $m$ is odd or even, it is readily checked that $\be_m=1$ and $\be_i>0$ for all $i=1,\dots,m$. Note that we can rewrite
\begin{equation}\label{beimip}
\be_l=\tau^{\nu(l)} \lf[{m-l+1\over \tau^{\nu(m)}} + {m-l\over\tau^{\nu(m+1)}}\rg].
\end{equation}
Hence, we shall prove the following useful properties from the definition \eqref{beimp} and \eqref{beimi}. Note that $1-\sigma(i) = \nu(i) =\sigma(i+1)$ and $\sigma(i)=1-\nu(i)=\nu(i+1)$.
 
 \begin{lem}
 If $\be_i$'s are given by either \eqref{beimp} or \eqref{beimi} then for any $l=2,\dots,m$
\begin{equation}\label{bell1}
\be_{l}=\begin{cases}
\tau \be_{l-1} - \tau - 1 , &\text{ if $l$ is even}\\[0.3cm]
 \ds{\be_{l-1}\over \tau} - 1 - {1\over \tau}&\text{ if $l$ is odd}
\end{cases}
\end{equation}
and for any $l=1,\dots,m-1$ 
\begin{equation}\label{belm1o2}
\frac{\be_l-1}{2\tau^{\nu(l)}}=(-1)^{\nu(l)} \sum_{j=l+1}^m {(-1)^{j}\over \tau^{\nu(j)}}\be_j.
\end{equation}
Furthermore, it holds that $\frac{\be_l}{\al_l} < \frac{\be_{l-1}}{\al_{l-1}}$ for any $l=2,\dots,m$, so that, ${\de_i\over\de_j}\to 0$ as  $\rho\to0$ for any $i<j.$
 \end{lem}
 
\begin{proof}[\dem] First, assume that $m$ is even. So, $\be_l$'s are given by \eqref{beimp}. If $l$ is even, then we have that $l-1$ is odd, so that, 
$$\be_{l-1}\ds =m- l+1 +{ m- l+2\over \tau}\qquad\text{ and}$$
$$\tau \beta_{l-1} - \tau -1= (m-l)\tau +\tau + m-l+2-\tau -1=\be_l.$$
Next, assume that $l$ is odd (still $m$ is even). Similarly as above, it follows that $ {\be_{l-1}\over \tau} - 1 - {1\over \tau}=\be_l$. Arguing in the same way for $\be_l$'s given by \eqref{beimi} when $m$ is odd, we conclude \eqref{bell1}.

Now, for any $m$, we shall prove that
\begin{equation}\label{bel}
\be_l=\begin{cases}
\ds 1 + 2\sum_{i=1}^{m-l} (-1)^{i+1}\tau^{\sigma(i)} \be_{l+i},&\text{if $l$ is even}\\[0.4cm]
\ds 1 + 2\sum_{i=1}^{m-l} {(-1)^{i+1}\over \tau^{\sigma(i)}} \be_{l+i}&\text{if $l$ is odd}
\end{cases}.
\end{equation}
Assume that $m$ is even. If $l$ is even then $\nu(i+l)=1-\sigma(i)$ and from \eqref{beimip} we get that
\begin{equation*}
\begin{split}
 \sum_{i=1}^{m-l} (-1)^{i+1}\tau^{\sigma(i)} \be_{l+i} 
&= \sum_{i=1}^{m-l} (-1)^{i+1}\tau \lf[ {\be_{l}\over\tau} - \Big(1+{1\over \tau}\Big) i \rg]=  \be_l \sum_{i=1}^{m-l} (-1)^{i+1} -(\tau + 1) \sum_{i=1}^{m-l} (-1)^{i+1} i
\end{split}
\end{equation*}
and \eqref{bel} follows, in view of $m-l$ is even, $\sum_{i=1}^{m-l} (-1)^{i+1} =0$ and $\sum_{i=1}^{m-l} (-1)^{i+1} i = - {m-l\over 2}.$
Next, assume that $l$ is odd (still $m$ is even) so that $\nu(i+l)= \sigma(i)$ and similarly as above
\begin{equation*}
\begin{split}
\sum_{i=1}^{m-l} { (-1)^{i+1}\over \tau^{\sigma(i)} }\be_{l+i} &= \sum_{i=1}^{m-l} (-1)^{i+1} \lf[ \be_{l} - \Big(1+{1\over \tau}\Big)i \rg]= \be_l \sum_{i=1}^{m-l} (-1)^{i+1} -\Big(1 + {1\over \tau} \Big) \sum_{i=1}^{m-l} (-1)^{i+1} i=\frac{\be_l-1}2,
\end{split}
\end{equation*}
in view of $m-l$ is odd, $\sum_{i=1}^{m-l} (-1)^{i+1} =1$ and $\sum_{i=1}^{m-l} (-1)^{i+1} i =  {m-l +1 \over 2}.$ Arguing in the same way for $\be_l$'s given by \eqref{beimi} when $m$ is odd we conclude \eqref{bel}. Now, from \eqref{bel} taking $j=l+i$ we have that if $l$ is even then
$$\be_l=1 + 2\sum_{i=1}^{m-l} (-1)^{i+1}\tau^{\sigma(i)} \be_{l+i}=1 + 2\sum_{j=l+1}^{m} (-1)^{j+1}{\tau\over \tau^{\nu(j)}} \be_{j}$$
in view of $\sigma(j-l)=\sigma(j)=1-\nu(j)$, and if $l$ is odd then
$$\be_l=1 + 2\sum_{i=1}^{m-l} \frac{(-1)^{i+1}}{\tau^{\sigma(i)} }\be_{l+i}=1 + 2\sum_{j=l+1}^{m} {(-1)^{j}\over \tau^{\nu(j)}} \be_{j},$$
in view of $\sigma(j-l)=\nu(j)$. Thus, we deduce \eqref{belm1o2}.

Now, we know that $0<\al_{l-1} + \al_{l-1}\tau + 2\be_{l-1} + 2\be_{l-1}\tau$ for any $l$. Hence, for $l$ even we have that $\be_l=\tau \be_{l-1} -\tau -1$ and $\al_l=(\al_{l-1}+2)\tau +2$ by using \eqref{bell1} and \eqref{ali1} respectively, so that
\begin{equation*}
\begin{split}
\al_{l-1}[\tau \be_{l-1} -\tau -1] &< [(\al_{l-1}+2)\tau +2] \be_{l-1}\iff
\al_{l-1}\be_l<\al_l\be_{l-1}
\end{split}
\end{equation*}
By using \eqref{ali1} and \eqref{bell1}, for $l$ odd we have that $\be_l=\frac{\be_{l-1}}{\tau} -\frac1\tau -1$ and $\al_l=\frac{\al_{l-1}+2}{\tau} +2$, and it is readily checked that $\al_{l-1}\be_l<\al_l\be_{l-1}$. Thus, we deduce that ${\de_i\over\de_j}\to 0$ as  $\rho\to0$ for any $i<j$, in view of
$${\de_i\over\de_j}={d_i^{1/\al_i}\rho^{\be_i/\al_i} \over d_j^{1/\al_j}\rho^{\be_j/\al_j}}={d_i^{1/\al_i} \over d_j^{1/\al_j} }\rho^{{\be_i\over \al_i}  - {\be_j\over\al_j}} \qquad \text{and}\qquad{\be_j\over \al_j} < {\be_i\over \al_i}.$$
This finishes the proof.
\end{proof}

\ms
Now, we define $d_i$'s by
\begin{equation}\label{di}
\log d_m=a_m \quad\text{and}\quad \log d_l=a_l + 2\tau^{\nu(l)} \sum_{i=l+1}^m {a_i\over \tau^{\nu(i)}}=\begin{cases}
\ds a_l + 2\sum_{i=l+1}^m a_i \tau^{\sigma(i)},&\text{if $l$ is even}\\[0.4cm]
\ds a_l + 2\sum_{i=l+1}^m {a_i\over \tau^{\nu(i)}}&\text{if $l$ is odd}
\end{cases}
\end{equation}
for $l=1,2,\dots,m-1$, where 
\begin{equation}\label{aimp}
a_l= \log\Big({\tau^{\nu(l)} V_{\nu(l)}(0)\over 2\al_l^2}\Big) +{ (-1)^{l}\over \tau^{\sigma (l)}} 2\pi (\al_m+2) h_0=\begin{cases}
\ds \log\Big({\tau V_1(0)\over 2\al_l^2}\Big) + 2\pi (\al_m+2) h_0,&\text{if $l$ is even}\\[0.3cm]
\ds \log\Big({ V_0(0)\over 2\al_l^2}\Big)-{2\pi\over \tau} (\al_m+2) h_0&\text{if $l$ is odd}
\end{cases}
\end{equation}
when $m$ even, while when $m$ is odd
\begin{equation}\label{aimi}
\begin{split}
a_l&= \log\Big({\tau^{\nu(l)} V_{\nu(l)}(0)\over 2\al_l^2}\Big) +(-1)^{l+1}\tau^{\nu(l)} 2\pi (\al_m+2) h_0\\
&=\begin{cases}
\ds \log\Big({\tau V_1(0)\over 2\al_l^2}\Big) - 2\pi (\al_m+2) \tau h_0,&\text{if $l$ is even}\\[0.3cm]
\ds \log\Big({ V_0(0)\over 2\al_l^2}\Big)+ 2\pi (\al_m+2) h_0&\text{if $l$ is odd}
\end{cases}.
\end{split}
\end{equation}

 \begin{lem}\label{ldi}
 If $d_i$'s are given by \eqref{di} then for any $l=1,\dots,m-1$ 
\begin{equation}\label{ldi0}
 \log d_l=\begin{cases}
\ds a_l + 2\sum_{i=1}^{m-l} (-1)^{i+1}\tau^{\sigma(i)} \log d_{l+i},&\text{if $l$ is even}\\[0.5cm]
\ds a_l + 2\sum_{i=1}^{m-l} {(-1)^{i+1}\over \tau^{\sigma(i)}} \log d_{l+i}&\text{if $l$ is odd}
\end{cases}
\end{equation}
and consequently,
\begin{equation}\label{delmalo2}
\frac{\log d_l-a_l}{2\tau^{\nu(l)}}=(-1)^{\nu(l)}\sum_{j=l+1}^m {(-1)^{j}\over \tau^{\nu(j)}}\log d_j.
\end{equation}
\end{lem}

\begin{proof}[\dem]
First, assume that $m$ is even. So, $a_l$'s are given by \eqref{aimp}. If $l$ is even, then we have that $\nu(i+l)=1-\sigma(i)$, $\frac\tau{\tau^{\nu(j)}}=\tau^{\sigma(j)}$ and $m-l$ is even, so that, $\sigma(m-l)=0$ and
\begin{equation*}
\begin{split}
 &\sum_{i=1}^{m-l} (-1)^{i+1}\tau^{\sigma(i)} \log d_{l+i}= \sum_{i=1}^{m-l-1} (-1)^{i+1}\tau^{\sigma(i)} \bigg[a_{l+i} + 2\tau^{\nu(l+i)} \sum_{j=l+i+1}^{m} {a_j\over \tau^{\nu(j)}}  \bigg] - a_m\\
&=\sum_{i=1}^{m-l-1} (-1)^{i+1}\tau^{\sigma(i)} a_{l+i} + 2\tau \sum_{i=1}^{m-l-1} \sum_{j=l+i+1}^{m} (-1)^{i+1} {a_j\over \tau^{\nu(j)}}  - a_m\\
&= \sum_{j=l+1}^{m-1} (-1)^{j+1}\tau^{\sigma(j)} a_{j} + \sum_{j=l+2}^{m} [1+(-1)^{j}] a_j \tau^{\sigma(j)}  - a_m= \sum_{j=l+1}^{m} a_j \tau^{\sigma(j)} 
\end{split}
\end{equation*}
and \eqref{ldi0} follows. Similarly as above, if $l$ is odd, then we have that $\nu(i+l)=\sigma(i)$, $\frac\tau{\tau^{\nu(j)}}=\tau^{\sigma(j)}$ and $m-l+1$ is even, so that, $\sigma(m-l+1)=0$ and
\begin{equation*}
\begin{split}
 & \sum_{i=1}^{m-l} {(-1)^{i+1}\over \tau^{\sigma(i)}} \log d_{l+i} 
=\sum_{i=1}^{m-l-1} {(-1)^{i+1} \over \tau^{\sigma(i)} }a_{l+i} + 2 \sum_{i=1}^{m-l-1} \sum_{j=l+i+1}^{m} (-1)^{i+1} {a_j\over \tau^{\nu(j)}} + { a_m\over \tau } \\
&=\sum_{j=l+1}^{m-1} {(-1)^{j}\over \tau^{\nu(j)}}  a_{j} + \sum_{j=l+2}^{m} [1+(-1)^{j+1}] {a_j \over \tau^{\nu (j)}}  + {a_m\over \tau} =\sum_{j=l+1}^{m} {a_j \over \tau^{\nu(j)} }
\end{split}
\end{equation*}
and we conclude \eqref{ldi0}. Similar arguments work out in case $m$ is odd. We deduce \eqref{delmalo2} by using the change $j=l+i$. This concludes the proof.
\end{proof}

Now, $\e=\e(\rho)$ is chosen so that
\begin{equation}\label{repla1}
\sum_{i=1}^m {(-1)^{i+1} \over \tau^{\nu(i)} } \gamma_i = 2\pi(\al_1-2), \qquad\text{where }\ \ \gamma_j=\gamma_{\de_j,\e}^{\al_j},\ \ j=1,\dots,m
\end{equation}
and $\gamma_{\de,\e}^{\al}$ is given in Lemma \ref{ewfxi}. Precisely, $\e=\e(\rho)$ is given by
\begin{equation}\label{eps}
\e^{\al_1-2}=\begin{cases}
\ds {[V_0(0)]^{m}[\tau V_1(0)]^{m\over\tau} \over \prod_{i=1}^m\al_i^{4/ \tau^{\nu(i)}}  } e^{{2\pi \over \tau} (\al_m+2) h_0}\Big({\rho\over 2}\Big)^{m+{m\over \tau} } &\text{if $m$ is even}\\[0.6cm]
\ds{[V_0(0)]^{m+1}[\tau V_1(0)]^{m-1\over\tau} \over \prod_{i=1}^m\al_i^{4/ \tau^{\nu(i)} } } e^{2\pi (\al_m+2) h_0}\Big({\rho\over 2}\Big)^{m+1+{m-1\over \tau} }&\text{if $m$ is odd}
\end{cases}
\end{equation}
It is readily checked that $\e(\rho)\to 0^+$ as $\rho\to0^+$. Moreover, $\e^{\al_1-2}\sim \rho^{\be_1+1}$, in view of
 \begin{equation}\label{be1}
 \be_1=\begin{cases}
m-1+{m\over \tau}&\text{if $m$ is even}\\
m+{m-1\over \tau}&\text{if $m$ is odd}
\end{cases}
\end{equation}

\begin{lem}
If $\e$ is given by \eqref{eps} then it holds \eqref{repla1} and $\frac{\e}{\de_1} \to 0$ as $\rho\to0$.
\end{lem}

\begin{proof}[\dem]
First, assume that $m$ is even so that, from \eqref{eps} it follows that
\begin{equation*}
\begin{split}
(\al_1-2)\log \e =m\log V_0(0) + {m\over \tau}\log[\tau V_1(0)] +{2\pi\over \tau}(\al_m+2) h_0-4\sum_{i=1}^m {\log\al_i\over \tau^{\nu(i)}} + \Big(m+{m\over \tau}\Big)\log\Big({\rho\over 2}\Big).
\end{split}
\end{equation*}
On the other hand, from the definition of $\de_i$ and $\gamma_i$ it follows that
$$\Big[-{1\over 2\pi}\log\e + h_0\Big] \gamma_i = -2\log d_i - 2\be_i \log \rho + 4\pi \al_i h_0$$
so that,
$$\Big[-{1\over 2\pi}\log\e + h_0\Big] \sum_{i=1}^m {(-1)^{i+1} \over \tau^{\nu(i)} }\gamma_i = -2\sum_{i=1}^m {(-1)^{i+1} \over \tau^{\nu(i)} }\log d_i - 2\log \rho \sum_{i=1}^m {(-1)^{i+1} \over \tau^{\nu(i)} }\be_i  + 4\pi h_0\sum_{i=1}^m {(-1)^{i+1} \over \tau^{\nu(i)} }\al_i.$$
Hence, we compute the sums involved as follows. From Lemma \ref{ldi} we have that for any $m$, either odd or even,
$$\log d_1 = a_1 + 2\sum_{j=2}^m {(-1)^{j}\over \tau^{\nu(j)}} \log d_j.$$
Then by using \eqref{di} and $m=2k$ for some $k\in\N$ we get that
\begin{equation*}
\begin{split}
&-\sum_{j=1}^m  {(-1)^{j} \over \tau^{\nu(j)} }\log d_j =  {a_1+\log d_1\over 2} = \sum_{i=1}^m{a_i\over \tau^{\nu(i)}}=\sum_{i=1}^k{1\over \tau}\lf[\log[\tau V_1(0)] + 2\pi(\al_m+2)h_0 - \log(2\al_{2i}^2)\rg] \\
& \quad+\sum_{i=1}^k \lf[\log V_0(0) - {2\pi\over \tau} (\al_m+2)h_0 - \log(2\al_{2i-1}^2)\rg]\\
&={m\over 2}\log V_0(0) + {m\over 2\tau}\log[\tau V_1(0)] -\Big(m+{m\over \tau}\Big) {\log 2\over 2} - 2\sum_{i=1}^m{\log\al_i\over \tau^{\nu(i)}}.
\end{split}
\end{equation*}
Also, from the choice of $\be_i$'s in \eqref{beimp}, it is readily checked that
\begin{equation*}
\begin{split}
\sum_{i=1}^m {(-1)^{i+1} \over \tau^{\nu(i)} } \be_i&= \sum_{i=1}^m (-1)^{i+1}\lf[m+{m+1\over\tau} -\Big(1+{1\over \tau}\Big) i\rg]={m\over 2} + {m\over 2\tau}.
\end{split}
\end{equation*}
Therefore, by using \eqref{sai} we obtain that
\begin{equation*}
\begin{split}
&\Big[-{1\over 2\pi}\log\e + h_0\Big] \sum_{i=1}^m {(-1)^{i+1} \over \tau^{\nu(i)} }\gamma_i = -m\log V_0(0) - {m\over \tau} \log [\tau V_1(0)] + \Big(m+{m\over \tau}\Big)\log 2 + 4\sum_{i=1}^m{\log\al_i\over \tau^{\nu(i)}}\\
&\quad - \Big(m+{m\over \tau}\Big)\log\rho + 4\pi h_0\Big(- m- {m\over \tau}\Big) =-(\al_1-2)\log\e +{2\pi\over \tau}(\al_m+2)h_0 -4\pi h_0\Big( m + {m\over \tau}\Big)\\
&=-(\al_1-2)\log\e + 2\pi(\al_1-2)  h_0,
\end{split}
\end{equation*}
in view of ${\al_m+2\over \tau} -2  m - {2m\over \tau}=\al_1-2$ and \eqref{repla1} follows. Similarly, \eqref{repla1} is proven if $m$ is odd. 

\ms
Finally, taking into account \eqref{be1}, \eqref{eps} and $\de_1^{\al_1}=d_1\rho^{\be_1}$, we deduce that ${\e\over \de_1}\to 0\iff {\be_1+1\over \al_1-2}>{\be_1\over \al_1},$ in view of $\e^{\al_1-2}\sim \rho^{\be_1+1}$. Indeed, it is readily checked that $ {\be_1+1\over \al_1-2}>{\be_1+1\over \al_1}>{\be_1\over \al_1}.$ This completes the proof.
\end{proof}

Let us stress that the behavior of $\e=\e(\rho) $ and $\de_j=\de_j(\rho)$'s with respect to $\rho$ is given by
$${\e\over \de_j},\ \ {\de_i\over\de_j}\to 0\quad\text{as }\ \ \rho\to0\ \ \text{for }\ i<j\ \ \text{and } \ \ \e\to0\ \ \text{as }\ \ \rho\to0 .$$
Assume that $\de_i$'s are given by \eqref{dei} with $\al_i$'s, $\be_i$'s and $d_i$'s defined in \eqref{ali}, \eqref{beimp}-\eqref{beimi} and \eqref{di}, respectively, and $\e$ is given by \eqref{eps}. Notice that $\frac{\log\de_j}{\log\e}=O(1)$ as $\rho\to0$, for all $j=1,\dots,m$.
Now, define the \emph{shrinking annuli}
$$A_i=\{x\in\Om\mid \sqrt{\de_{i-1}\de_i}<|x|\le \sqrt{\de_i\de_{i+1}}\},\qquad j=1,\dots,m,$$
where for simplicity we denote $\de_0=\frac{\e^2}{\de_1}$ and $\de_{m+1}=\frac{M_0^2}{\de_m}$ with $M_0=\sup\{|x| :  x\in\Om\}$, so that, 
$\Om_\e=\cup_{j=1}^mA_j$, $\cap_{j=1}^m A_j=\emptyset$ and $\frac{A_j}{\de_j}$ approaches to $\R^2$ as $\rho\to0$ for each $j=1,\dots,m$. 

\begin{lem}\label{expaU}
There exists $\eta>0$ such that the following expansions hold
\begin{equation}\label{waje}
\begin{split}
U(\de_j y)=&\, {(-1)^{j+1}\over\tau^{\nu(j)}} \lf[- 2\log\de_j-a_j-\log\rho +\log{|y|^{\al_j-2}\over(1+|y|^{\al_j})^2} \rg] \\
&\, +(-1)^{m+1} {2\pi\over \tau^{\nu(m)}} (\al_m+2)H(\de_j y,0)  +O\lf(\rho^\eta\rg)
\end{split}
\end{equation}
 for any $j=1,\dots,m$, uniformly for $\de_j y \in A_j$, $j=1,\dots,m$.
\end{lem}

\begin{proof}[\dem]
From the expansions of $P_\e w_j$, $j=1,\dots,m$ and the definition of $\gamma_j$ in Lemma \ref{ewfxi}, \eqref{4psa} and \eqref{repla1} we obtain that
\begin{equation}\label{expUx}
\begin{split}
U(x)
=&\, \sum_{i=1}^m {(-1)^{i+1}\over \tau^{\nu(i)}} \lf[w_i-\log(2\al_i^2\de_i^{\al_i})\rg] +  { 1\over 2\pi} \sum_{i=1}^m {(-1)^{i+1}\over \tau^{\nu(i)}}  \gamma_i \log |x|  \\
&\,+\sum_{i=1}^m {(-1)^{i+1}\over \tau^{\nu(i)}} \lf[4\pi\al_i -   \gamma_i \rg]H(x,0) + O\lf(\sum_{j=1}^m\lf[\de_j^{\al_j}+\Big({\e\over \de_j}\Big)^{\al_j}\rg]+\e\rg)\\
=&\,\sum_{i=1}^m {(-1)^{i+1}\over \tau^{\nu(i)}} \log{1\over ( \de_i^{\al_i} +|x|^{\al_i} )^2} + (\al_1-2)\log |x| + (-1)^{m+1} {2\pi\over \tau^{\nu(m)} } (\al_m+2)H(x,0) \\
&\,+ O\lf(\rho^\eta\rg),
\end{split}
\end{equation}
where we choose
\begin{equation}\label{eta1}
0<\eta\le \min\bigg\{1, {\be_1+1\over \al_1-2},\min\Big\{\al_j\Big({\be_1+1\over \al_1-2} - {\be_j\over \al_j}\Big)\ : \ j=1,\dots,m\Big\} \bigg\},
\end{equation}
in view $\be_j$ are decreasing so that $\de_j^{\al_j}=O(\rho)$ for all $j=1,\dots,m$, $\big({\e\over \de_j}\big)^{\al_j} =O\big(\rho^{\al_j({\be_1+1\over \al_1-2} - {\be_j\over \al_j})}\big)$ and $\e=O\big(\rho^{\be_1+1\over \al_1-2}\big)$. Now, from the choice of $\de_i$'s in \eqref{dei}, \eqref{belm1o2} and \eqref{delmalo2} we have that for $j$ odd
\begin{equation*}
\begin{split}
2\sum_{i=j+1}^m{(-1)^{i}\over \tau^{\nu(i)}} \al_i\log\de_i & = 2\sum_{i=j+1}^m{(-1)^{i}\over \tau^{\nu(i)}} \log d_i +2\log\rho \sum_{i=j+1}^m{(-1)^{i}\over \tau^{ \nu(i)}} \be_i\\
& = 2\frac{\log d_j-a_j}{2} + 2\frac{\be_j -1}2 \log\rho = \al_j\log\de_j - a_j -\log\rho
\end{split}
\end{equation*}
and similarly, for $j$ even $ 2\sum_{i=j+1}^m{(-1)^{i}\over \tau^{\nu(i)}} \al_i\log\de_i = -{1\over \tau}\big[ \al_j\log\de_j - a_j -\log\rho\big].$ Thus, we rewrite as
\begin{equation}\label{2saildi}
2\sum_{i=j+1}^m{(-1)^{i}\over \tau^{\nu(i)}} \al_i\log\de_i = {(-1)^{j+1}\over \tau^{\nu(j)} }\big[ \al_j\log\de_j - a_j -\log\rho\big].
\end{equation}

On the other hand, for any $y\in\frac{A_1}{\de_1}$ it holds that
\begin{equation}\label{wa1}
\begin{split}
 U(\de_1 y) 
=&\,-2\al_1\log\de_1+\log{1\over(1+|y|^{\al_1})^2}+(\al_1-2)\log (\de_1|y|) \\
&\,+ \sum_{i=2}^m{(-1)^{i+1}\over \tau^{\nu(i)}} \log\frac1{(\de_i^{\al_i} + \de_1^{\al_i}|y|^{\al_i})^2}+(-1)^{m+1}{2\pi\over \tau^{\nu(m)}} (\al_m+2)H(\de_1 y,0)+ O\lf(\rho^\eta\rg)\\
 = &\,-(\al_1+2)\log\de_1+\sum_{i=2}^m(-1)^{i}{2\al_i\over\tau^{\nu(i)}} \log\de_i +\log{|y|^{\al_1-2}\over(1+|y|^{\al_1})^2} \\
&\,  +\sum_{i=2}^m {(-1)^{i}\over \tau^{\nu(i)}}2\log\lf(1+\Big({\de_1|y|\over \de_i}\Big)^{\al_i}\rg) + (-1)^{m+1}{2\pi\over \tau^{\nu(m)}} (\al_m+2) H(\de_1 y,0) + O\lf(\rho^\eta\rg) \\
=&\, - 2\log\de_1-a_1-\log\rho +\log{|y|^{\al_1-2}\over(1+|y|^{\al_1})^2} + (-1)^{m+1}{2\pi\over \tau^{\nu(m)}} (\al_m+2) H(\de_1 y,0)  \\
&\, + O\lf( \sum_{j=2}^m\Big({\de_1\over \de_j}\Big)^{\al_j\over 2}+\rho^\eta\rg),
\end{split}
\end{equation}
in view of \eqref{2saildi} and
$$\log\lf(1+\Big({\de_1|y|\over \de_i}\Big)^{\al_i}\rg)=O\lf(\Big({\de_1|y|\over \de_i}\Big)^{\al_i}\rg)=O\lf(\Big({\de_1\over \de_i}\Big)^{\al_i\over 2}\rg),\quad i=2,\dots,m.$$
Choosing $\eta>0$ satisfying \eqref{eta1} and
\begin{equation}\label{eta2}
0<\eta\le {1\over 2}\min\Big\{\al_i\Big({\be_1\over \al_1} - {\be_i\over \al_i}\Big)\ : \ 1<i\Big\}
\end{equation}
we get that $\big( {\de_1\over \de_i}\big)^{\al_i\over 2}=O\lf(\rho^\eta\rg)$, $i=2,\dots,m.$
Similarly, for $y\in \frac{A_j}{\de_j}$, $1<j<m$ we find that

\begin{equation}\label{wa2}
\begin{split}
&U(\de_j y)=\sum_{i=1}^{j-1}\frac{(-1)^{i+1}}{\tau^{\nu(i)} } \bigg[-2\al_i\log(\de_j|y|)-2\log\lf(1+\Big({\de_i\over \de_j|y|}\Big)^{\al_i}\rg)\bigg] + (\al_1-2)\log(\de_j |y|)\\
&\,+{(-1)^{j+1}\over \tau^{\nu(j)} }\bigg[-2\al_j\log\de_j+\log{1\over(1+|y|^{\al_j})^2}\bigg]+\sum_{i=j+1}^m \frac{(-1)^{i}}{\tau^{\nu(i)}} 2\al_i\log\de_i\\
&\,+\sum_{i=j+1}^m \frac{(-1)^{i}}{\tau^{\nu(i)}} 2\log\bigg(1+\Big({\de_j|y|\over \de_i}\Big)^{\al_i}\bigg) +(-1)^{m+1}{2\pi \over \tau^{\nu(m)} }(\al_m+2)H(\de_j y,0) + O\lf(\rho^\eta\rg)\\
&= -2\log\de_j \sum_{i=1}^{j} \frac{(-1)^{i+1} }{\tau^{\nu(i)} }\al_i + 2\sum_{i=j+1}^m {(-1)^{i}\over \tau^{\nu(i) }}\al_i \log\de_i + (\al_1-2) \log\de_j + (\al_1-2)\log|y|  \\
&\, - 2 \log|y| \sum_{i=1}^{j-1} \frac{(-1)^{i+1}}{\tau^{\nu(i)} } \al_i 
+ {(-1)^{j+1}\over\tau^{\nu(j)} } \log{1\over (1+|y|^{\al_j})^2} \\
&\, +(-1)^{m+1} {2\pi\over \tau^{\nu(m)}} (\al_m+2)H(\de_j y,0) +O\bigg( \sum_{i=1}^{j-1}\Big({\de_i\over \de_j}\Big)^{\al_i\over 2} +\sum_{i=j+1}^m\Big({\de_j\over \de_i}\Big)^{\al_i\over 2} +\rho^\eta \bigg),
\end{split}
\end{equation}
in view of
$$\log\lf(1+\Big({\de_i\over \de_j|y|}\Big)^{\al_i}\rg)=O\lf(\Big({\de_i\over \de_j|y|}\Big)^{\al_i}\rg)=O\lf(\Big({\de_i\over \de_j}\Big)^{\al_i\over 2}\rg), \quad i<j,$$
$$\log\lf(1+\Big({\de_j|y|\over \de_i}\Big)^{\al_i}\rg)=O\lf(\Big({\de_j|y|\over \de_i}\Big)^{\al_i}\rg)=O\lf(\Big({\de_j\over \de_i}\Big)^{\al_i\over 2}\rg),\quad j<i,$$
 and by using again \eqref{repla1}. Now, we choose $\eta$ satisfying \eqref{eta1}, \eqref{eta2} and smaller, if necessary, so that
$$0<\eta\le {1\over 2}\min\Big\{\al_i\Big({\be_i\over \al_i} - {\be_j\over \al_j}\Big)\ : \ i<j\Big\}, \ \ j=2,\dots,m$$
and
$$0<\eta\le {1\over 2}\min\Big\{\al_i\Big({\be_j\over \al_j} - {\be_i\over \al_i}\Big)\ : \ j<i \Big\},\ \ j=1,\dots,m-1$$
thus, $ \big({\de_i\over \de_j}\big)^{\al_i\over 2} = O\lf(\rho^\eta\rg)$, $i<j$ and $ \big({\de_j\over \de_i}\big)^{\al_i\over 2} = O\lf(\rho^\eta\rg)$, $j<i$. Hence, by using \eqref{sai} and \eqref{2saildi} for $j$ even we get that
\begin{equation*}
\begin{split}
&U(\de_j y)= -2\log\de_j\Big[-j-{j\over \tau}\Big] - {1\over \tau} \lf[\al_j\log\de_j - a_j -\log\rho\rg] +(\al_1-2)\log\de_j+ (\al_1-2)\log |y|\\
&\, -2\log |y|\lf[\al_1 + j-2 + \frac{j-2}\tau\rg]  - {1\over \tau} \log{1\over (1+ |y|^{\al_j})^2}  +(-1)^{m+1} {2\pi\over \tau^{\nu(m)}} (\al_m+2)H(\de_j y,0) +O\lf(\rho^\eta\rg)\\ 
\end{split}
\end{equation*}
and we conclude \eqref{waje}. Similarly, for $j$ odd we get that
\begin{equation*}
\begin{split}
&U(\de_j y)= -2\log\de_j\Big[\al_1 + j-1 + {j-1\over \tau}\Big] + \al_j\log\de_j - a_j -\log\rho +(\al_1-2)\log\de_j + (\al_1-2)\log |y|\\
&\, -2\log |y|\lf[- j +1- \frac{j-1}\tau\rg] + \log{1\over (1+ |y|^{\al_j})^2}  +(-1)^{m+1} {2\pi\over \tau^{\nu(m)}} (\al_m+2)H(\de_j y,0) +O(\rho^\eta)\\ 
\end{split}
\end{equation*}
and we obtain \eqref{waje}. Also, we have for any $y\in {A_m\over \de_m}$

\begin{equation}\label{wam}
\begin{split}
&U(\de_m y)= \sum_{i=1}^{m-1} \frac{(-1)^{i+1} }{\tau^{\nu(i)}} \lf[ -2\al_i\log(\de_m|y|) - 2\log\bigg(1+\Big( {\de_i\over \de_m|y|}\Big)^{\al_i} \bigg)\rg] +(\al_1-2)\log(\de_m |y|)\\
&\,+ {(-1)^{m+1}\over \tau^{\nu(m)}} \lf[ -2\al_m\log\de_m + \log{1\over (1+|y|^{\al_m})^2}\rg]   +(-1)^{m+1} {2\pi\over \tau^{\nu(m)}} (\al_m+2)H(\de_m y,0) + O(\rho^\eta)\\ 
&= {(-1)^{m+1}\over \tau^{\nu(m)}} \lf[ -2\log\de_m -a_m - \log\rho + \log{|y|^{\al_m-2}\over(1+|y|^{\al_m})^2} + 2\pi (\al_m+2)H(\de_m y,0)\rg]+ O(\rho^\eta). \end{split}
\end{equation} 
This completes the proof.
\end{proof}

\section{Problem in Liouville form \eqref{spsh}}
\subsection{Error estimate}
\noindent In order to evaluate how well the approximation $U$ satisfies the equation \eqref{spsh} (and get the invertibility of the linearized operator), we will use the norms
$$\| h \|_p=\lf(\int_{\Om_\e} |h(x)|^p\, dx\rg)^{1/ p}\qquad\text{and}\qquad\| h \|=\lf(\int_{\Om_\e} |\grad h(x)|^2\, dx\rg)^{1/2},$$
the usual norms in the Banach spaces $L^p(\Om_\e)$ and $H_0^1(\Om_\e)$, respectively. Assume that $\de_i$'s are given by \eqref{dei} with $\al_i$'s, $\be_i$'s and $d_i$'s defined in \eqref{ali}, \eqref{beimp}-\eqref{beimi} and \eqref{di}, respectively, and $\e$ is given by \eqref{eps}. Let us evaluate the approximation rate of $U$ in $\|\cdot\|_p$, encoded in \eqref{R}

\begin{lem}\label{estrr0}
There exist $\rho_0>0$, a constant $C>0$ and $1<p_0<2$, \st for any $\rho\in(0,\rho_0)$ and $p\in(1,p_0)$ it holds
\begin{equation}\label{re}
\|R\|_p\le  C\rho^{\eta_p}
\end{equation}
for some $\eta_p>0$.
\end{lem}

\begin{proof}[\dem]
First, note that  $\Delta U=\sum_{i=1}^m {(-1)^{i}\over \tau^{\nu(i)}} |x|^{\al_i-2} e^{w_i}$
so that, for any $1< j< m$ and $y\in \frac{A_j}{\de_j}$ we have that
\begin{equation}\label{ludjy}
\Delta U(\de_jy)= {(-1)^{j}\over \tau^{\nu(j)}} {2\al_j^2|y|^{\al_j-2}\over \de_j^2(1+|y|^{\al_j})^2} +O\bigg(\sum_{i=1}^{j-1} \Big({\de_i\over \de_j}\Big)^{\al_i} {1\over \de_j^2|y|^{\al_i+2}} + \sum_{i=j+1}^m \Big({\de_j\over \de_i}\Big)^{\al_i} {|y|^{\al_i-2}\over \de_j^2}\bigg).
\end{equation}
Similarly, we obtain that
\begin{equation}\label{lud1y}
\Delta U(\de_1y)= - {2\al_j^2|y|^{\al_j-2}\over \de_j^2(1+|y|^{\al_j})^2} +O\bigg( \sum_{i=2}^m \Big({\de_j\over \de_i}\Big)^{\al_i} {|y|^{\al_i-2}\over \de_j^2}\bigg),\qquad y\in{A_1\over \de_1}
\end{equation}
and
\begin{equation}\label{ludmy}
\Delta U(\de_my)= {(-1)^{m}\over \tau^{\nu(m)}} {2\al_j^2|y|^{\al_j-2}\over \de_j^2(1+|y|^{\al_j})^2} +O\bigg(\sum_{i=1}^{m-1} \Big({\de_i\over \de_j}\Big)^{\al_i} {1\over \de_j^2|y|^{\al_i+2}}  \bigg),\qquad y\in{A_m\over \de_m}
\end{equation}
By using \eqref{aimp} and \eqref{waje} (from Lemma \ref{expaU}) for $y\in \frac{A_j}{\de_j}$ and any $j$ odd we have that
\begin{equation}\label{veujo}
\begin{split}
\rho V_0(\de_j y) e^{U(\de_j y)} & = {2\al_j^2|y|^{\al_j-2}\over \de_j^2(1+|y|^{\al_j})^2} {V_0(\de_j y)\over V_0(0)} \exp\bigg((-1)^{m+1} {2\pi\over \tau^{\nu(m)}} (\al_m+2)[H(\de_j y,0) -h_0] + O(\rho^\eta) \bigg)\\
&= {2\al_j^2|y|^{\al_j-2}\over \de_j^2(1+|y|^{\al_j})^2} \lf[1+O(\de_j|y| + \rho^\eta)\rg]
\end{split}
\end{equation}
and
\begin{equation}\label{veujo2}
\begin{split}
&\rho V_1(\de_j y) e^{-\tau U(\de_j y)} = \rho V_1(\de_j y) \exp\bigg( -\tau\Big[-2\log\de_j -a_j-\log\rho + \log {|y|^{\al_j-2}\over \de_j^2(1+|y|^{\al_j})^2}\\
&\, +(-1)^{m+1} {2\pi\over \tau^{\nu(m)}} (\al_m+2) H(\de_j y,0) \Big] + O(\rho^\eta) \bigg)
= O\bigg(\rho^{1+\tau} \de_j^{2\tau} {(1+|y|^{\al_j})^{2\tau} \over |y|^{(\al_j-2)\tau }} \bigg),
\end{split}
\end{equation}
Similarly, by using \eqref{aimi} and $1-\nu(m)=\sigma(m)$ for $y\in \frac{A_j}{\de_j}$ and $j$ even
\begin{equation}\label{vetuje}
\begin{split}
\hspace{-0.2cm}\rho V_1(\de_j y) e^{-\tau U(\de_j y)} & = {2\al_j^2|y|^{\al_j-2}\over \de_j^2\tau (1+|y|^{\al_j})^2} {V_1(\de_j y)\over V_1(0)} \exp\bigg((-1)^{m} 2\pi \tau^{\sigma(m)} (\al_m+2)[H(\de_j y,0) - h_0] + O(\rho^\eta) \bigg)\\
&= {2\al_j^2|y|^{\al_j-2}\over \de_j^2\tau (1+|y|^{\al_j})^2} \lf[1+O(\de_j|y| + \rho^\eta)\rg]
\end{split}
\end{equation}
and
\begin{equation}\label{vetuje2}
\rho V_0(\de_j y) e^{U(\de_j y)} = O\bigg(\rho^{1+1/\tau} \de_j^{2/ \tau} {(1+|y|^{\al_j})^{2/\tau}\over |y|^{(\al_j-2)/\tau}} \bigg).
\end{equation}
Hence, by using \eqref{lud1y}, \eqref{veujo} and \eqref{veujo2} for $\de_1 y\in A_1$ we get that
\begin{equation}\label{estRd1}
\begin{split}
R(\de_1 y) 
=&\,  {1\over\de_1^2} {2\al_1^2|y|^{\al_1-2}\over (1+|y|^{\al_1})^2} O(\de_1|y| + \rho^\eta)  + O\bigg( \rho^{1+\tau} \de_1^{2\tau} {(1+|y|^{\al_1})^{2\tau} \over |y|^{(\al_1-2)\tau }}+ \sum_{i=2}^m \Big({\de_1\over \de_i}\Big)^{\al_i} {|y|^{\al_i-2}\over \de_1^2}\bigg),
\end{split}
\end{equation}
by using \eqref{ludjy},  \eqref{veujo} and \eqref{veujo2} for $1<j<m$, $j$ odd and $\de_jy\in A_j$
\begin{equation}
\begin{split}
R(\de_j y)=&\,  {1\over\de_j^2} {2\al_j^2|y|^{\al_j-2}\over (1+|y|^{\al_j})^2} O(\de_j|y| + \rho^\eta) \\
&\, + O\bigg( \rho^{1+\tau} \de_j^{2\tau} {(1+|y|^{\al_j})^{2\tau} \over |y|^{(\al_j-2)\tau }} + \sum_{i=1}^{j-1} \Big({\de_i\over \de_j}\Big)^{\al_i} {1\over \de_j^2|y|^{\al_i+2}}+ \sum_{i=j+1}^m \Big({\de_j\over \de_i}\Big)^{\al_i} {|y|^{\al_i-2}\over \de_j^2}\bigg),
\end{split}
\end{equation}
by using \eqref{ludjy}, \eqref{vetuje} and \eqref{vetuje2} for $1<j<m$, $j$ even and $\de_jy\in A_j$
\begin{equation}
\begin{split}
R(\de_j y)=&\,  {1\over\de_j^2\tau} {2\al_j^2|y|^{\al_j-2}\over (1+|y|^{\al_j})^2} O(\de_j|y| + \rho^\eta) \\
&\, + O\bigg( \rho^{1+1/\tau} \de_j^{2/\tau} {(1+|y|^{\al_j})^{2/\tau} \over |y|^{(\al_j-2)/\tau }} + \sum_{i=1}^{j-1} \Big({\de_i\over \de_j}\Big)^{\al_i} {1\over \de_j^2|y|^{\al_i+2}}+ \sum_{i=j+1}^m \Big({\de_j\over \de_i}\Big)^{\al_i} {|y|^{\al_i-2}\over \de_j^2}\bigg)
\end{split}
\end{equation}
and by \eqref{ludmy} for $\de_m y\in A_m$
\begin{equation}\label{estRdm}
\begin{split}
R(\de_m y)=&\,  {1\over\de_m^2\tau^{\nu(m)} } {2\al_m^2|y|^{\al_m-2}\over (1+|y|^{\al_m})^2} O(\de_m|y| + \rho^\eta) \\
&\, + O\bigg( \rho^{1+\tau^{(-1)^{m+1} } } \de_m^{2\tau^{(-1)^{m+1} } } {(1+|y|^{\al_m})^{2\tau^{(-1)^{m+1} } } \over |y|^{(\al_m-2)\tau^{(-1)^{m+1} } }} + \sum_{i=1}^{m-1} \Big({\de_i\over \de_m}\Big)^{\al_i} {1\over \de_m^2|y|^{\al_i+2}} \bigg)
\end{split}
\end{equation}
By \eqref{estRd1}-\eqref{estRdm}, we finally get that there exist $\rho_0>0$ small enough and $p_0>1$ close to 1, so that for all $0<\rho\le \rho_0$ and $1<p\le p_0$
\begin{equation*}
\begin{split}
\int_{\Om_\e}|R(x)|^p\,dx=&\,\sum_{j=1}^m \int_{A_j}|R(x)|^p\,dx = \sum_{j=1}^m \de_j^2\int_{A_j}|R(\de_j y )|^p\,dy = O\lf(\rho^{p \eta_p}\rg),
\end{split}
\end{equation*}
for some $ \eta_p>0 $, see Appendix, Lemma \ref{estint} for more details. This completes the proof.
\end{proof}


\subsection{The nonlinear problem and proof of Theorem \ref{main}.}\label{sec4}

\noindent In this subsection, we will look for a solution $u$ of \eqref{spsh} in the form
$u=U+\phi$, for some small remainder term $\phi$. In terms of
$\phi$, the problem \eqref{spsh} is equivalent to find $\phi\in
H_0^1(\Om_\e)$ so that
\begin{equation}\label{ephi}
\left\{ \begin{array}{ll} 
L(\phi)=-[R+\Lambda(\phi)+N(\phi)], &\text{ in $\Om_\e$},
\\
\phi=0,& \text{ on $\fr\Om_\e$}.
\end{array} \right.
\end{equation}
where the linear operators $L$ and $\Lambda$ are defined as
\begin{equation}\label{ol}
L(\phi) = \Delta \phi +  K(x) \phi,\quad K(x)=\sum_{i=1}^m |x|^{\al_i-2}e^{w_i}
\end{equation}
and
\begin{equation}\label{La}
\Lambda(\phi)=\rho [V_0(x)e^{U} + \tau V_1(x)e^{-\tau U}] \phi - K\phi
\end{equation}
The nonlinear part $N$ is given by
\begin{equation}\label{nlt}
\begin{split}
N(\phi)=&\,\rho V_0(x)e^{U}\big(e^{ \phi}- \phi - 1\big)-\rho  V_1(x) e^{-\tau U}\big( e^{ -\tau \phi}+ \tau  \phi  -1 \big).
\end{split}
\end{equation}

\begin{prop}\label{p2}
For any $p>1$ there exists $\rho_0>0$ so that for all $0<\rho\leq \rho_0$,
$h\in L^p(\Om_\e)$ there is a unique solution $\phi \in H_0^1(\Om_\e)$ of
\begin{equation}\label{plco}
\left\{ \begin{array}{ll}
L(\phi) = h&\text{in }\Om_\e\\
\phi =0 &\text{on }\fr\Om_\e.
\end{array} \right.
\end{equation}
Moreover, there is a constant $C>0$ independent of $\rho$ such that 
\begin{equation}\label{est}
\|\phi \| \le C |\log \rho| \|h\|_p.
\end{equation}
\end{prop}

The latter proposition implies that the unique solution $\phi=T(h)$ of \eqref{plco} defines a continuous linear map from $L^p(\Om_\e)$ into $H_0^1(\Om_\e)$, with norm bounded by $C|\log\rho|$. We are now in position to study the nonlinear problem. 

\begin{prop}\label{p3}
There exist $p_0>1$ and $\rho_0>0$ so that for any $1<p<p_0$ and 
all $0<\rho\leq \rho_0$, the problem
\begin{equation}\label{pnlabis}
\left\{ \begin{array}{ll}
L(\phi)= -[R + \Lambda(\phi)+ N(\phi)], & \text{in } \Om_\e\\
\phi=0, &\text{on }\fr\Om_\e
\end{array} \right.
\end{equation}
admits a unique solution $\phi(\rho) \in H_0^{1}(\Om_\e)$, where $N$, $\Lambda$ and $R$ are given by \eqref{nlt}, \eqref{La} and \eqref{R}, respectively. Moreover, there is a constant $C>0$ such that for some $\eta_p>0$
$$\|\phi\|_\infty\le C\rho^{\eta_p} |\log \rho |$$
\end{prop}

Here, $\eta_p$ is the same as in \eqref{re}. We shall use the following estimates.

\begin{lem}
There exist $p_0>1$ and $\rho_0>0$ so that for any $1<p<p_0$ and 
all $0<\rho\leq \rho_0$ it holds
\begin{equation}\label{estlaphi}
\| \Lambda (\phi) \|_p  \le C\rho^{\eta_p'} \|\phi\|,
\end{equation}
for all $\phi\in H_0^1(\Om_\e)$ with $\|\phi\|\le M \rho^{\eta_p}|\log\rho|$, for some $\eta_p'>0$.
\end{lem}

\begin{proof}[\dem]
Arguing in the same way as in \cite[Lemma 3.3]{EFP}, for simplicity, denote $W=  \rho[V_0 e^{ U} + \tau V_1 e^{-\tau U }] $, so that the linear operator $\Lambda$ is re-written as $\Lambda(\phi)= (W-K)\phi$. By using \eqref{veujo}-\eqref{veujo2} and \eqref{vetuje}-\eqref{vetuje2}, we find that for $i$ odd
$$\de_i^2W(\de_i y)={2\al_i^2|y|^{\al_i-2}\over (1+|y|^\al_i)^2} \lf[1+O(|\de_i y|+\rho^{\eta})\rg] + O\Big(\rho^{1+\tau} \de_i^{2+2\tau}{(1+|y|^{\al_i})^{2\tau}\over |y|^{(\al_i-2)\tau}}\Big)$$
uniformly for $\de_iy\in A_i$ and for $i$ even
$$\de_i^2W(\de_i y)={2\al_i^2|y|^{\al_i-2}\over (1+|y|^\al_i)^2} \lf[1+O(|\de_i y|+\rho^{\eta})\rg] + O\Big(\rho^{1+1/\tau} \de_i^{2+2/\tau}{(1+|y|^{\al_i})^{2/\tau}\over |y|^{(\al_i-2)/\tau}}\Big)$$
uniformly for $x\in A_i$ and $i$ even. Hence, for any $q\ge 1$ and for $i$ odd there holds
\begin{equation*}
\begin{split}
\lf\|W-K\rg\|_{L^q(A_i)}^q\le&\, C\bigg[\de_i^{2-q} \int_{ A_i\over \de_i } \bigg| {2\al_i^2 |y |^{\al_i-1} \over (1+|y|^{\al_i} )^2 }\bigg|^q\, dy  +\rho^{\eta q }\de_i^{2-2q}  \int_{ A_i\over \de_i } \bigg| {2\al_i^2 |y |^{\al_i-2} \over (1+|y|^{\al_i} )^2 }\bigg|^q\, dy  \\
&\,+\rho^{(1+\tau)q} \de_i^{2+2\tau q}\int_{A_i\over \de_i} \bigg|{(1+|y|^{\al_i})^{2\tau }\over |y|^{(\al_i-2)\tau }} \bigg|^q dy + \sum_{j<i} \de_i^{2-2q}\Big({\de_j\over \de_i }\Big)^{\al_j q} \int_{A_i\over \de_i} \bigg| {1\over |y|^{\al_j+2}}\bigg|^q dy\\
&+ \sum_{j>i} \de_i^{2-2q}\Big({\de_i\over \de_j }\Big)^{\al_j q} \int_{A_i\over \de_i} |y|^{(\al_j-2)q}  dy \bigg] \le C \rho^{ q \eta'_{1,q} } 
\end{split}
\end{equation*}
for some $\eta'_{1,q}>0$, see Lemma \ref{estint}. Similarly, there holds $\lf\|W-K\rg\|_{L^q(A_i)}^q\le C \rho^{ q \eta'_{2,q} } $ for any $q\ge 1$ and for $i$ even 
for some $\eta'_{2,q}>0$. Hence, we get that
\begin{equation*}
\begin{split}
\| \Lambda (\phi) \|_p  \le &\, \lf\|\lf(W-K\rg)\phi\rg\|_p \le   \lf\| W-K\rg\|_{pr } \|\phi \|_{ps } \le C \rho^{\eta'_{p } }  \|\phi \|,
\end{split}
\end{equation*}
where $ \eta'_p=\min\lf\{\eta'_{1,pr }, \eta'_{2,ps} \rg\}$ with $ r $, $s$ satisfying $\frac{1}{ r }+\frac{1}{ s }=1$. Furthermore, we have used the H\" older's inequality $\|uv\|_q\le \|u\|_{qr}\|v\|_{qs}$ with ${1\over r }+{1\over s }=1$ and the inclusions $L^{p}(\Om_\e)\hookrightarrow L^{pr}(\Om_\e)$  for any $r>1$ and $H^{1}_0(\Om_\e)\hookrightarrow L^{q}(\Om_\e)$ for any $q>1$. Let us stress that we can choose $p$, $r$ and $s$ close enough to 1 such that $\eta_p'>0$.
\end{proof}

\medskip

\begin{lem}
There exist $p_0>1$ and $\rho_0>0$ so that for any $1<p<p_0$ and 
all $0<\rho\leq \rho_0$ it holds
\begin{equation}\label{estnphi}
\|  N (\phi_1)- N(\phi_2) \|_p  \le C\rho^{\eta_p''} \|\phi_1-\phi_2\|
\end{equation}
for all $\phi_i\in H_0^1(\Om_\e)$ with $\|\phi_i\|\le M \rho^{\eta_p}|\log\rho|$, $i=1,2$, and for some $\eta_p''>0$. In particular, we have that
\begin{equation}\label{estnphi1}
\| N (\phi) \|_p  \le C\rho^{\eta_p''} \ \|\phi\|
\end{equation}
for all $\phi\in H_0^1(\Om_\e)$ with $\|\phi\|\le M \rho^{\eta_p}|\log\rho|$.
\end{lem}

\begin{proof}[\dem] We will argue in the same way as in \cite[Lemma 5.1]{AP}, see also \cite[Lemma 3.4]{EFP}. First, denoting $f_i(\phi)= V_i(x) e^{(-\tau)^{i}(U+\phi)}$ we point out that
\begin{equation}\label{mvtn}
N(\phi)=\sum_{i=0}^1\rho \lf\{f_i(\phi)-f_i(0)-f'_i(0)[\phi] \rg\}\ \ \ \text{and} \ \ \
N(\phi_1)-N(\phi_2) =\sum_{i=0}^1\rho  f_i''(\ti \phi_{\mu_i}) [\phi_{\theta_i}, \phi_1-\phi_2],
\end{equation}
by the mean value theorem, where $\phi_{\theta_i}=\theta_i\phi_1+(1-\theta_i)\phi_2$, $\ti\phi_{\mu_i}=\mu_i\phi_{\theta_i}$ for some $\theta_i,\mu_i\in[0,1]$, $i=0,1$, and \linebreak $f_i''(\phi)[\psi,v]=\tau^{2i} V_i(x)e^{ (-\tau)^{i}(U+\phi) } \psi v$. Using H\"older's inequalities we get that
\begin{equation}\label{hin}
\lf\| f_i''(\phi)[\psi,v]\rg\|_p\le  |\tau|^{2i} \| V_ie^{ (-\tau)^{i}(U+\phi)} \|_{pr_i}  \|\psi\|_{ps_i} \|v\|_{pt_i}
\end{equation}
with $ {1\over r_i} +{1\over s_i} + {1\over t_i}=1$. We have used the H\" older's inequality and $ \|uvw\|_q\le \|u\|_{qr}\|v\|_{qs}\|w\|_{qt}$ with $ {1\over r }+{1\over s }+{1\over t}=1$. Now, let us estimate $  \| V_ie^{ (-\tau)^{i}(U+\phi)} \|_{pr_i} $ with $\phi=\ti\phi_{\mu_i}$, $i=0,1$. By \eqref{veujo} and the change of variable $x=\delta_i y$ let us estimate
\begin{equation}\label{1036}
\int_{A_i} \lf| V_0 e^{ U} \rg|^q dx
= O\lf( \rho^{-q} \delta_i^{2- 2q}  \int_{\frac{A_i}{\delta_i} } \lf| \frac{|y|^{\alpha_i-2}}{(1+|y|^{\alpha_i})^2 } \rg|^q \lf[1+  O\big( \rho^{\eta} +\delta_i |y| \big)\rg]^q \rg) = O\lf( \rho^{{\be_i\over \al_i}(2-2q)-q} \rg) 
\end{equation}
for any $i$ odd and similarly, by \eqref{vetuje} we get that
\begin{eqnarray} \label{10372}
\int_{A_i} \lf|V_1 e^{-\tau U} \rg|^qdx=O\lf( \rho^{{\be_i\over \al_i}(2-2q)-q}   \rg)
\end{eqnarray}
for any $i$ even, in view of \eqref{dei}. By \eqref{veujo2} we get the estimate
\begin{eqnarray} \label{0054}
\int_{A_i} \lf|V_0 e^{ U} \rg|^qdx&=&
 O\bigg(\rho^{q\over \tau} \de_i^{2+{2q \over \tau}} \int\limits_{\frac{A_i}{\delta_i} } \Big[\frac{|y|^{\alpha_i-2}}{(1+|y|^{\alpha_i})^2 }  \Big]^{-{q\over \tau} } dy\bigg) = O\big(\rho^{-q+q\eta_q}\big) \nonumber 
\end{eqnarray}
for all $i$ even, in view of Lemma \ref{estint}. Similarly, by \eqref{vetuje2} we deduce that
\begin{eqnarray} \label{0107}
\int_{A_i} \lf| V_1 e^{-\tau  U} \rg|^q dx=O(\rho^{-q+q\eta_q})
\end{eqnarray}
for $i $ odd, again in view of Lemma \ref{estint}. Therefore, by using \eqref{1036} and \eqref{0054} and that $\frac{\be_i}{\al_i}$ are decreasing, we deduce that
\begin{equation}\label{v0euq}
\lf\| V_0e^{U  } \rg\|_{q}^q=\sum_{i=1\atop i \text{ odd}}^{m} O\lf( \rho^{{\be_i\over \al_i}(2-2q) - q } \rg) + O\big(\rho^{-q+q\eta_q}\big)=  O\lf( \rho^{ {\be_1\over \al_1}(2-2q) -q }\rg) \quad\text{for any } q\ge 1
\end{equation}
and, by using \eqref{10372} and \eqref{0107} we obtain that 
\begin{equation}\label{v1etuq}
\lf\| V_1e^{ -\tau U } \rg\|_{q}^q= O\lf( \rho^{ {\be_1\over \al_1}(2-2q) -q }\rg)\quad\text{for any } q\ge 1.
\end{equation}
 
On the other hand, using the estimate $|e^a-1|\le C |a|$ uniformly for any $a$ in compact subsets of $\R$, H\"older's inequality 
with ${1\over s_i'}+{1\over t_i'}=1$, $i=0,1$, $\|\ti\phi_{\mu_i}\|\le M \rho^{\eta_p}|\log \rho|\le C$, $i=0,1$ and triangle inequality  we find that
\begin{equation}\label{ev}
 \big\| V_i e^{(-\tau)^{i} (U+\ti\phi_{\mu_i} )}\big\|_{q}=O\big( \rho^{ \eta_{0,qs_i'} -1 + \eta_p}|\log \rho|+ \rho^{ \eta_{0,q} -1 } \big),
 \end{equation}
where for $q>1$ we denote $ \eta_{0, q } = {\be_1(2- 2  q)\over \al_1 q }$ (on the line of \cite[eq.(3.14) proof of Lemma 3.3]{F1}). Note that $\eta_{0,q}<0$ for any $q>1$.
Also, choosing $q$ and $s_i'$, $i=0,1$, close enough to 1, we get that  $0< \eta_p +\eta_{0,qs_i'}$. Now, we can conclude the estimate by using \eqref{mvtn}-\eqref{ev} to get
\begin{equation*}
\begin{split}
\|N(\phi_1)-N(\phi_2)\|_p &\,
\le\,C \sum_{i=0}^1\rho  \|V_i e^{ (- \tau)^{i}(U+\ti\phi_{\mu_i} )}\|_{p r_i}  \| \phi_{\theta_i} \| \| \phi_1-\phi_2\| 
 \le\,C \sum_{i=0}^1 \rho^{\eta_p +\eta_{0,pr_i} }  | \log \rho| \| \phi_1-\phi_2\|
\end{split}
\end{equation*}
and \eqref{estnphi} follows, where $\eta''_p={1\over 2}\min\{ \eta_p + \eta_{0,pr_i}  \ : \ i=0,1\} >0$ choosing $r_i$ close to 1 so that $\eta_p+\eta_{0,pr_i} >0$ for $i=0,1$. Let us stress that $p>1$ is chosen so that $\eta_p>0$ and $\eta_p'>0$.
\end{proof}

\begin{proof}[\bf Proof of the Proposition \ref{p3}]  Taking into account Proposition \ref{p2}, estimates \eqref{re}, \eqref{estlaphi}, \eqref{estnphi}, \eqref{estnphi1} and standard arguments it turns out that for all $\rho>0$ sufficiently small $\ml{A}$ is a
contraction mapping of $\ml{F}_M$ (for $M$ large enough), and
therefore a unique fixed point of $\ml{A}$ exists in $\ml{F}_M$, where $\ml{A}(\phi):= -T(R +\Lambda(\phi) +N(\phi))$ and $
\ml{F}_M = \{\phi\in H_0^1(\Om_\e) : \| \phi \| \le
M \rho^{\eta_p} |\log\rho|\}$. See the proof of Proposition 3.2 in \cite{EFP} for more details.
\end{proof}

\begin{proof}[\bf Proof of the Theorem \ref{main}] Taking into account \eqref{solurho} and the definition of $U$, the existence of a solution to equation \eqref{spsh} follows directly by Proposition \ref{p3}. The asymptotic behavior of $u_\rho$ as $\rho\to 0^+$ follows from \eqref{waje} in Lemma \ref{expaU} and estimate for $\phi$ in Proposition \ref{p3}. Furthermore, we have that $u_\rho$ has the desired concentration properties \eqref{aburo}  as $\rho\to 0^+$
locally uniformly in $\bar\Om \sm\{0\}$.	
\end{proof}


\section{Problem in mean field form \eqref{mfsh}}
\subsection{Error estimate}
Assume that $\de_i$'s are given by \eqref{dei} with $\al_i$'s, $\be_i$'s and $d_i$'s defined in \eqref{ali}, \eqref{beimp}-\eqref{beimi} and \eqref{di}, respectively, and $\e$ is given by \eqref{eps}. Let us evaluate the error term in $\|\cdot\|_p$, encoded in \eqref{Rmf}.
\begin{lem}\label{estrr0mf}
There exist $\rho_0>0$, a constant $C>0$ and $1<p_0<2$, \st for any $\rho\in(0,\rho_0)$ and $p\in(1,p_0)$ it holds
\begin{equation}\label{remf}
\|\ml R\|_p\le  C\rho^{\eta_p}
\end{equation}
for some $\eta_p>0$.
\end{lem}

\begin{proof}[\dem]
By using \eqref{veujo} and \eqref{vetuje2} we have that
\begin{equation}\label{intv0eu}
\begin{split}
\int_{\Om_\e}&\, V_0(x) e^{U}  =  \sum_{j=1}^m \int_{A_j\over \de_j} \de_j^2 V_0(\de_j y) e^{U(\de_j y)}dy
= \sum_{j=1\atop j\text{ odd}}^m {1\over \rho} \int_{A_j\over \de_j} {2\al_j^2|y|^{\al_j-2}\over (1+|y|^{\al_j})^2} \lf[1+O(\de_j|y| + \rho^\eta)\rg]dy \\
& + \sum_{j=1\atop j\text{ even} }^mO\bigg(\rho^{1/\tau} \de_j^{2+2/ \tau} \int_{A_j\over \de_j} {(1+|y|^{\al_j})^{2/\tau}\over |y|^{(\al_j-2)/\tau}} dy\bigg)
= {1\over\rho} \bigg[4\pi\sum_{j=1\atop j\text{ odd}}^m\al_j + O(\rho^{\eta})\bigg]
\end{split}
\end{equation}
and similarly, from \eqref{veujo2} and \eqref{vetuje} we get that
\begin{equation}\label{intv1etu}
\begin{split}
\int_{\Om_\e} V_1(x) e^{-\tau U } = &\,{1\over\rho \tau} \bigg[4\pi\sum_{j=1\atop j\text{ even}}^m\al_j + O(\rho^{\eta})\bigg],
\end{split}
\end{equation}
in view of $\int\limits_{\mathbb R^2}  {|y|^{\alpha_i-2}\over \left(1+|y|^{\al_i}\right)^2}dy={2\pi\over\alpha_i}.$ From assumptions \eqref{la0}-\eqref{la1} it follows that $\la_0=4\pi\sum_{j=1\atop j\text{ odd}}^m\al_j$ and $\la_1\tau^2=4\pi\sum_{j=1\atop j\text{ even}}^m\al_j$
so that, by using \eqref{intv0eu}-\eqref{intv1etu} we find that for $y\in \frac{A_j}{\de_j}$ and any $j$ odd
\begin{equation}\label{l0veujo}
\la_0{ V_0(\de_j y) e^{U(\de_j y)} \over \int_{\Om_\e} V_0(x)e^{U} }= {2\al_j^2|y|^{\al_j-2}\over \de_j^2(1+|y|^{\al_j})^2} \lf[1+O(\de_j|y| + \rho^\eta)\rg]
\end{equation}
and
\begin{equation}\label{la1veujo2}
\la_1\tau{ V_1(\de_j y) e^{-\tau U(\de_j y)} \over \int_{\Om_\e} V_1(x) e^{-\tau U} } =  O\bigg(\rho^{1+\tau} \de_j^{2\tau} {(1+|y|^{\al_j})^{2\tau} \over |y|^{(\al_j-2)\tau }} \bigg).
\end{equation}
Similarly, for any $y\in \frac{A_j}{\de_j}$ and $j$ even
\begin{equation}\label{l0veujo2}
\la_0{ V_0(\de_j y) e^{U(\de_j y)} \over \int_{\Om_\e} V_0(x)e^{U} }= O\bigg(\rho^{1+1/\tau} \de_j^{2/\tau} {(1+|y|^{\al_j})^{2/\tau} \over |y|^{(\al_j-2)/\tau }} \bigg)
\end{equation}
and
\begin{equation}\label{la1veujo}
\la_1\tau{ V_1(\de_j y) e^{-\tau U(\de_j y)} \over \int_{\Om_\e} V_1(x) e^{-\tau U} } =  {2\al_j^2|y|^{\al_j-2}\over \de_j^2\tau (1+|y|^{\al_j})^2} \lf[1+O(\de_j|y| + \rho^\eta)\rg].
\end{equation}
Hence, similar to the proof of Lemma \ref{estrr0} by using \eqref{ludjy}-\eqref{ludmy} and \eqref{l0veujo}-\eqref{la1veujo} for $\de_1 y\in A_1$ we get that
\begin{equation}\label{estRd1mf}
\begin{split}
\ml R(\de_1 y) 
=&\,  {1\over\de_1^2} {2\al_1^2|y|^{\al_1-2}\over (1+|y|^{\al_1})^2} O(\de_1|y| + \rho^\eta)  + O\bigg( \rho^{1+\tau} \de_1^{2\tau} {(1+|y|^{\al_1})^{2\tau} \over |y|^{(\al_1-2)\tau }}+ \sum_{i=2}^m \Big({\de_1\over \de_i}\Big)^{\al_i} {|y|^{\al_i-2}\over \de_1^2}\bigg),
\end{split}
\end{equation}
for $1<j<m$, $j$ odd	and $\de_jy\in A_j$
\begin{equation}
\begin{split}
\ml R(\de_j y)=&\,  {1\over\de_j^2} {2\al_j^2|y|^{\al_j-2}\over (1+|y|^{\al_j})^2} O(\de_j|y| + \rho^\eta) \\
&\, + O\bigg( \rho^{1+\tau} \de_j^{2\tau} {(1+|y|^{\al_j})^{2\tau} \over |y|^{(\al_j-2)\tau }} + \sum_{i=1}^{j-1} \Big({\de_i\over \de_j}\Big)^{\al_i} {1\over \de_j^2|y|^{\al_i+2}}+ \sum_{i=j+1}^m \Big({\de_j\over \de_i}\Big)^{\al_i} {|y|^{\al_i-2}\over \de_j^2}\bigg)
\end{split}
\end{equation}
for $1<j<m$, $j$ even and $\de_jy\in A_j$
\begin{equation}
\begin{split}
\ml R(\de_j y)=&\,  {1\over\de_j^2\tau} {2\al_j^2|y|^{\al_j-2}\over (1+|y|^{\al_j})^2} O(\de_j|y| + \rho^\eta) \\
&\, + O\bigg( \rho^{1+1/\tau} \de_j^{2/\tau} {(1+|y|^{\al_j})^{2/\tau} \over |y|^{(\al_j-2)/\tau }} + \sum_{i=1}^{j-1} \Big({\de_i\over \de_j}\Big)^{\al_i} {1\over \de_j^2|y|^{\al_i+2}}+ \sum_{i=j+1}^m \Big({\de_j\over \de_i}\Big)^{\al_i} {|y|^{\al_i-2}\over \de_j^2}\bigg)
\end{split}
\end{equation}
and for $\de_m y\in A_m$
\begin{equation}\label{estRdmmf}
\begin{split}
\ml R(\de_m y)=&\,  {1\over\de_m^2\tau^{\nu(m)} } {2\al_m^2|y|^{\al_m-2}\over (1+|y|^{\al_m})^2} O(\de_m|y| + \rho^\eta) \\
&\, + O\bigg( \rho^{1+\tau^{(-1)^{m+1} } } \de_m^{2\tau^{(-1)^{m+1} } } {(1+|y|^{\al_m})^{2\tau^{(-1)^{m+1} } } \over |y|^{(\al_m-2)\tau^{(-1)^{m+1} } }} + \sum_{i=1}^{m-1} \Big({\de_i\over \de_m}\Big)^{\al_i} {1\over \de_m^2|y|^{\al_i+2}} \bigg)
\end{split}
\end{equation}
By \eqref{estRd1mf}-\eqref{estRdmmf} and Appendix, Lemma \ref{estint}, it follows \eqref{remf} and this conclude the proof.
\end{proof}

\subsection{The nonlinear problem and proof of Theorem \ref{main2}.}
In this section we shall study the following nonlinear problem:
\begin{equation}\label{ephimf}
\left\{ \begin{array}{ll}
\mathcal L(\phi)= -[\mathcal R+\Lambda_0(\phi)+\mathcal N(\phi)] & \text{in } \Om_\e\\
\phi=0, &\text{on }\fr\Om_\e,
\end{array} \right.
\end{equation}
where the linear operators $\mathcal L,\Lambda_0$ are defined as
\begin{equation}\label{olmf}
\mathcal L(\phi) = \Delta \phi + K_0\lf(\phi - {1\over \la_0}\int_{\Om_\e} K_0 \phi  dx \rg) + K_1 \lf(\phi - {1\over \la_1\tau^2} \int_{\Om_\e} K_1\phi dx \rg)
\end{equation}
and
\begin{equation}\label{olamf}
\begin{split}
\Lambda_0 (\phi) =&\,  \la_0 {V_0 e^{U}\over\int_{\Om_\e} V_0e^{U} dx}\lf(\phi - {\int_{\Om_\e} V_0e^{U}\phi dx \over\int_{\Om_\e} V_0e^{U} dx } \rg)+\la_1\tau^2 {V_1e^{-\tau U}\over\int_{\Om_\e} V_1e^{-\tau U}dx }\lf(\phi - {\int_{\Om_\e} V_1 e^{-\tau U}\phi dx \over\int_{\Om_\e} V_1(x)e^{-\tau U}dx} \rg)\\
&\, - K_0\lf(\phi - {1\over \la_0}\int_{\Om_\e} K_0 \phi  dx \rg)-  K_1 \lf(\phi - {1\over \la_1\tau^2} \int_{\Om_\e} K_1\phi dx \rg)
\end{split}
\end{equation}
with
\begin{equation}\label{K12}
K_0=\sum_{k=1\atop k\text{ odd} }^{m}|x|^{\al_k-2}e^{w_k}, \quad K_1=\sum_{k=1\atop k\text{ even} }^{m}|x|^{\al_k-2}e^{w_k}.
\end{equation}
The nonlinear term $\mathcal N(\phi) $ is given by
\begin{equation}\label{nltmf}
\begin{aligned}
\mathcal N(\phi)= &\la_0 \lf[{V_0e^{U+\phi}\over\int_{\Om_\e}
V_0e^{U+\phi} dx}-{V_0 e^{U}\over\int_{\Om_\e}V_0e^{U} dx}-{V_0e^{U}\over\int_{\Om_\e}V_0
e^{U} dx}\lf(\phi - {\int_{\Om_\e} V_0e^{U}\phi dx \over\int_{\Om_\e}V_0
e^{U} dx} \rg)\rg]\\ 
& - \la_1\tau \lf[{V_1 e^{-\tau( U+\phi) } \over\int_{\Om_\e}
V_1 e^{-\tau (U+\phi)}dx}-{V_1 e^{-\tau U}\over\int_{\Om_\e} V_1 e^{-\tau U} dx } +\tau{V_1 e^{ - \tau U}\over\int_{\Om_\e}V_1
e^{-\tau U} dx }\lf(\phi - {\int_{\Om_\e} V_1e^{ - \tau U}\phi dx \over\int_{\Om_\e} V_1
e^{-\tau U} dx } \rg)\rg].
\end{aligned}
\end{equation}
It is readily checked that $\phi$ is a solution to \eqref{ephimf} if and only if $u_\e$ given by \eqref{solurho} is a solution to \eqref{mfsh}. In Section \ref{sec3} we will prove the following result.

\begin{prop}\label{elle}
For any $p>1,$ there exists $\rho_0>0 $ and  $C>0$ such that for any $\rho\in(0,\rho_0)$ and $h \in L^p(\Om_\e)$ there exists a unique $\phi\in H^1_0(\Om_\e)$ solution of
\begin{equation}\label{plmf}
\mathcal L(\phi)=h \ \hbox{ in }\ \Om_\e,\ \ \ \phi=0\ \hbox{ on }\ \partial\Omega_\e,
\end{equation}
which satisfies
\begin{equation}\label{estphi}
\|\phi\|\le C|\log\rho|\, \|h\|_p.
\end{equation}
\end{prop}

\begin{lem}
There exist $p_0>1$ and $\rho_0>0$ so that for any $1<p<p_0$ and 
all $0<\rho\leq \rho_0$ it holds
\begin{equation}\label{estlaphimf}
\| \Lambda_0 (\phi) \|_p  \le C\rho^{\sigma_p'} \|\phi\|,
\end{equation}
for all $\phi\in H_0^1(\Om_\e)$ with $\|\phi\|\le M \rho^{\sigma_p}|\log\rho|$, for some $\sigma_p'>0$.
\end{lem}

\begin{proof}[\dem]
Arguing in the same way as in \cite[Lemma 3.3]{EFP}, denote $W_i= { \la_i\tau^{2i} V_i e^{(-\tau)^{i} U}\over \int_{\Om_\e} V_i e^{(-\tau)^{i} U } dx}$ for $i=0,1$. By using \eqref{l0veujo}, \eqref{l0veujo2}, Lemma \ref{estint} and similar computations to prove \eqref{estlaphi}, we find that $\lf\|W_0- K_0\rg\|_{L^q(A_i)}^q\le C \rho^{ q \sigma'_{0,q} } $ for any $i$ 
for some $\sigma'_{0,q}$. Similarly, we find that $\lf\|W_1- K_1\rg\|_{L^q(A_i)}^q \le C \rho^{q\sigma'_{1,q} } 
$ for any $i$ for some $\sigma'_{1,q}$. It is possible to see that taking $q>1$ close enough to 1, we get that $\sigma'_{i,q}>0$ for $i=0,1$. 

\medskip \noindent Notice that $\Lambda_0$ is a linear operator and we re-write $\Lambda_0(\phi)$ as
\begin{equation*}
\begin{split}
\Lambda_0(\phi)=&\,\sum_{i=0}^1\bigg[ \lf(W_i-K_i\rg)\phi-{1\over\la_i\tau^{2i} } \lf(W_i-K_i \rg) \int_{\Om_\e}W_i\phi+ {1\over\la_i \tau^{2i} } K_i \int_{\Om_\e} \lf(K_i-W_i\rg)\phi
 \bigg].
\end{split}
\end{equation*}
Hence, we get that
\begin{equation*}
\begin{split}
\| \Lambda_0 (\phi) \|_p  
\le &\, \sum_{i=0}^1\bigg[\lf\| W_i-K_i\rg\|_{pr_{i0}} \|\phi \|_{ps_{i0} }+{ \|W_i\|_{r_{i1}} \over\la_i \tau^{2i} } \lf \| W_i-K_i \rg\|_p \|\phi\|_{s_{i1} } + {\| K_i\|_p\over\la_i \tau^{2(i-1)} }  \lf\| K_i-W_i\rg\|_{r_{i2}} \|\phi \|_{s_{i2}} \bigg]\\
\le &\, C\sum_{i=0}^1\bigg[\rho^{\sigma'_{i,pr_{i0} } }  \|\phi \| + \rho^{\sigma'_{i,p} +\sigma_{3,r_{i1}  }}  \|\phi\| + \rho^{\sigma'_{i,r_{i2} } +\sigma_{3,p} }  \|\phi \| \bigg]
\end{split}
\end{equation*}
and \eqref{estlaphimf} follows, where $ \sigma'_p=\min\lf\{\sigma'_{i,pr_{i0} }; \sigma'_{i,p} +\sigma_{3,r_{i1}  } ;  \sigma'_{i,r_{i2}} +\sigma_{3,p  } \mid i=0,1\rg\}$ with $ r_{ij} $, $s_{ij} $, $i=0,1$, $j=0,1,2$ satisfying $\frac{1}{ r_{ij} }+\frac{1}{ s_{ij} }=1$. We have used that
\begin{equation*}
\begin{split}
\|W_0\|_{r_{01}}^{r_{01} } &\le C\bigg[ \sum_{j=1\atop j \text{ odd}}^{m} \de_j^{2-2r_{01} }\int_{A_j\over \de_j} \lf| {2\al_j^2|y|^{\al_j -2}\over (1+|y|^{\al_i})^2 } \rg|^{r_{01} } + \sum_{j=1\atop j \text{ even}}^m\rho^{(1+1/\tau) r_{01} }\de_j^{2+2{r_{01}\over \tau} } \int_{A_j\over \de_j}\bigg|{(1+|y|^{\al_j})^{2/\tau}\over |y|^{(\al_j-2)/\tau}}\bigg|^{r_{01}}\bigg] \\
&\le C\rho^{\sigma_{3,r_{01}}} 
\end{split}
\end{equation*}
and similarly, $\|W_1\|_{r_{11}}^{r_{11} } \le C\rho^{\sigma_{3,r_{11}} } $, 
where $\sigma_{3,q}= {\be_1\over \al_1}(2-2 q) $ and similarly that $\| K_i\|_p^p\le C\rho^{\sigma_{3,p}   }$, $i=0,1$. Note that 
${2- 2  q\over\al_j q}<1$ for any $j=1,\dots,m$. 
Furthermore, we have used the H\" older's inequality $\|uv\|_q\le \|u\|_{qr}\|v\|_{qs}$ with ${1\over r }+{1\over s }=1$ and the inclusions $L^{p}(\Om_\e)\hookrightarrow L^{pr}(\Om_\e)$  for any $r>1$ and $H^{1}_0(\Om_\e)\hookrightarrow L^{q}(\Om_\e)$ for any $q>1$. Let us stress that we can choose $p$, $r_{ij}$ and $s_{ij}$, $i=1,2$, $j=0,1,2$, close enough to 1 such that $\sigma_p'>0$.
\end{proof}

\medskip
\begin{lem}
There exist $p_0>1$ and $\rho_0>0$ so that for any $1<p<p_0$ and 
all $0<\rho\leq \rho_0$ it holds
\begin{equation}\label{estnphimf}
\| \mathcal N (\phi_1)- \ml N(\phi_2) \|_p  \le C\rho^{\sigma_p''} \|\phi_1-\phi_2\|
\end{equation}
for all $\phi_i\in H_0^1(\Om_\e)$ with $\|\phi_i\|\le M \rho^{\sigma_p}|\log\rho|$, $i=1,2$, and for some $\sigma_p''>0$. In particular, we have that
\begin{equation}\label{estnphi1mf}
\| \ml N (\phi) \|_p  \le C\rho^{\sigma_p''} \ \|\phi\|
\end{equation}
for all $\phi\in H_0^1(\Om_\e)$ with $\|\phi\|\le M \rho^{\sigma_p}|\log\rho|$.
\end{lem}

\begin{proof}[\dem] Arguing in the same way as in \cite[Lemma 5.1]{AP}, we denote $g_i(\phi)= {V_i(x) e^{(-\tau)^{i}(U+\phi)} \over \int_{\Om_\e} V_i(x)e^{  (-\tau)^{i}( U+\phi)} }$ and point out that
$$\ml N(\phi)=\sum_{i=0}^1\la_i(-\tau)^{i}\lf\{g_i(\phi)-g_i(0)-g'_i(0)[\phi] \rg\} \ \ \ \text{and by the mean value theorem we get that}$$
\begin{equation}\label{mvtnmf}
\ml N(\phi_1)-\ml N(\phi_2) =\sum_{i=1}^2\la_i ( - \tau )^{i} g_i''(\ti \phi_{\mu_i}) [\phi_{\theta_i}, \phi_1-\phi_2],
\end{equation}
where $\phi_{\theta_i}=\theta_i\phi_1+(1-\theta_i)\phi_2$, $\ti\phi_{\mu_i}=\mu_i\phi_{\theta_i}$ for some $\theta_i,\mu_i\in[0,1]$, $i=0,1$, and
\begin{equation*}
\begin{split}
g_i''(\phi)[\psi,v]=&\,\tau^{2i}\bigg[ { V_i(x)e^{ u_i } \psi v\over \int_{\Om_\e} V_i(x)e^{ u_i} }- { V_i(x)e^{u_i } v \int_{\Om_\e} V_i(x)e^{ u_i }\psi \over \big(\int_{\Om_\e} V_i(x)e^{ u_i }\big)^2 }- { V_i(x)e^{ u_i } \psi \int_{\Om_\e} V_i(x)e^{ u_i } v \over \big(\int_{\Om_\e} V_i(x)e^{ u_i } \big)^2 }\\
&\,- { V_i(x)e^{ u_i }  \int_{\Om_\e} V_i(x)e^{ u_i }\psi v\over \big(\int_{\Om_\e} V_i(x)e^{ u_i }\big)^2 } +2{ V_i(x)e^{ u_i } \int_{\Om_\e} V_i(x) e^{ u_i  }v \int_{\Om_\e} V_i(x)e^{ u_i }\psi \over \big(\int_{\Om_\e} V_i(x)e^{ u_i }\big)^3 }\bigg],
\end{split}
\end{equation*}
where for simplicity we denote $ u_i= (-\tau)^{i}(U+\phi)$. Using H\"older's inequalities we get that
\begin{equation}\label{hin}
\begin{split}
\lf\| g_i''(\phi)[\psi,v]\rg\|_p
\le&\,C \lf[{\| V_ie^{ u_i } \|_{pr_i}\over \|V_ie^{ u_i}\|_1 }+ {\| V_ie^{u_i }\|_{pr_i}^2  \over \|V_ie^{ u_i }\|_1^2 }  +{ \|V_ie^{ u_i }\|_{p r_i}^3  \over \|V_ie^{ u_i }\|_1^3 }  \rg]\|\psi\| \| v\|,
\end{split}
\end{equation}
with $ {1\over r_i} +{1\over s_i} + {1\over t_i}=1$, $ {1\over r_i}+{1\over q_i}=1$, $ {1\over pr_i} + {1\over \ti r_i}=1$. We have used the H\" older's inequality, the inclusions presented in the previous Lemma and $ \|uvw\|_q\le \|u\|_{qr}\|v\|_{qs}\|w\|_{qt}$ with ${1\over r }+{1\over s }+{1\over t}=1$. Now, let us estimate $ {\| V_ie^{ u_i } \|_{pr_i}\over \| V_i e^{ u_i}\|_1 }$ with $\phi=\ti\phi_{\mu_i}$, $i=0,1$. Taking into account  \eqref{v0euq}-\eqref{v1etuq} and \eqref{intv0eu}-\eqref{intv1etu}, we obtain that for $i=0,1$
$$\lf\| V_ie^{(-\tau)^{i} U  } \rg\|_{q}^q=O\lf( \rho^{ {\be_1\over \al_1}(2-2q)-q }\rg) \quad\text{for any } q\ge 1\quad\text{and}\quad \lf\| V_ie^{(- \tau)^{i} U } \rg\|_{1}\ge {C_i\over \rho }\quad\text{for some } C_i> 0.$$
Note that $ \sigma_{3,q} -1\le {2-(\al_i+2)q\over \al_i q}\quad\text{ for any} \quad i=1,\dots,m$.  By the previous estimates we find that 
$ \big\| V_i e^{(-\tau)^{i} U+\ti\phi_{\mu_i}}\big\|_{q}=O\big( \rho^{ \eta_{0,qs_i'} -1 + \eta_p}|\log \rho|+ \rho^{ \eta_{0,q} -1 } \big)$ (on the line of \cite[eq. (3.14) proof of Lemma 3.3]{F1}. Also, choosing $s_i$, $i=0,1$, close enough to 1, we get that  $\sigma_p +\eta_{0,s_i}>0$ and 
$$\lf\| V_ie^{(-\tau)^{i}(U+\ti\phi_{\mu_i} ) }\rg\|_{1}\ge {C_i \over \rho } - C \rho^{\eta_{0,s_i} -1+ \sigma_p}|\log \rho| \ge { 1 \over \rho } \lf(C_i -  C \rho^{\eta_{0,s_i} + \sigma_p }\  |\log \rho| \rg)\ge {C_i \over 2 \rho } .$$
Taking $q=pr_i$, we obtain the estimate for $i=0,1$
\begin{equation}\label{qie}
\begin{split}
{\| V_ie^{(-\tau)^{i}( U+\ti\phi_{\mu_i}) } \|_{pr_i}\over \| V_i e^{(-\tau)^{i}(  U +\ti\phi_{\mu_i} ) }\|_1 }
&=O\lf(  \rho^{\sigma_{3,pr_i} } \Big[ \rho^{\sigma_p+\sigma_{3,pr_is_i} - \sigma_{pr_i} } |\log \rho|+ 1\Big]     \rg)
=O\lf(  \rho^{ \sigma_{3,pr_i} }    \rg)
\end{split}
\end{equation}
choosing $s_i>1$ close enough to 1 so that $ \sigma_p +\sigma_{3,pr_is_i} - \sigma_{pr_i}>0$, $i=1,2$. 
Now, we can conclude the estimate by using \eqref{mvtnmf}-\eqref{qie} to get
\begin{equation*}
\begin{split}
\|\ml N(\phi_1)-\ml N(\phi_2)\|_p 
&\, \le\,C \sum_{i=0}^1\la_i \rho^{\sigma_p + 3 \sigma_{3,pr_i} }  | \log \rho| \| \phi_1-\phi_2\|
 \le\,C \rho^{\sigma_p'' }  \|\phi_1-\phi_2\|,
\end{split}
\end{equation*}
where $\sigma''_p={1\over 2}\min\{ \sigma_p +3 \sigma_{3,pr_i}  \ : \ i=0,1\} >0$ choosing $r_i$ close to 1 so that $\sigma_p+3 \sigma_{3,pr_i} >0$ for $i=0,1$. Let us stress that $p>1$ is chosen so that $\sigma_p>0$.
\end{proof}

\medskip
Taking into account previous estimates \eqref{remf}, \eqref{estlaphimf}, \eqref{estnphimf} and \eqref{estnphi1mf}, and by using the same argument as in the proof of Proposition \ref{p3} we conclude the following 

\begin{prop}\label{p3mf}
There exist $p_0>1$ and $\rho_0>0$ so that for any $1<p<p_0$ and 
all $0<\rho\leq \rho_0$, the problem \eqref{ephimf} admits a unique solution $\phi(\rho) \in H_0^1(\Om_\e)$, where $\mathcal L$, $\mathcal R$, $\Lambda_0(\phi)$ and $\mathcal N$ are given by \eqref{olmf}, \eqref{Rmf}, \eqref{olamf} and \eqref{nltmf}, respectively. Moreover, there is a constant $C>0$ such that
$$\|\phi\|_\infty\le C\rho^{\sigma_p} |\log \rho |,$$
for some $\sigma_p>0$.
\end{prop}

\begin{proof}[\bf Proof of the Theorem \ref{main2}] Arguing in the same way as in the proof of Theorem \ref{main}, the existence of a solution \eqref{solurho} to equation \eqref{mfsh} follows directly by Proposition \ref{p3mf} and the definition of $U$ in \eqref{ansatz}.
\end{proof}


\section{The linear theory}\label{sec3}
\noindent In this section we present the invertibility of the linear operators $L$ and $\ml L$ defined in \eqref{ol} and \eqref{olmf} respectively. Roughly speaking in the scale annulus $\frac{A_j}{\de_j}$ the operator $L$ apporaches to the following linear operator in $\R^2$
$$L_j(\phi)=\Delta\phi+{2\al_j^2|y|^{\al_j-2}\over (1+|y|^{\al_j})^2}\phi,\qquad j=1,\dots,m.$$
It is well known that, in case $\al_j\in 2\N$, the bounded solutions of $L_j(\phi)=0$ in
$\R^2$ are precisely linear combinations of the functions
\begin{equation*}
Y_{1j}(y) = { |y|^{\al_j\over 2} \over 1+|y|^{\al_j}}\cos\Big({\al_j\over 2}\theta\Big),\quad Y_{2j}(y) = { |y|^{\al_j\over 2} \over 1+|y|^{\al_j}}\sin\Big({\al_j\over 2}\theta\Big)\quad\text{and}\quad Y_{0j}(y) = \,{1-|y|^{\al_j}\over 1+|y|^{\al_j}},
\end{equation*}
which are written in polar coordinates for $j=1,\dots,m$. See \cite{DEM5} for a proof. In our case, we will consider solutions of $L_j(\phi)=0$ with $\al_j\notin 2\N$ for all $j=1,\dots,m$ such that $\int_{\R^2}|\nabla\phi(y)|^2\, dy<+\infty$, which are multiples of $Y_{0j}$, see \cite[Theorem A.1]{AP} for a proof. 
Another key element in the study of $\ml L$, which shows technical details, is to get rid of the presence of
\begin{equation}\label{ctjphi}
\ti c_{j}(\phi)=-{1\over \la_j\tau^{2j}}\int_{\Om_{\e}} K_j\phi\qquad j=0,1.
\end{equation}

Let us introduce the following Banach spaces for $j=1,2,\dots,m$
\begin{equation}\label{lai}
L_{\al_j}(\R^2)=\lf\{u\in W_{\text{loc} }^{1,2}(\R^2)\ :\ \int_{\R^2}{|y|^{\al_j-2}\over (1+|y|^{\al_j})^2}|u(y)|^2\, dy<+\infty\rg\}
\end{equation}
and
\begin{equation}\label{hai}
H_{\al_j}(\R^2)=\lf\{u\in W_{\text{loc} }^{1,2}(\R^2)\ :\ \int_{\R^2}|\nabla u(y)|^2\, dy+\int_{\R^2}{|y|^{\al_j-2}\over (1+|y|^{\al_j})^2}|u(y)|^2\, dy<+\infty\rg\}
\end{equation}
endowed with the norms
$$\|u\|_{L_{\al_j}}:=\lf(\int_{\R^2}{|y|^{\al_j-2}\over (1+|y|^{\al_j})^2}|u(y)|^2\, dy\rg)^{1/2}$$
and
$$\|u\|_{H_{\al_j}}:=\lf(\int_{\R^2} |\nabla u(y)|^{2}\, dy+\int_{\R^2}{|y|^{\al_j-2}\over (1+|y|^{\al_j})^2}|u(y)|^2\, dy\rg)^{1/2}.$$
We point out the compactness of the embedding $i_{\al_j}:H_{\al_j}(\R^2)\to L_{\al_j}(\R^2)$, (see for example \cite{GP}).

\begin{proof}[\bf Proof of Proposition \ref{p2}]

Let us assume the existence of $p>1$, sequences $\rho_n\to0$, $\e_n=\e(\rho_n)\to 0$, functions $h_n\in L^p(\Om_{\e_n})$, $\phi_n\in H_0^{1}(\Om_{\e_n})$ such that
$$L(\phi_n)=h_n\ \ \text{in}\ \ \Om_{\e_n},\ \ \phi_n=0\ \ \text{on}\ \ \fr\Om_{\e_n}$$
$\|\phi_n\|=1$ and $|\log\rho_n|\, \|h_n\|_p=o(1)$ as $n\to+\infty$. We will omit the subscript $n$ in $\de_{i,n} =\de_i$. Recall that $\de_i^{\al_i}=d_{i,n}\rho_n^{\be_i}$. Now, define $\Phi_{i,n}(y):=\phi_{i,n}(\de_i y)$ for $y\in\Om_{i,n}:= \de_i^{-1}\Om_{\e_n}$. Thus, extending $\phi_n=0$ in $\R^2\sm\Om_{\e_n}$ and arguing in the same way as in \cite[Claim 1, section 4]{F1} we can prove the following fact. We provide a sketch of the proof. 

\begin{claim}\label{claim2}
There holds that the sequence $\Phi_{i,n}$ converges (up to subsequence) to $\Phi_i^*=a_iY_{0i}$ for $i=1,\dots,m$, weakly in $H_{\al_i}(\R^2)$ and strongly in $L_{\al_i}(\R^2)$ as $n\to+\infty$ for some constant $a_{i}\in\R$, $i=1,\dots,m$.
\end{claim}

\begin{proof}[\dem]
First, notice that $\|\Phi_{i,n}\|_{H^1_0(\Omega_{i,n})}=1$, for $i=1,\dots,m$.  Then, we want to prove that there is a constant $M>0$ such for all $n$ (up to a subsequence) $\|\Phi_{i,n}\|_{L_{\al_i}}^2\le M$. Notice that for any $i\in\{1,\dots,m\}$ we find that
 in $\Om_{i,n}$
\begin{equation}\label{eqPhi0}
\begin{split}
\lap \Phi_{i,n}&+\de_i^2K (\de_i y) \Phi_{i,n} =\de_i^2h_n(\de_i y).
\end{split}
\end{equation}
Furthermore, it follows that $\Phi_{i,n}\to\Phi_i^*$ weakly in $H^1_0(\Omega_{i,n})$ and strongly in $L^p(K)$ for any $K$   compact sets in $\mathbb R^2$.
Now, we multiply \eqref{eqPhi0} by $\Phi_{i,n}$ for any $i\in\{1,\dots,m\}$, integrate by parts and we obtain that
\begin{equation}\label{npsi}
\sum_{i=1}^{m} 2\al_i^2\| \Phi_{i,n}\|_{L_{\al_i}}^2
=\,1 + o(1)
\end{equation}
since
\begin{equation}\label{dik}
\de_i^2K( \de_i y)=\ds{2\al_i^2|y|^{\al_i-2}\over (1+|y|^{\al_i})^2}+O\bigg(\sum_{j<i}\Big({\de_j\over \de_i}\Big)^{\al_j} + \sum_{i<j} \Big({\de_i\over \de_j}\Big)^{\al_j}\bigg)\quad \text{ for all }i=1,\dots,m
\end{equation}
uniformly for $y$ on compact subsets of $\R^2\sm\{0\}$. Thus, we deduce that $\| \Phi_{i,n}\|_{L_{\al_i}}^2
=O(1)$. Therefore, the sequence $\{\Phi_{i,n}\}_n$ is bounded in $H_{\al_i}(\R^2)$, so that there is a subsequence $\{\Phi_{i,n}\}_n$ and functions $\Phi_i^*$, $i=1,2$ such that $\{\Phi_{i,n}\}_n$ converges to $\Phi_i^*$ weakly in $H_{\al_i}(\R^2)$ and strongly in $L_{\al_i}(\R^2)$.  Hence, taking into account \eqref{eqPhi0} and \eqref{dik} we deduce that $\Phi^*_i$ a solution to 
$$\lap\Phi+{2\al_i^2|y|^{\al_i-2}\over(1+|y|^{\al_i})^2}\Phi=0,\qquad i=1,\dots,m,\qquad  \text{in $\R^2\sm\{0\}$}.$$
It is standard that $\Phi_i^*$, $i=1,\dots,m$ extend to a solution in the whole $\R^2$. Hence, by using symmetry assumptions if necessary, we get that $\Phi^*_i=a_iY_{0i}$ for some constant $a_{i}\in\R$, $i=1,\dots,m$.
\end{proof}

For the next step consider for any $j\in\{1,\dots, m\}$ the functions $Z_{0j}(x)=Y_{0j}(\de_j^{-1} x)=\frac{\de_j^{\al_j}-|x|^{\al_j}}{\de_j^{\al_j}+|x|^{\al_j}}$ so that $-\Delta Z_{0j}=|x|^{\al_j-2}e^{w_j}Z_{0j}.$
By similar arguments as to obtain expansion in Lemma \ref{ewfxi}, 
it follows that  $P_\e Z_{0j}=Z_{0j} +1 -  \gamma_0 G(x,0)+O(\rho^{\ti\sigma})$ uniformly in $\Om_\e$ for some $\ti\sigma>0$, as a consequence of \cite[Lemma 4.1]{EFP} with $m=1$ and $\xi_1=0$. Now, introduce the coefficient $\gamma_{0}$ as
\begin{equation}\label{gama0}
\gamma_{0}\lf[-{1\over 2\pi}\log\e+H(0,0) \rg]=2, 
\end{equation}
so that from \eqref{eps}, $ \gamma_{0}
= -\frac{4\pi(\al_1-2)}{\be_1+1} \cdot\frac{1}{\log\rho} + O\big({1\over |\log\rho|^2}\big)$. 

\bs
For simplicity, we shall denote $\Pi_i(y)=\frac{2\al_i^2|y|^{\al_i-2}}{(1+|y|^{\al_i})^2}$, for $i=1,\dots,m$.

\begin{claim}\label{claimmu}
For all $j=1,\dots,m$, there exist the limit, up to subsequences,
\begin{equation}\label{muj}
\mu_j=\lim_{n\to+\infty} \log\rho_n \int_{\Om_{j,n} }\Pi_j(y)\Phi_{j,n}(y)\, dy.
\end{equation}
Furthermore, it holds
\begin{equation}\label{eqmu1}
\begin{split}
2\sum_{i=1}^m a_i=\bigg(-\frac{\be_1+1}{4\pi(\al_1-2)} + \sum_{i=1}^m(-1)^{i+1}\frac{\be_i}{2\pi \al_i}\bigg) \mu_1,
\end{split}
\end{equation}
and for all $j=1,\dots,m-1$.
\begin{equation}\label{eqmuj}
\mu_j=(-1)^{j-1} \mu_1.
\end{equation}
\end{claim}

\begin{proof}[\dem] Consider tests functions $\gamma_0^{-1} PZ_{0j}$'s and the assumption on $h_n$, $|\log \rho_n|\ \|h_n\|_*=o(1)$. Hence, multiplying equation \eqref{eqPhi0} by $ P_\e Z_{0j}$ and integrating by parts we obtain that
{\small
\begin{equation*}
\begin{split}
& \int_{\Om_\e} hPZ_{0j} 
=-\int_{\Om_\e} \phi |x|^{\al_j-2}e^{w_j}Z_{0j} + \int_{\Om_\e} K\phi  P_\e Z_{0j}
= \int_{\Om_\e} \(K - |x|^{\al_j-2}e^{w_j}\)  P_\e Z_{0j} \phi\\
&+ \int_{\Om_\e} |x|^{\al_j-2} e^{w_j} \( PZ_{0j}-Z_{0j}\) \phi 
= \int_{\Om_\e}\phi \sum_{i=1\atop i\ne j}^m |x|^{\al_i-2}e^{w_i}PZ_{0j} +   \int_{\Om_\e} |x|^{\al_j-2} e^{w_j}\lf[ 1 - \gamma_0 G(x,0) + O(\rho^{\ti \sigma})\rg]\phi
\end{split}
\end{equation*}}
in view of $P_\e Z_{0j}=0$ on $\fr\Om_{\e_n}$ and $ \int_{\Om_\e}\lap\phi P_\e Z_{0j}=\int_{\Om_\e}\phi\; \lap P_\e Z_{0j}=\int_{\Om_\e}\phi\; \lap Z_{0j}$.  Now, multiplying by $\gamma_0^{-1}$ and estimating every integral term we find that $\gamma_0^{-1} \int_{\Om_\e} hPZ_{0j}=O\lf(\, |\log\rho|\,\|h\|_p\rg)=o\lf( 1 \rg),$ 
in view of $PZ_{0j}=O(1)$ and $\gamma_0^{-1}=- \frac{\be_1+1}{4\pi (\al_1-2)} \log\rho + O(1) =O(|\log\rho|)$ in \eqref{gama0}. Next, scaling we obtain that
$$\gamma_0^{-1}\int_{\Om_\e} \phi |x|^{\al_j-2}e^{w_j}=- \frac{\be_1+1}{4\pi (\al_1-2)} \log\rho \int_{\Om_{j,n } }  {2\al_j^2|y|^{\al_j-2}\over (1+|y|^{\al_j})^2 }\Phi_{j,n} \, dy + o(1), $$
in view of
$$\lim \int_{\Om_{j,n } }  {2\al_j^2|y|^{\al_j-2}\over (1+|y|^{\al_j})^2 }\Phi_{j,n} \, dy = a_j\int_{\R^2 }  {2\al_j^2|y|^{\al_j-2}\over (1+|y|^{\al_j})^2 } Y_{0,j} \, dy=0$$
Also, by using \eqref{gama0} and expansion of $P_\e Z_{0j}$, we get that for $i<j$
\begin{equation*}
\begin{split}
&\gamma_0^{-1} \int_{\Om_\e} |x|^{\al_i-2} e^{w_i} \phi   PZ_{0j} =
\int_{\Om_{i,n } }  {2\al_i^2|y|^{\al_i-2}\over (1+|y|^{\al_i})^2 } \Phi_{i,n}\lf[\frac{2\gamma_0^{-1}}{1+(\frac{\de_i|y|}{\de_j})^{\al_j}} +\frac{1}{2\pi}\log(\de_i|y|) -H (\de_i|y|,0)\rg]\, dy \\
&\, +O(\rho^{\ti\sigma}|\log\rho|)
= 2\gamma_0^{-1} \int_{\Om_{i,n} }  \Pi_i\Phi_{i,n}(y)\,dy + \frac{1}{2\pi}\log\de_i\int_{\Om_{i,n} } \Pi_i\Phi_{i,n}(y)\, dy \\
&\, +\frac{1}{2\pi} \int_{\Om_{i,n} } \Pi_i\Phi_{i,n}(y)\log |y| \, dy + o(1),
\end{split}
\end{equation*}
and for $i>j$ we have that
\begin{equation*}
\begin{split}
&\gamma_0^{-1} \int_{\Om_\e} |x|^{\al_i-2} e^{w_i} \phi   PZ_{0j} = 
\int_{\Om_{i,n } }  \Pi_i \Phi_{i,n}\lf[\Big(\frac{\de_j}{\de_i}\Big)^{\al_j}\frac{2\gamma_0^{-1}}{(\frac{\de_j}{\de_i})^{\al_j} + |y|^{\al_j} } +\frac{1}{2\pi}\log(\de_i|y|) -H (\de_i|y|,0)\rg]\, dy \\
&\, +O(\rho^{\ti\sigma}|\log\rho|)
= \frac{1}{2\pi}\log\de_i\int_{\Om_{i,n} } \Pi_i\Phi_{i,n}(y)\, dy  +\frac{1}{2\pi} \int_{\Om_{i,n} } \Pi_i\Phi_{i,n}(y)\log |y| \, dy + o(1),
\end{split}
\end{equation*}
in view of $\big(\frac{\de_j}{\de_i}\big)^{\al_j} \gamma_0^{-1}=o(1)$. Furthermore, it follows that
\begin{equation*}
\begin{split}
\int_{\Om_\e} |x|^{\al_j-2} e^{w_j} G(x,0) \phi
=&\, - \frac{1}{2\pi}\log\de_j\int_{\Om_{j,n} } \Pi_j(y)\Phi_{j,n}(y)\, dy  - \frac{1}{2\pi} \int_{\Om_{j,n} } \Pi_j(y)\Phi_{j,n}(y)\log |y| \, dy + o(1).
\end{split}
\end{equation*}
From the choice of $\de_j$'s in \eqref{dei}, it follows that $\log\de_j=\frac1{\al_j}\big[\log d_j +\be_j\log\rho\big] $ and from previous expansions we deduce that for all $j=2,\dots,m-1$
{\small
\begin{equation*}
\begin{split}
&\gamma_0^{-1} \int_{\Om_\e} h PZ_{0j} =\sum_{i<j}\bigg[\Big(-\frac{\be_1+1}{2\pi(\al_1-2)} + \frac{\be_i}{2\pi \al_i}\Big)\log \rho \int_{\Om_{i,n } }  \Pi_i \Phi_{i,n}\, dy +\frac{1}{2\pi} \int_{\Om_{i,n } }  \Pi_i\Phi_{i,n}(y)  \log |y|\, dy \bigg]\\
&\,+\sum_{i>j}\bigg[ \frac{\be_i}{2\pi\al_i}\log\rho \int_{\Om_{i,n} } \Pi_i \Phi_{i,n}(y)\, dy  +\frac{1}{2\pi} \int_{\Om_{i,n} } \Pi_i\Phi_{i,n}(y)\log |y| \, dy\bigg] -\frac{\be_1+1}{4\pi(\al_1-2)} \log \rho \int_{\Om_{j,n } }  \Pi_j\Phi_{j,n}\, dy\\
&\, + \frac{\be_j}{2\pi \al_j}\log \rho \int_{\Om_{j,n } }  \Pi_j(y) \Phi_{j,n}\, dy +\frac{1}{2\pi} \int_{\Om_{j,n } }  \Pi_j(y)\Phi_{j,n}(y)  \log |y|\, dy+ o(1),
\end{split}
\end{equation*}
}
Since for all $i=1,\dots,m$
\begin{equation}\label{ipy0l}
\lim \int_{\Om_{i,n} } \Pi_i\Phi_{i,n}(y)\log |y| \, dy =a_i\int_{\R^2 } \Pi_iY_{0,i}(y)\log |y| \, dy=-4\pi a_i
\end{equation}
we obtain that for all $i=2,\dots,m-1$
{\small
\begin{equation}\label{j0}
\begin{split}
2\sum_{i=1}^m a_i= &\,\sum_{i<j}\bigg[\Big(-\frac{\be_1+1}{2\pi(\al_1-2)} + \frac{\be_i}{2\pi \al_i}\Big)\log \rho \int_{\Om_{i,n } }  \Pi_i \Phi_{i,n}\, dy \bigg] +\sum_{i>j}\bigg[ \frac{\be_i}{2\pi\al_i}\log\rho \int_{\Om_{i,n} } \Pi_i\Phi_{i,n}(y)\, dy  \bigg] \\
&\,+\Big( -\frac{\be_1+1}{4\pi(\al_1-2)} + \frac{\be_j}{2\pi \al_j}\Big) \log \rho \int_{\Om_{j,n } }  \Pi_j(y) \Phi_{j,n}\, dy + o(1).
\end{split}
\end{equation}
}
Choosing $j\in\{2,\dots,m-2\}$ we rewrite previous expression for $j+1$ as
{\small
\begin{equation}\label{j+1}
\begin{split}
&2\sum_{i=1}^m a_i= \sum_{i<j}\bigg[\Big(-\frac{\be_1+1}{2\pi(\al_1-2)} + \frac{\be_i}{2\pi \al_i}\Big)\log \rho \int_{\Om_{i,n } }  \Pi_i \Phi_{i,n} \bigg]+ \Big(-\frac{\be_1+1}{2\pi(\al_1-2)} + \frac{\be_j}{2\pi \al_j}\Big)\log \rho \int_{\Om_{j,n } }  \Pi_j \Phi_{j,n}  \\
&\,+\sum_{i>j+1}\bigg[ \frac{\be_i}{2\pi\al_i}\log\rho \int_{\Om_{i,n} } \Pi_i \Phi_{i,n}  \bigg] + \Big(-\frac{\be_1+1}{4\pi(\al_1-2)}  + \frac{\be_{j+1} }{2\pi \al_{j+1} } \Big) \log \rho \int_{\Om_{j+1,n } }  \Pi_{j+1} \Phi_{j+1,n} + o(1).
\end{split}
\end{equation}
}
Hence, by subtracting \eqref{j0} from \eqref{j+1} we obtain that
$$0= -\frac{\be_1+1}{4\pi(\al_1-2)} \log \rho \left[\int_{\Om_{j,n } }  \Pi_{j}(y) \Phi_{j,n}\, dy + \int_{\Om_{j+1,n } }  \Pi_{j+1}(y) \Phi_{j+1,n}\, dy\rg] + o(1),$$
so that, for all $j=2,\dots,m-2$
\begin{equation}\label{limj+1mj}
\lim \log \rho_n \left[\int_{\Om_{j,n } }  \Pi_{j}(y) \Phi_{j,n}\, dy + \int_{\Om_{j+1,n } }  \Pi_{j+1}(y) \Phi_{j+1,n}\, dy\rg]=0.
\end{equation}
Similarly as above it extends to $j=1$
\begin{equation*}
\begin{split}
2\sum_{i=1}^m a_i= &\, \sum_{i>1}\bigg[ \frac{\be_i}{2\pi\al_i}\log\rho \int_{\Om_{i,n} } \Pi_i\Phi_{i,n} dy  \bigg] +\Big( -\frac{\be_1+1}{4\pi(\al_1-2)} + \frac{\be_1}{2\pi \al_1}\Big) \log \rho \int_{\Om_{1,n } }  \Pi_1\Phi_{1,n} dy + o(1)
\end{split}
\end{equation*}
and
\begin{equation*}
\lim \log \rho_n \left[\int_{\Om_{1,n } }  \Pi_{1}(y) \Phi_{1,n}\, dy + \int_{\Om_{2,n } }  \Pi_{2}(y) \Phi_{2,n}\, dy\rg]=0;
\end{equation*}
and to $j=m$
\begin{equation*}\label{j}
\begin{split}
2\sum_{i=1}^m a_i= &\,\sum_{i<m}\bigg[\Big(-\frac{\be_1+1}{2\pi(\al_1-2)} + \frac{\be_i}{2\pi \al_i}\Big)\log \rho \int_{\Om_{i,n } }  \Pi_i \Phi_{i,n}\, dy \bigg]\\
&\,+\Big( -\frac{\be_1+1}{4\pi(\al_1-2)} + \frac{\be_m}{2\pi \al_m}\Big) \log \rho \int_{\Om_{m,n } }  \Pi_m(y) \Phi_{m,n}\, dy + o(1)
\end{split}
\end{equation*}
and
\begin{equation*}
\lim \log \rho_n \left[\int_{\Om_{m-1,n } }  \Pi_{m-1}(y) \Phi_{m-1,n}\, dy + \int_{\Om_{m,n } }  \Pi_{m}(y) \Phi_{m,n}\, dy\rg]=0.
\end{equation*}

Now, note that for $j=1$ we deduce that
{\small
\begin{equation*}
\begin{split}
&2\sum_{i=1}^m a_i= \lim \bigg[\sum_{i=2}^m \frac{\be_i}{2\pi\al_i}\log\rho \int_{\Om_{i,n} } \Pi_i\Phi_{i,n} dy +\Big( -\frac{\be_1+1}{4\pi(\al_1-2)} + \frac{\be_1}{2\pi \al_1}\Big) \log \rho \int_{\Om_{1,n } }  \Pi_1 \Phi_{1,n}dy\bigg]\\
&= \lim \bigg(-\frac{\be_1+1}{4\pi(\al_1-2)} + \sum_{i=1}^m(-1)^{i+1}\frac{\be_i}{2\pi \al_i}\bigg) \log \rho \int_{\Om_{1,n } }  \Pi_1 \Phi_{1,n}\, dy
=\bigg(-\frac{\be_1+1}{4\pi(\al_1-2)} + \sum_{i=1}^m(-1)^{i+1}\frac{\be_i}{2\pi \al_i}\bigg) \mu_1,
\end{split}
\end{equation*}
}
and there exists $\mu_1=\lim \log \rho_n \int_{\Om_{1,n } }  \Pi_1(y) \Phi_{1,n}\, dy$. In fact, by using \eqref{limj+1mj} we have that
{\small
\begin{equation*}
\begin{split}
\sum_{i=1}^m &\frac{\be_i}{2\pi\al_i}\log\rho \int_{\Om_{i,n} } \Pi_i\Phi_{i,n}(y)\, dy \\
=&\, \sum_{i=1}^{m-2} \frac{\be_i}{2\pi\al_i}\log\rho \int_{\Om_{i,n} } \Pi_i\Phi_{i,n}(y)\, dy  +\Big( \frac{\be_{m-1} }{2\pi\al_{m-1}} - \frac{\be_m}{2\pi \al_m}\Big) \log \rho \int_{\Om_{m-1,n } }  \Pi_{m-1}(y) \Phi_{m-1,n}\, dy\\
&\, +\frac{\be_m}{2\pi\al_m}\underbrace{\log\rho\bigg( \int_{\Om_{m-1,n } }  \Pi_{m-1}(y) \Phi_{m-1,n}\, dy + \int_{\Om_{m,n } }  \Pi_{m}(y) \Phi_{m,n}\, dy\bigg)}_{=o(1)}\\
= &\, \sum_{i=1}^{m-3} \frac{\be_i}{2\pi\al_i}\log\rho \int_{\Om_{i,n} } \Pi_i\Phi_{i,n}(y)\, dy  +\frac{1}{2\pi}\Big(\frac{\be_{m-2}}{\al_{m-2}}-\frac{\be_{m-1}}{\al_{m-1}} + \frac{\be_m}{\al_m}\Big)\log\rho \int_{\Om_{m-2,n}} \Pi_{m-2}(y)\Phi_{m-2}(y)\, dy\\
&+\Big( \frac{\be_{m-1} }{2\pi\al_{m-1}} - \frac{\be_m}{2\pi \al_m}\Big) \underbrace{ \log \rho \bigg(\int_{\Om_{m-2,n}} \Pi_{m-2}(y)\Phi_{m-2}(y)\, dy + \int_{\Om_{m-1,n } }  \Pi_{m-1}(y) \Phi_{m-1,n}\, dy\bigg)}_{=o(1)} +o(1)\\
& \vdots \\
=&\, \frac{1}{2\pi}\Big(\frac{\be_{1}}{\al_{1}}-\frac{\be_{2}}{\al_{2}} + \frac{\be_{3}}{\al_{3}} +\cdots + (-1)^{m+1} \frac{\be_{m}}{\al_{m}}\Big)\log\rho \int_{\Om_{1,n}} \Pi_{1}(y)\Phi_{1,n}(y)\, dy + o(1).
\end{split}
\end{equation*}
}
Thus, we get \eqref{eqmu1}. From the existence of $\mu_1$ and \eqref{limj+1mj}, it follows the existence of $\mu_2,\dots,\mu_m$ satisfying $\mu_j+\mu_{j+1}=0$ for all $j=1,\dots,m-1$. Therefore, it is readily checked that \eqref{eqmuj} it holds.
\end{proof}

\begin{claim}\label{claim3}
There hold that $a_{j}=0$ for all $j=1,\dots,m$.
\end{claim}
\begin{proof}[\dem] To this aim we will use as tests functions $P_\e w_j$ and the assumption on $h_n$, $|\log \rho_n|\ \|h_n\|_*=o(1)$. 
Hence, multiplying equation \eqref{eqPhi0} by $P_\e w_j$ and integrating by parts we obtain that
\begin{equation*}
\begin{split}
\int_{\Om_\e} &hPw_{j}= \int_{\Om_\e}\lap w_{j}\phi + \int_{\Om_\e} K\phi  P_\e w_{j}  
=- \int_{\Om_\e} \phi |x|^{\al_j-2}e^{w_j} +\int_{\Om_\e} K\phi  P_\e w_{j} =- \int_{\Om_\e} \phi |x|^{\al_j-2}e^{w_j} \\
&\, + \sum_{i=1}^m \int_{\Om_\e} |x|^{\al_i-2}e^{w_i} \phi \left[-2\log\big(\de_j^{\al_j}+|x|^{\al_j}\big) + (4\pi \al_j-\gamma_j) H(x,0) + \frac{\gamma_j}{2\pi}\log |x| + O(\rho^{\eta})\right]  
\end{split}
\end{equation*}
in view of $P_\e w_{j}=0$ on $\fr\Om_{\e_n}$ and $ \int_{\Om_\e}\lap\phi P_\e w_{j}=\int_{\Om_\e}\phi\; \lap P_\e w_{j}$. Now, estimating every integral term we find that
$$\int_{\Om_\e} hPw_{j}=O\lf(\, |\log\rho|\,\|h\|_p\rg)=o\lf( 1 \rg),$$
in view of $Pw_{j}=O(|\log\rho|)$ and $G(x,0)=O(|\log\e|)=O(|\log\rho|)$. Next, scaling we obtain that
$$\int_{\Om_\e} \phi |x|^{\al_j-2}e^{w_j}=\int_{\Om_{j,n } }  {2\al_j^2|y|^{\al_j-2}\over (1+|y|^{\al_j})^2 }\Phi_{j,n} =a_j\int_{\R^2 }  {2\al_j^2|y|^{\al_j-2}\over (1+|y|^{\al_j})^2 }Y_{0j} \, dy +o(1)=o(1).$$
For $i=j$ we have that
{\small
\begin{equation*}
\begin{split}
&\int_{\Om_\e}|x|^{\al_j-2}e^{w_j}\phi P_\e w_j= \int_{\Om_{j,n}} \Pi_j\Phi_j\Big[-2\al_j\log \de_j - 2\log(1+|y|^{\al_j} )+  (4\pi \al_j-\gamma_j) H(\de_j y,0)  + \frac{\gamma_j}{2\pi}\log |\de_j y| \\
&\, + O(\rho^{\eta}) \Big] = -2\al_j\log \de_j \int_{\Om_{j,n}} \Pi_j\Phi_j - 2\int_{\Om_{j,n}} \Pi_j\Phi_j \log(1+|y|^{\al_j} ) +  (4\pi \al_j-\gamma_j)\int_{\Om_{j,n}} \Pi_j\Phi_j H(\de_j y,0)  \\
&\,+ \frac{\gamma_j \log \de_j}{2\pi}\int_{\Om_{j,n}} \Pi_j\Phi_j +\frac{\gamma_j}{2\pi}\int_{\Om_{j,n}} \Pi_j\Phi_j \log | y| + o(1),
\end{split}
\end{equation*}
}
similarly, for $i>j$ we obtain that
{\small
\begin{equation*}
\begin{split}
&\int_{\Om_\e}|x|^{\al_i-2}e^{w_i}\phi P_\e w_j= \int_{\Om_{i,n}} \Pi_i\Phi_i\Big[-2\al_j\log \de_i - 2\log\Big(\Big(\frac{\de_j}{\de_i}\Big)^{\al_j} +|y|^{\al_j} \Big)+  (4\pi \al_j-\gamma_j) H(\de_i y,0)  \\
&\,+ \frac{\gamma_j}{2\pi}\log |\de_i y| + O(\rho^{\eta}) \Big] =-2\al_j\log \de_i \int_{\Om_{i,n}} \Pi_i\Phi_i - 2\int_{\Om_{i,n}} \Pi_i\Phi_i \log\Big(\Big(\frac{\de_j}{\de_i}\Big)^{\al_j} +|y|^{\al_j} \Big)   \\
&\,+  (4\pi \al_j-\gamma_j)\int_{\Om_{i,n}} \Pi_i\Phi_i H(\de_i y,0)+ \frac{\gamma_j \log \de_i}{2\pi}\int_{\Om_{i,n}} \Pi_i\Phi_i +\frac{\gamma_j}{2\pi}\int_{\Om_{i,n}} \Pi_i\Phi_i \log | y| + o(1),
\end{split}
\end{equation*}
}
and for $i<j$
{\small
\begin{equation*}
\begin{split}
&\int_{\Om_\e}|x|^{\al_i-2}e^{w_i}\phi P_\e w_j= \int_{\Om_{i,n}} \Pi_i\Phi_i\Big[-2\al_j\log \de_j - 2\log\Big(1+ \Big( \frac{\de_i|y|}{\de_j}\Big)^{\al_j}  \Big)+  (4\pi \al_j-\gamma_j) H(\de_i y,0)  + \frac{\gamma_j}{2\pi}\log |\de_i y| \\
&\, + O(\rho^{\eta}) \Big] =-2\al_j\log \de_j \int_{\Om_{i,n}} \Pi_i\Phi_i - 2\int_{\Om_{i,n}} \Pi_i\Phi_i \log\Big(1+ \Big( \frac{\de_i|y|}{\de_j}\Big)^{\al_j}  \Big) +  (4\pi \al_j-\gamma_j)\int_{\Om_{i,n}} \Pi_i\Phi_i H(\de_i y,0) \\
&\,+ \frac{\gamma_j \log \de_i}{2\pi}\int_{\Om_{i,n}} \Pi_i\Phi_i  +\frac{\gamma_j}{2\pi}\int_{\Om_{i,n}} \Pi_i\Phi_i \log | y| + o(1),
\end{split}
\end{equation*}
}
Now, by using the definition of $\gamma_j$'s in \eqref{repla1} and Lemma \ref{ewfxi} we find that $ \frac{\gamma_j}{2\pi} =2\be_j\frac{\al_1-2}{\be_1+1} + O\big(\frac{1}{|\log\rho|}\big)$ so that from the choice of $\de_j$'s we obtain that $ \frac{\gamma_j}{2\pi}\log\de_i =2\be_j\frac{\al_1-2}{\be_1+1} \frac{\be_i}{\al_i} \log\rho + O(1).$ Hence, from previous expansions we deduce that for all $j=2,\dots,m-1$
{\small
\begin{equation}\label{siphipw}
\begin{split}
o(1)&=\sum_{i<j}\Bigg[\Big(-2\al_j\log \de_j + \frac{\gamma_j \log \de_i}{2\pi} \Big) \int_{\Om_{i,n}} \Pi_i\Phi_i  + \frac{\gamma_j}{2\pi}\int_{\Om_{i,n}} \Pi_i\Phi_i \log | y| \bigg]\\
&\, + \Big(-2\al_j\log \de_j + \frac{\gamma_j \log \de_j}{2\pi} \Big) \int_{\Om_{j,n}} \Pi_j\Phi_j + \frac{\gamma_j}{2\pi}\int_{\Om_{j,n}} \Pi_j\Phi_j \log | y| - 2\int_{\Om_{j,n}} \Pi_j\Phi_j \log (1+|y|^{\al_j})  \\
&\,+\sum_{i>j}\Bigg[\Big(-2\al_j\log \de_i + \frac{\gamma_j \log \de_i}{2\pi} \Big) \int_{\Om_{i,n}} \Pi_i\Phi_i + \frac{\gamma_j}{2\pi}\int_{\Om_{i,n}} \Pi_i\Phi_i \log | y| \\
&\,  - 2\int_{\Om_{i,n}} \Pi_i \Phi_i \log\Big(\Big(\frac{\de_j}{\de_i}\Big)^{\al_j} + |y|^{\al_j}\Big) \bigg]  = - 2\int_{\Om_{j,n}} \Pi_j\Phi_j \log (1+|y|^{\al_j})   \\
&\,+ \sum_{i\le j} \Big(-2\be_j + 2\be_j\frac{\al_1-2}{\be_1+1} \frac{\be_i}{\al_i}\Big) \log\rho \int_{\Om_{i,n} } \Pi_i\Phi_i +\sum_{i=1}^m 2\be_j\frac{\al_1-2}{\be_1+1} \int_{\Om_{i,n} } \Pi_i\Phi_i \log|y|  + o(1) \\
&\,+\sum_{i>j} \bigg[\Big( -2\al_j \frac{\be_i}{\al_i} + 2\be_j \frac{\al_1-2}{\be_1+1} \frac{\be_i}{\al_i}\Big) \log\rho \int_{\Om_{i,n} }\Pi_i\Phi_i - 2\int_{\Om_{i,n } }\Pi_i\Phi_i \log \Big(\Big(\frac{\de_j}{\de_i}\Big)^{\al_j} + |y|^{\al_j}\Big)\bigg].
\end{split}
\end{equation}
}
Since for all $j=1,\dots,m$
\begin{equation}\label{ipy0l1}
\lim \int_{\Om_{j,n}} \Pi_j\Phi_j \log (1+|y|^{\al_j}) = a_j\int_{ \R^2} \Pi_j Y_{0j} \log (1+|y|^{\al_j})= -2\pi \al_j a_j,
\end{equation}
by using \eqref{muj} 
$$\lim\int_{\Om_{i,n } }\Pi_i\Phi_i \log \Big(\Big(\frac{\de_j}{\de_i}\Big)^{\al_j
}+ |y|^{\al_j}\Big)=\al_j a_i\int_{\R^2} \Pi_i Y_{0i}\log |y|=-4\pi\al_j a_i,$$
in view of $\frac{\de_j}{\de_i}=o(1)$, and \eqref{ipy0l}, as $n\to+\infty$ we get that
{\small
\begin{equation*}
\begin{split}
0=& 4\pi \al_j a_j + \sum_{i\le j} \Big(-2\be_j + 2\be_j\frac{\al_1-2}{\be_1+1} \frac{\be_i}{\al_i}\Big) \mu_i +\sum_{i=1}^m 2\be_j\frac{\al_1-2}{\be_1+1} (-4\pi a_i) \\
&\,+\sum_{i>j} \bigg[\Big( -2\al_j \frac{\be_i}{\al_i} + 2\be_j \frac{\al_1-2}{\be_1+1} \frac{\be_i}{\al_i}\Big) \mu_i- 2 (-4\pi \al_j a_i)\bigg]
\end{split}
\end{equation*}
}
so that
{\small
\begin{equation*}
\begin{split}
0=&\, 4\pi \al_j a_j + \sum_{i\le j} \Big(-2\be_j + 2\be_j\frac{\al_1-2}{\be_1+1} \frac{\be_i}{\al_i}\Big) \mu_i - 8\pi \be_j\frac{\al_1-2}{\be_1+1}\sum_{i=1}^m  a_i \\
&\,+\sum_{i>j} \Big( -2\al_j \frac{\be_i}{\al_i} + 2\be_j \frac{\al_1-2}{\be_1+1} \frac{\be_i}{\al_i}\Big) \mu_i + 8\pi \al_j \sum_{i>j}a_i .
\end{split}
\end{equation*}
}
It is readily checked that $ 8\pi \al_j \sum_{i>j}a_i =4\pi \al_j \sum_{i>j}a_i  + 4\pi \al_j \sum_{i\ge j+1}a_i, $ so we rewrite
{\small
\begin{equation*}
\begin{split}
&0= 4\pi \al_j \sum_{i\ge j} a_i + 4\pi \al_j \sum_{i\ge j+1}a_i   - 8\pi \be_j\frac{\al_1-2}{\be_1+1}\sum_{i=1}^m  a_i  -2\be_j \mu_1\sum_{i\le j} (-1)^{i+1}  -2\al_j \mu_1 \sum_{i>j}(-1)^{i+1} \frac{\be_i}{\al_i} \\
&\,+ 2\be_j \frac{\al_1-2}{\be_1+1} \mu_1\sum_{i=1}^m   (-1)^{i+1} \frac{\be_i}{\al_i}
= 4\pi \al_j \sum_{i\ge j} a_i + 4\pi \al_j \sum_{i\ge j+1}a_i   - 8\pi \be_j\frac{\al_1-2}{\be_1+1}\sum_{i=1}^m  a_i  \\
&\, + \mu_1\bigg[-2\be_j \frac{1+(-1)^{j+1}}{2}-2\al_j  \sum_{i>j}(-1)^{i+1} \frac{\be_i}{\al_i} + 2\be_j \frac{\al_1-2}{\be_1+1} \sum_{i=1}^m   (-1)^{i+1} \frac{\be_i}{\al_i} \bigg],
\end{split}
\end{equation*}
}
in view of \eqref{eqmuj}. For simplicity we shall denote for $j=1,\dots,m$
$$x_j=\sum_{i=j}^m a_i,\qquad A_j=-\frac{\be_j}{2\pi} \frac{1+(-1)^{j+1}}{2}-\frac{\al_j}{2\pi}  \sum_{i>j}(-1)^{i+1} \frac{\be_i}{\al_i} + \frac{\be_j}{2\pi} \frac{\al_1-2}{\be_1+1} B,$$
$$\text{and}\quad A_{m+1}=-\frac{1}{2\pi}\Big(- \frac{\be_1+1}{\al_1-2} + B\Big),\quad \text{with}\quad B=\sum_{i=1}^m   (-1)^{i+1} \frac{\be_i}{\al_i}.$$
Hence, taking into account \eqref{eqmu1} and dividing by $4\pi$, we have the following linear system in $m+1$ variables $X=[x_2\ x_3 \ \dots \ x_m \ x_1\ \mu_1]^T$ 
{\small
\begin{equation*}
\begin{split}
0=&\, \al_j x_j + \al_j x_{j+1}   -  2 \be_j\frac{\al_1-2}{\be_1+1}x_1 +  A_j \mu_1, \quad j=2,\dots,m-1 \\
0=&\,  \al_m x_m   - 2\be_j\frac{\al_1-2}{\be_1+1}x_1 + A_m \mu_m\\
0=&\,  \al_1 x_{2}  +\Big( \al_1-  2 \be_1\frac{\al_1-2}{\be_1+1}\Big)x_1 +  A_1 \mu_1 \\
0=&\, 2x_1+A_{m+1}\mu_1
\end{split}
\end{equation*}
}
which in matrix form becomes $A X=\vec 0$ with
{\small
$$A=\left[\begin{array}{ccccccc}
\al_2&\al_2&0&\cdots &0& -2\be_2\frac{\al_1-2}{\be_1+1}&A_2\\
0&\al_3&\al_3&\cdots &0& -2\be_3\frac{\al_1-2}{\be_1+1}&A_3\\
0&0&\al_4&\cdots & 0&-2\be_4\frac{\al_1-2}{\be_1+1}&A_4\\
\vdots&\vdots&\vdots&\ddots&\vdots&\vdots&\vdots\\
0&0&0 & \cdots&\al_m&-2\be_m\frac{\al_1-2}{\be_1+1}&A_m\\
\al_1&0&0&\cdots&0& \al_1-  2 \be_1\frac{\al_1-2}{\be_1+1} & A_1\\
0&0&0&\cdots&0&2&A_{m+1}
\end{array} \right]
= \left[\begin{array}{c}
f_1 \\
f_2\\
\vdots\\
f_m\\
f_{m+1}
\end{array} \right]
$$
}
Operating with rows $f_j$'s of $A$ in the following form
$$f_m \to f_m + (-1)^{j} \frac{\al_1}{\al_{j+1}} f_j,\qquad j=1,\dots,m-1$$ 
we are lead to consider the matrix
$$\left[\begin{array}{ccccccc}
\al_2&\al_2&0&\cdots &0& -2\be_2\frac{\al_1-2}{\be_1+1}&A_2\\
0&\al_3&\al_3&\cdots &0& -2\be_3\frac{\al_1-2}{\be_1+1}&A_3\\
0&0&\al_4&\cdots & 0&-2\be_4\frac{\al_1-2}{\be_1+1}&A_4\\
\vdots&\vdots&\vdots&\ddots&\vdots&\vdots&\vdots\\
0&0&0 & \cdots&\al_m&-2\be_m\frac{\al_1-2}{\be_1+1}&A_m\\
0&0&0&\cdots&0& \al_1-  2 \al_1B\frac{\al_1-2}{\be_1+1} & \frac{\al_1}{2\pi}\Big(-1 + \frac{\al_1-2}{\be_1+1}B\Big)\\
0&0&0&\cdots&0&2&A_{m+1}
\end{array} \right]
$$
since
$$\al_1-  2 \be_1\frac{\al_1-2}{\be_1+1}+\sum_{j=1}^{m-1}(-1)^{j}\frac{\al_1}{\al_{j+1}} (-2\be_{j+1})\frac{\al_1-2}{\be_1+1}=\al_1-  2 \al_1B\frac{\al_1-2}{\be_1+1}$$
and
{\small
\begin{equation*}
\begin{split}
&A_1+\sum_{j=1}^{m-1} (-1)^{j}\frac{\al_1}{\al_{j+1}} A_{j+1}=-\frac{\al_1}{2\pi}\bigg(\underbrace{\frac{\be_1}{\al_1} + \sum_{i=2}^m (-1)^{i+1} \frac{\be_i}{\al_i}}_{=B} \bigg)  + \underbrace{\frac{\be_1}{2\pi} \frac{\al_1-2}{\be_1+1} B + \frac{\al_1}{2\pi} \frac{\al_1-2}{\be_1+1} B \sum_{j=1}^{m-1} (-1)^{j}\frac{\be_{j+1}}{\al_{j+1} } }_{=\frac{\al_1}{2\pi} \frac{\al_1-2}{\be_1+1}B^2}\\
&\,+\frac{\al_1}{2\pi} \underbrace{\sum_{j=1}^{m-1} (-1)^{j} \bigg [-\frac{\be_{j+1}}{\al_{j+1} } \frac{1+(-1)^{j}}{2} +  \sum_{i>j+1}(-1)^{i} \frac{\be_i}{\al_i} \bigg]}_{=0}
= \frac{\al_1}{2\pi}  B\Big(-1+\frac{\al_1-2}{\be_1+1} B\Big).
\end{split}
\end{equation*}
}
Notice that $A$ is invertible, since
$$\text{det}\left[ \begin{array}{cc}
\al_1-  2 \al_1B\frac{\al_1-2}{\be_1+1} & \frac{\al_1}{2\pi}\Big(-1 + \frac{\al_1-2}{\be_1+1}B\Big)\\
2&A_{m+1}
\end{array} \right]=\frac{\al_1}{4\pi}\cdot\frac{\be_1+1}{\al_1-2}>0. $$
Therefore, $x_2=x_3= \dots = x_m= x_1= \mu_1=0$, which implies our claim.
\end{proof}

Now, by using Claims \ref{claim2}-\ref{claim3} we deduce that $\Phi_{j,n}$ converges to zero weakly in $H_{\alpha_j}(\R^2)$ and strongly in $L_{\al_j}(\R^2)$ as $n\to+\infty$. Thus, we arrive at a contradiction with \eqref{npsi} and it follows the priori estimate $\|\phi\|\le C|\log \rho | \, \|h\|_p$. It only remains to prove the solvability assertion. As usual, expressing \eqref{plco} in weak form, with the aid of Riesz's representation theorem and Fredholm's alternative, the existence of a unique solution follows from the a priori estimate
\eqref{est}. This finishes the proof of Proposition \ref{p2}.
\end{proof}


\begin{proof}[\bf Proof of the Proposition \ref{elle}]
Arguing as above, it is enough to prove the a-priori estimate \eqref{estphi}. Let us assume by contradiction the existence of $p>1$, sequences $\rho_n\to0$, functions $h_n\in L^p(\Om_{\e_n})$, $\phi_n\in W^{2,2}(\Om_{\e_n})$ such that
\begin{equation}\label{eqphin}
\mathcal L(\phi_n)=h_n\ \ \text{in}\ \ \Om_{\e_n},\ \ \phi_n=0\ \ \text{on}\ \ \fr\Om_{\e_n},
\end{equation}
with $\|\phi_n\|=1$ and $|\log\rho_n|\ \|h_n\|_p=o(1)$ as $n\to+\infty$. We will omit the subscript $n$ in $\de_{i,n} =\de_i$. Recall that $\de_i^{\al_i}=d_{i,n}\rho_n^{\be_i}$. 
Now, define $\Phi_{i,n}(y):=\phi_{n}(\de_i y)$ for $y\in\Om_{i,n}:=\de_i^{-1}\Om_{\e_n}$, $i=1,\dots,m$ and extend $\phi_n=0$ in $\R^2\sm\Om_{\e_n}$. Arguing in the same way as in \cite[Claim 1, section 4]{EFP}, we can prove the following fact. We provide a sketch of the proof.

\begin{claim}\label{5.4}
The sequence $\{\Phi_{i,n}\}_n$ converges (up to a subsequence) to $\Phi_i^*$ weakly in $H_{\al_i}(\R^2)$ and strongly in $L_{\al_i}(\R^2)$.
\end{claim}

\begin{proof}[\dem]
First, we shall show that the sequence $\{\Phi_{i,n}\}_n$ is bounded in $H_{\al_i}(\R^2)$. Notice that \linebreak $\|\Phi_{i,n}\|_{H^1_0(\Omega_{i,n})}=1$ for $i=1,\dots,m$. Thus, we want to prove that there is a constant $M>0$ such for all $n$ (up to a subsequence) $\|\Phi_{i,n}\|_{L_{\al_i}}^2 \le M.$ Notice that for any $i\in\{1,\dots,m\}$ we find that
 in $\Om_{i,n}$
\begin{equation}\label{eqPhi1}
\begin{split}
\lap \Phi_{i,n}&+\de_i^2K_0(\de_i y) \lf(\Phi_{i,n} +\ti c_{0,n} \rg)+\de_i^2K_1(\de_i y) \lf(\Phi_{i,n} +\ti c_{1,n} \rg)=\de_i^2h_n(\de_i y), 
\end{split}
\end{equation}
where for simplicity we denote $\ti c_{j,n}=\ti c_{j}(\phi_n)$, with $\ti c_j$ given by \eqref{ctjphi}. Furthermore, it follows that $\Phi_{i,n}\to\Phi_i^*$ weakly in $H^1_0(\Omega_{i,n})$ and strongly in $L^p(K)$ for any $K$   compact sets in $\mathbb R^2$.
Now, let $\chi$ a smooth function with compact support in $\mathbb R^2.$ We multiply \eqref{eqPhi1} by $\chi$, integrate by parts and we obtain that 
\begin{equation*}
\begin{split}
-\int_{\Omega_{i,n} }&\nabla \Phi_{i,n}\nabla \chi+\int_{\Omega_{i,n} } \bigg[\ds{2\al_i^2|y|^{\al_i-2}\over (1+|y|^{\al_i})^2}+o(1)\bigg]\Phi_{i,n}\chi+\ti c_{0,n} \int_{\Omega_{i,n} }\bigg[\ds{2\al_i^2|y|^{\al_i-2}\over (1+|y|^{\al_i})^2}+o(1)\bigg]\chi\\
&+\int_{\Om_{i,n}} o(1) \Phi_{i,n}\chi  + \ti c_{1,n} \int_{\Om_{i,n}} o(1) \chi = \int_{\Omega_{i,n} }\de_i^2h_n(\de_i y) \chi 
\end{split}
\end{equation*}
for $i$ odd and 
\begin{equation*}
\begin{split}
-\int_{\Omega_{i,n} }&\nabla \Phi_{i,n}\nabla \chi +\int_{\Om_{i,n}} o(1)\Phi_{i,n}\chi +\ti c_{0,n}  \int_{\Om_{i,n}} o(1) \chi +\int_{\Omega_{i,n} } \bigg[\ds{2\al_i^2|y|^{\al_i-2}\over (1+|y|^{\al_i})^2}+o(1)\bigg]\Phi_{i,n}\chi \\
& + \ti c_{1,n} \int_{\Omega_{i,n} }\bigg[\ds{2\al_i^2|y|^{\al_i-2}\over (1+|y|^{\al_i})^2}+o(1)\bigg]\chi  
= \int_{\Omega_{i,n} }\de_i^2h_n(\de_i y) \chi 
\end{split}
\end{equation*}
for $i$ even, in view of
{\small
\begin{equation}\label{dik0}
\de_i^2K_0(\de_i y)=
\begin{cases}
\ds{2\al_i^2|y|^{\al_i-2}\over (1+|y|^{\al_i})^2}+O\bigg(\sum_{j<i \atop j \text{ odd} }\Big({\de_j\over \de_i}\Big)^{\al_j} + \sum_{i<j\atop j \text{ odd} } \Big({\de_i\over \de_j}\Big)^{\al_j}\bigg)& \text{if $i$ is odd}\\[0.6cm]
\ds O\bigg(\sum_{j<i \atop j \text{ odd} }\Big({\de_j\over \de_i}\Big)^{\al_j} + \sum_{i<j\atop j \text{ odd} } \Big({\de_i\over \de_j}\Big)^{\al_j}\bigg)&\text{ if $i$ is even}
\end{cases}
\end{equation}
}
and
{\small
\begin{equation}\label{dik1}
\de_i^2K_1(\de_i y)=
\begin{cases}
\ds O\bigg(\sum_{j<i \atop j \text{ even} }\Big({\de_j\over \de_i}\Big)^{\al_j} + \sum_{i<j\atop j \text{ even} } \Big({\de_i\over \de_j}\Big)^{\al_j}\bigg)&\text{ if $i$ is odd}\\[0.6cm]
\ds{2\al_i^2|y|^{\al_i-2}\over (1+|y|^{\al_i})^2}+O\bigg(\sum_{j<i \atop j \text{ even} }\Big({\de_j\over \de_i}\Big)^{\al_j} + \sum_{i<j\atop j \text{ even} } \Big({\de_i\over \de_j}\Big)^{\al_j}\bigg)& \text{if $i$ is even}
\end{cases}
\end{equation}
}
\hspace{-0.15cm}uniformly on compact subsets of $\R^2\sm \{0\}$. We re-write the system for $\ti c_{0,n}$ and $\ti c_{1,n}$ as a diagonal dominant one as $n\to+\infty$
{\small
\begin{equation*}
\begin{split}
\ti c_{0,n} \int_{\Omega_{i,n} }\bigg[{2\al_i^2|y|^{\al_i-2}\over (1+|y|^{\al_i})^2} +o(1)\bigg] \chi+o(1)\ti c_{1,n} =&\,O(1)\\[0.4cm]
o(1)\ \ti c_{0,n}+\ti c_{1,n}\int_{\Omega_{j,n} } \bigg[{2\al_j^2|y|^{\al_j-2}\over (1+|y|^{\al_j})^2} +o(1)\bigg]  \chi=&\, O(1),
\end{split}
\end{equation*}
}
\hspace{-0.15cm}choosing $i $ odd and $j$ even. Thus, if we choose $\chi$ so that $ \int_{\mathbb R^2}{2\al_k^2|y|^{\al_k-2}\over (1+|y|^{\al_k})^2}  \chi\ dy\not=0$ for $k=i,j$ then we obtain that $ \ti c_{i,n}=O(1)$,  for $i=0,1$. Now, we multiply \eqref{eqPhi1} by $\Phi_{i,n}$ for any $i\in\{1,\dots,m\}$, integrate by parts and we get 
\begin{equation}\label{snpsi}
\sum_{i=1}^{m} 2\al_i^2\| \Phi_{i,n}\|_{L_{\al_i}}^2
=\,1 + \la_0 (\ti c_{0,n})^2 + \la_1\tau^2 (\ti c_{1,n})^2 +o(1).
\end{equation}
Therefore, the sequence $\{\Phi_{i,n}\}_n$ is bounded in $H_{\al_i}(\R^2)$, so that there is a subsequence $\{\Phi_{i,n}\}_n$ and functions $\Phi_i^*$, $i=1,\dots,m$ such that $\{\Phi_{i,n}\}_n$ converges to $\Phi_i^*$ weakly in $H_{\al_i}(\R^2)$ and strongly in $L_{\al_i}(\R^2)$. That proves our claim.
\end{proof}

Define the sequences $\psi_{i,n}=\phi_n+\ti c_{i,n}$, $i=0,1$. Notice that clearly
\begin{equation}\label{eqpsi}
\lap \psi_{i,n}+K_0 \psi_{1,n}+K_1\psi_{2,n}=h_n\quad\text{in}\ \ \oen,\qquad i=0,1.
\end{equation}
Now, define $\Psi_{i,j,n}(y):=\psi_{i,n}(\de_j y)$ for $y\in \Om_{j,n}$, $i=0,1$ and $j=1,\dots,m$. Note that $\Psi_{i,j,n}=\Phi_{j,n}+\ti c_{i,n}$. Thus, we can prove the following fact.

\begin{claim}\label{claim2mff} $\Psi_{0,j,n} \to a_jY_{0j}$ for $j$ odd  and $\Psi_{1,j,n} \to a_jY_{0j}$ for $j$ even, weakly in $H_{\al_j}(\R^2)$ and strongly in $L_{\al_j}(\R^2)$ as $n\to+\infty$ for some constant $a_{j}\in\R$, $j=1,\dots,m$.
\end{claim}

\begin{proof}[\dem]
From the previous computations, it is clear that in $\Om_{j,n}$
$$\lap \Psi_{0,j,n}+\de_j^2K_0(\de_j y ) \Psi_{0,j,n}+\de_j^2K_1( \de_j y ) \lf(\Psi_{0,j,n}-\ti c_{0,n}+\ti c_{1,n}\rg)=\de_j^2h_n(\de_j y)$$
and
$$\lap \Psi_{1,j,n}+\de_j^2K_0( \de_j y ) \lf(\Psi_{1,n}-\ti c_{1,n}+\ti c_{0,n}\rg) + \de_j^2K_1( \de_j y )  \Psi_{1,j,n}=\de_j^2h_n( \de_j y).$$
Furthermore, $\{\ti c_{i,n}\}$ is a bounded sequence in $\R$, so it follows that $\{\Psi_{i,j,n}\}_n$ is bounded in $H_{\al_j}(\R^2)$ for $i=1,2$ and $j=1,\dots,m$. Also, we have that
$$\int_{\Om_{j,n}} (\de_j^2|h_n(\de_j y)|)^p\, dy=\de_j^{2p-2}\int_{\Om_{\e_n} }|h_n(x)|^p\, dx=\de_i^{2p-2}\|h_n\|_p^p=o(1).$$
Therefore, taking into account \eqref{dik0} and \eqref{dik1} we deduce that $\Psi_{i,j,n}\to\Psi_j^*$ as $n\to+\infty$ with $i=1$ if $j=1,\dots,m_1$ and $i=2$ if $j=m_1+1,\dots,m$, where $\Psi^*_j$ is a solution to 
$$\lap\Psi+{2\al_j^2|y|^{\al_j-2}\over(1+|y|^{\al_j})^2}\Psi=0,\qquad j=1,\dots,m,\qquad  \text{in $\R^2\sm\{0\}$}.$$
It is standard that $\Psi_j^*$, $j=1,\dots,m$, extends to a solution in the whole $\R^2$. Hence, by using symmetry assumptions if necessary, we get that $\Psi^*_j=a_jY_{0j}$ for some constant $a_{j}\in\R$, $j=1,\dots,m$.
\end{proof}

Denote $\ti c_i=\displaystyle \lim_{n\to+\infty} \ti c_{i,n}$ for $i=0,1$, up to a subsequence if necessary. Hence, we get that
\begin{equation}\label{cpsi}
\Phi_{j,n} \to a_jY_{0j}-\ti c_0, \ \text{ for $j$ odd}\quad\text{and}\quad \Phi_{j,n} \to a_jY_{0j}-\ti c_1, \  \text{ for $j$ even,}
\end{equation}
weakly in $H_{\al_j}(\R^2)$ and strongly in $L_{\al_j}(\R^2)$, since $\Phi_{j,n}=\Psi_{i,j,n}-\ti c_{i,n}$.

\begin{claim}
For all $j=1,\dots,m$, there exist the limit, up to subsequences,
\begin{equation}\label{mujmff}
\mu_j=\lim_{n\to+\infty} \log\rho_n \int_{\Om_{j,n} }\Pi_j(y)\Psi_{l,j,n}(y)\, dy=0,
\end{equation}
where $l=0$ for $j$ odd and $l=1$ for $j$ even. Furthermore, it holds
\begin{equation}\label{eqmu1mff}
\begin{split}
\sum_{i=1}^m a_i=0.
\end{split}
\end{equation}
\end{claim}

\begin{proof}[\dem] To this aim we use the test function $P_\e Z_{0j}$, where $Z_{0j} $ as in Claim \ref{claim3}. Thus,  from the assumption on $h_n$, $|\log \rho_n|\ \|h_n\|_*=o(1)$, we get \eqref{mujmff}-\eqref{eqmu1mff}.

\medskip \noindent Assume that either $l=0$ for all $j$ odd or $l=1$ for all  $j$ even. Multiplying equation \eqref{eqphin} by $P_\e Z_{0j}$ and integrating by parts we obtain that
\begin{equation*}
\begin{split}
\int_{\Om_\e} hP_\e Z_{0j}=&\ \int_{\Om_\e}\lap Z_{0j} \lf[\psi_l - \ti c_{l,n}\rg] + \int_{\Om_\e} \lf[ K_0\psi_0+ K_1\psi_1\rg]  P_\e Z_{0j}  ,
\end{split}
\end{equation*}
in view of $P_\e Z_{0j}=0$ and $\psi_l=\ti c_{l,n}$ on $\fr\Om_\e$ and
\begin{equation*}
\begin{split}
\int_{\Om_\e}\lap\psi_l P_\e Z_{0j}&=\int_{\Om_\e}\psi_l\lap P_\e Z_{0j}-\psi_l\bigg|_{\fr\Om_\e}\int_{\Om_e}\lap PZ_{0j}=\int_{\Om_\e} \lap Z_{0j} \lf[\psi_l- \ti c_{l,n}\rg].
\end{split}
\end{equation*}
Furthermore, we have that
{\small
\begin{equation*}
\begin{split}
&\int_{\Om_\e} hPZ_{0j} 
= - \int_{\Om_\e} \lf[\psi_l - \ti c_{l,n}\rg]  |x|^{\al_j-2}e^{w_j}Z_{0j}  +\int_{\Om_\e} \lf[ K_0\psi_0+ K_1\psi_1\rg]  P_\e Z_{0j} 
 \\
&= \int_{\Om_\e} |x|^{\al_j-2} e^{w_j}\psi_l  \(PZ_{0j} - Z_{0j}\),+\ti c_{l,n} \int_{\Om_\e} |x |^{\al_j-2} e^{w_j}  Z_{0j} + \int_{\Om_\e} \(K_0\psi_0+K_1\psi_1 - |x |^{\al_j-2}e^{w_j}\psi_l \)  P_\e Z_{0j} .
\end{split}
\end{equation*}}
Multiplying by $\gamma_0^{-1}$ and arguing in the same way as in Claim \ref{claimmu} (replacing $\Psi_{l,i,n}$ by $\Phi_{i,n}$ with $l=0$ for $j$ odd and $l=1$ for $j$ even in \eqref{j0} and \eqref{j+1}) we deduce 
\begin{equation*}
\begin{split}
2\sum_{i=1}^m a_i=\bigg(-\frac{\be_1+1}{4\pi(\al_1-2)} + \sum_{i=1}^m(-1)^{i+1}\frac{\be_i}{2\pi \al_i}\bigg) \mu_1,
\end{split}
\end{equation*}
and $\mu_j=(-1)^{j-1} \mu_1$ for all $j=1,\dots,m-1$, in view of
$$\gamma_0^{-1} \int_{\Om_\e} hPZ_{0j}=O\lf(\, |\log\rho|\,\|h\|_p\rg)=o\lf( 1 \rg)\quad\text{and} \quad \gamma_0^{-1}\int_{\Om_\e} |x |^{\al_j-2} e^{w_j}  Z_{0j}=O(\rho^{\ti\sigma}|\log\rho|).$$
On the other hand, it is readily checked that
\begin{equation*}
\begin{split}
\int_{\Om_\e} K_0&=\sum_{k=1\atop k\text{ odd} }^{m}\int_{\Om_\e} |x|^{\al_k-2}e^{w_k}\ dx 
=\sum_{k=1 \atop k\text{ odd}}^{m }\lf[4\pi\al_k + O(\de_k^{\al_k}) +O\Big({\e^2\over\de_k^2}\Big)\rg]
=\la_0+O(\rho^{\ti\sigma})
\end{split}
\end{equation*}
and similarly $ \int_{\Om_\e} K_1=\la_2 \tau^2+O(\rho^{\ti\sigma})$ for some $\ti\sigma>0$, so that for $\psi_0$ and $\psi_1$ we have that
$$\int_{\Om_\e} K_i\psi_i=\lf(1-{1\over \la_i\tau^{2i} }\int_{\Om_\e} K_i\rg)\int_{\Om_\e} K_i\phi=O(\rho^{\ti\sigma})\int_{\Om_\e} K_i\phi=O(\rho^{\ti\sigma}).$$
Also, we get that 
$$\int_{\Om_\e} K_0\psi_0 =\sum_{j=1\atop j\text{ odd}}^{m} \int_{\Om_\e} |x|^{\al_j-2} e^{w_j} \psi_i = \sum_{j=1\atop j\text{ odd}}^{m_1}\int_{ \Om_{j,n} } {2\al_j^2|y|^{\al_j-2}\over (1+|y|^{\al_j})^2 }  \Psi_{0,j,n}(y)\, dy$$
and 
$$\int_{\Om_\e} K_1\psi_1  = \sum_{j=1\atop j\text{ even}}^{m}\int_{ \Om_{j,n} } {2\al_j^2|y|^{\al_j-2}\over (1+|y|^{\al_j})^2 } \Psi_{1,j,n}(y)\, dy.$$
Hence, it follows that
$$\lim_{n\to +\infty} \log\rho_n \int_{\Om_{\e_n}} K_0\psi_0=0=\sum_{j=1\atop j\text{ odd}}^{m}\mu_j =\mu_1\sum_{j=1\atop j\text{ odd}}^{m}1,$$
so that, $\mu_j=0$ for all $j=1,\dots,m$.
\end{proof}

\begin{claim}
There hold that
\begin{equation}\label{tclaj2sai}
0=\ti c_{l} +   a_j  + 2 \sum_{i>j}a_i,
\end{equation}
where $l=0$ for $j$ odd and $l=1$ for $j$ even.
\end{claim}

\begin{proof}[\dem] To this aim we will use as tests functions $P_\e w_j$. 
Hence, multiplying equation \eqref{eqPhi0} by $P_\e w_j$ and integrating by parts as in the previous Claim we obtain that
\begin{equation*}
\begin{split}
\int_{\Om_\e} hPw_{j} 
=&\,- \int_{\Om_\e} [\psi_l - \ti c_{l,n} ] |x|^{\al_j-2}e^{w_j} +\int_{\Om_\e} [K_0\psi_0+K_1\psi_j]  P_\e w_{j}\\
=&\,- \int_{\Om_\e} \psi_l |x|^{\al_j-2}e^{w_j} + \ti c_{l,n}\int_{\Om_\e} |x|^{\al_j-2}e^{w_j} + \int_{\Om_\e} [K_0\psi_0+K_1\psi_1]  P_\e w_{j}
\end{split}
\end{equation*}
in view of $P_\e w_{j}=0$ on $\fr\Om_{\e_n}$ and $ \int_{\Om_\e}\lap\psi_l P_\e w_{j}=\int_{\Om_\e}\psi_l\; \lap P_\e w_{j}$. Arguing in the same way as in Claim \ref{claim3} (replacing $\Psi_{l,i,n}$ by $\Phi_{i,n}$ with $l=0$ for $j$ odd and $l=1$ for $j$ even in \eqref{siphipw} and \eqref{ipy0l1}) we find that
\begin{equation*}
\begin{split}
0=&\,4\pi \al_j\ti c_{l} + 4\pi \al_j a_j + \sum_{i\le j} \Big(-2\be_j + 2\be_j\frac{\al_1-2}{\be_1+1} \frac{\be_i}{\al_i}\Big) \mu_i - 8\pi \be_j\frac{\al_1-2}{\be_1+1}\sum_{i=1}^m  a_i \\
&\,+\sum_{i>j} \Big( -2\al_j \frac{\be_i}{\al_i} + 2\be_j \frac{\al_1-2}{\be_1+1} \frac{\be_i}{\al_i}\Big) \mu_i + 8\pi \al_j \sum_{i>j}a_i,
\end{split}
\end{equation*}
in view of
$$\int_{\Om_\e} hPw_{j}=O\lf(\, |\log\rho|\,\|h\|_p\rg)=o\lf( 1 \rg),$$
$$\int_{\Om_\e} \psi_l |x|^{\al_j-2}e^{w_j}=\int_{\Om_{j,n } }  {2\al_j^2|y|^{\al_j-2}\over (1+|y|^{\al_j})^2 }\Psi_{l,j,n} =a_j\int_{\R^2 }  {2\al_j^2|y|^{\al_j-2}\over (1+|y|^{\al_j})^2 }Y_{0j} \, dy +o(1)=o(1)$$
and
$$\int_{\Om_\e} |x|^{\al_j-2}e^{w_j}=\int_{\Om_{j,n } }  {2\al_j^2|y|^{\al_j-2}\over (1+|y|^{\al_j})^2 } \, dy =\int_{\R^2 }  {2\al_j^2|y|^{\al_j-2}\over (1+|y|^{\al_j})^2 } \, dy +o(1)=4\pi \al_j + o(1).$$
By using \eqref{mujmff} and \eqref{eqmu1mff} we obtain that  \eqref{tclaj2sai} for $l=0$  with $j$ odd and $l=1$ with $j$ even.
\end{proof}

For the next step, consider the function $\eta_j(x)={4\over 3}\log(\de_j^{\al_j}+|x|^{\al_j}){\de_j^{\al_j}-|x|^{\al_j}\over \de_j^{\al_j}+|x|^{\al_j} }+ {8\over 3}{\de_j^{\al_j} \over \de_j^{\al_j}+|x|^{\al_j} },$ for any $j\in\{1,\dots, m\}$ 
 so that $\lap\eta_j+|x|^{\al_j-2}e^{w_j}\eta_j=|x|^{\al_j-2}e^{w_j}Z_{0j}.$ Notice that 
$$P_\e \eta_{j}=\eta_{j}+{8\pi\over 3}\al_j H(x,0)-  \ti\gamma_{j} G(x,0)+O(\rho^{\ti\sigma})\quad\text{uniformly in $\Om_\e$ for some $\ti\sigma>0$},$$
by using similar arguments as to obtain expansion in Lemma \ref{ewfxi}, as shown in \cite[Lemma 4.1]{EFP} with $m=1$ and $\xi=0$, where the coefficients $\ti\gamma_{i}$'s, $i=1,\dots,m$, are given by
\begin{equation}\label{gamatij}
\ti\gamma_{i}\lf[-{1\over 2\pi}\log\e+H(0,0)\rg]=
{4\over 3} \al_i\log\de_i +{8\over 3} + {8\pi\over 3}\al_i H(0,0)
\end{equation}
From \eqref{dei} it follows that $ \ti\gamma_{i}=
-\frac{8\pi(\al_1-2)\be_i}{3(\be_1+1)} + O\big({1\over |\log\rho|}\big),
$.

\begin{claim}\label{claim32}
There hold that $a_{j} +4\sum_{i>j} a_i+ 2\ti c_l=0$ with $l=0$ for all $j$ odd and $l=1$ for all $j$ even. Consequently, combining with \eqref{tclaj2sai} it follows that $a_i=0$ for all $i=1,\dots,m$.
\end{claim}

\begin{proof}[\dem]
We use the following test function $P_\e \eta_j$. Thus,  from the assumption on $h_n$, $|\log \rho_n|\ \|h_n\|_*=o(1)$, we get the above relation between $a_j$ and $\ti c_i$ either for $l=0$ and all $j$ odd or for $l=1$ and all $j$ even. Assume that $l=0$ for all $j$ odd or $l=1$ for all  $j$ even. Multiplying equation \eqref{eqphin} by $P_\e \eta_j$ and again integrating by parts we obtain that
{\small
\begin{equation*}
\begin{split}
&\int_{\Om_\e} hP\eta_{j} \!
=\! \int_{\Om_\e} \! \lf[\psi_l - \ti c_{l,n}\rg] \Big[|x|^{\al_j-2}e^{w_j}Z_{0j} - |x|^{\al_j-2} e^{w_j}\eta_j \Big] \! +\!  \int_{\Om_\e} \! \lf[
K_0\psi_0+ K_1\psi_1\rg]  P_\e \eta_{j} 
\! =\! \int_{\Om_\e} \! \psi_i |x |^{\al_j-2}e^{w_j}Z_{0j} \\
&+\int_{\Om_\e} |x|^{\al_j-2} e^{w_j}\psi_l \( P\eta_j-\eta_j\)  +\int_{\Om_\e} \(K_0\psi_0+K_1\psi_1 - |x |^{\al_j-2}e^{w_j}\psi_l \)  P_\e \eta_{j} -\ti c_{l,n} \int_{\Om_\e} |x |^{\al_j-2} e^{w_j} \( Z_{0j}-\eta_j \),  
\end{split}
\end{equation*}}
Now, estimating every integral term we find that $ \int_{\Om_\e} hP\eta_{j}=O\lf(\, |\log\rho|\,\|h\|_p\rg)=o\lf( 1 \rg)$ for all $j=1,\dots, m$, in view of $P\eta_{j}=O(|\log\rho|)$ and $G(x,0)=O(|\log\rho|)$. Next, by scaling we obtain that either for $l=0$ and all $j$ odd or $l=1$ and all $j$ even, it holds
$$\int_{\Om_\e} \psi_l |x|^{\al_j-2}e^{w_j}Z_{0j}=\int_{\Om_{j,n } }  {2\al_j^2|y|^{\al_j-2}\over (1+|y|^{\al_j})^2 }\Psi_{l,j,n}Y_{0j}\, dy =a_j\int_{\R^2 }  {2\al_j^2|y|^{\al_j-2}\over (1+|y|^{\al_j})^2 }Y_{0j}^2 \, dy +o(1).$$
Note that
$$\int_{\R^2} {2\al_j^2|y|^{\al_j-2}\over (1+|y|^{\al_j})^2}Y_{0j}^2=\int_{\R^2} {2\al_j^2|y|^{\al_j-2}\over (1+|y|^{\al_j})^2}\({1-|y|^{\al_j} \over 1+|y|^{\al_j} }\)^2dy={4\pi\over 3}\al_j$$
and 
$$\int_{\R^2} {2\al_j^2|y|^{\al_j-2}\over (1+|y|^{\al_j})^2}Y_{0j}\log|y|=\int_{\R^2} {2\al_j^2|y|^{\al_j-2}\over (1+|y|^{\al_j})^2}\ {1-|y|^{\al_j} \over 1+|y|^{\al_j} }\ \log|y|\, dy= -4\pi .$$
Also, by using \eqref{gamatij} we get that
\begin{equation*}
\begin{split}
& \int_{\Om_\e} |x|^{\al_j-2} e^{w_j} \psi_l \( P_\e \eta_j-\eta_j \)= \int_{\Om_\e} |x|^{\al_j-2} e^{w_j} \psi_l\bigg[ \( {8\pi\over 3} \al_j - \ti\gamma_{j} \) H(x,0)+ {1\over 2\pi} \ti\gamma_{j} \log|x | +O(\rho^{\ti\sigma}) \bigg] \\
&= \( {8\pi\over 3} \al_j - \ti\gamma_{j} \) \int_{\Om_{j,n } }  {2\al_j^2|y|^{\al_j-2}\over (1+|y|^{\al_j})^2 }\Psi_{l,j,n} H(\de_jy,0)\, dy + {\ti\gamma_{j}\over 2\pi}\log\de_j \int_{\Om_{j,n } }  {2\al_j^2|y|^{\al_j-2}\over (1+|y|^{\al_j})^2 }\Psi_{l,j,n}  \, dy \\
&\, + {\ti\gamma_{j}\over 2\pi} \int_{\Om_{j,n } }  {2\al_j^2|y|^{\al_j-2}\over (1+|y|^{\al_j})^2 }\Psi_{l,j,n} \log |y|  dy +O(\rho^{\ti\sigma}) = -\frac{4(\al_1-2)\be_j}{3(\be_!+1)} a_j \int_{\R^2 }  {2\al_j^2|y|^{\al_j-2}\over (1+|y|^{\al_j})^2 }Y_{0j} \log |y| dy \! + \! o(1),
\end{split}
\end{equation*}
in view of \eqref{mujmff}. Furthermore, using \eqref{K12} we have that

Notice that
$$\int_{\Om_{j,n} } {2\al_j^2|y|^{\al_j-2}\over (1+|y|^{\al_j})^2 }  \Psi_{i,j,n}(y)\, dy=a_j\int_{\R^2 } {2\al_j^2|y|^{\al_j-2}\over (1+|y|^{\al_j})^2 }  Y_{0j}(y)\, dy+o(1)=o(1),$$
since
$$\int_{\R^2 } {2\al_j^2|y|^{\al_j-2}\over (1+|y|^{\al_j})^2 }  Y_{0j}(y)\, dy=\int_{\R^2 } {2\al_j^2|y|^{\al_j-2}\over (1+|y|^{\al_j})^2 } \cdot {1-|y|^{\al_j} \over 1+|y|^{\al_j} }\, dy=0.$$
If $i<j$ then
\begin{equation*}
\begin{split}
P_\e\eta_j(\de_i y)=&\, \bigg[\frac43 \al_j\log\de_j +\frac43 \log\Big(1+\Big(\frac{\de_i|y|}{\de_j}\Big)^{\al_j}\Big)\bigg] \frac{1-\big(\frac{\de_i|y|}{\de_j}\big)^{\al_j} }{1 +\big(\frac{\de_i|y|}{\de_j}\big)^{\al_j} } + \frac83\cdot \frac{1}{1+\big(\frac{\de_i|y|}{\de_j}\big)^{\al_j} }\\
&\, +\frac{\ti\gamma_j}{2\pi } \log(\de_i |y|) + \Big(\frac{8\pi}{3}\al_j - \ti\gamma_j\Big)H(\de_i y ,0) + O(\rho^{\ti\sigma})
\end{split}
\end{equation*}
and
\begin{equation*}
\begin{split}
\int_{\Om_{\e} } |x|^{\al_i-2} e^{w_i} \psi_l P_\e\eta_j=&\,\int_{\Om_{i,n} }\Pi_i \Psi_{l,i,n} P_\e\eta_j(\de_i y)= \frac{\ti\gamma_j}{2\pi}\int_{\Om_{i,n}}\Pi_i\Psi_{l,i,n}(y)\log|y| + o(1)\\
=&\,\frac{16\pi(\al_1-2)\be_j}{3(\be_1+1)} a_i + o(1)
\end{split}
\end{equation*}
in view of $\frac{\de_i|y|}{\de_j}=o(1)$ uniformly for $y$ on compact subsets, \eqref{mujmff} and dominated convergence. Similarly, if $i>j$ then
\begin{equation*}
\begin{split}
P_\e\eta_j(\de_i y)=&\, \bigg[\frac43 \al_j\log\de_j +\frac43 \log\Big(\Big(\frac{\de_j}{\de_i}\Big)^{\al_j} + |y|^{\al_j}\Big)\bigg] \frac{\big(\frac{\de_j}{\de_i}\big)^{\al_j} - |y|^{\al_j} }{\big(\frac{\de_j}{\de_i}\big)^{\al_j} + |y|^{\al_j} } + \frac83\Big(\frac{\de_j}{\de_i}\Big)^{\al_j} \cdot \frac{1}{\big(\frac{\de_j}{\de_i}\big)^{\al_j}+ |y|^{\al_j} }\\
&\, +\frac{\ti\gamma_j}{2\pi } \log(\de_i |y|) + \Big(\frac{8\pi}{3}\al_j - \ti\gamma_j\Big)H(\de_i y ,0) + O(\rho^{\ti\sigma})
\end{split}
\end{equation*}
and
{\small
\begin{equation*}
\begin{split}
&\int_{\Om_{\e} } |x|^{\al_i-2} e^{w_i} \psi_l P_\e\eta_j= \int_{\Om_{i,n} }\Pi_i \Psi_{l,i,n} P_\e\eta_j(\de_i y)
=  \frac43 \int_{\Om_{i,n}}\Pi_i \Psi_{l,i,n}  \log\Big(\Big(\frac{\de_j}{\de_i}\Big)^{\al_j} + |y|^{\al_j}\Big) \frac{\big(\frac{\de_j}{\de_i}\big)^{\al_j} - |y|^{\al_j} }{\big(\frac{\de_j}{\de_i}\big)^{\al_j} + |y|^{\al_j} } \\
&\, +  \frac{\ti\gamma_j}{2\pi}\int_{\Om_{i,n}}\Pi_i\Psi_{l,i,n}(y)\log|y| + o(1)
= -\frac43 \al_j a_i(-4\pi ) + \frac{16\pi(\al_1-2)\be_j}{3(\be_1+1)} a_i + o(1)
\end{split}
\end{equation*}
}
in view of $\frac{\de_j}{\de_i}=o(1)$,
{\small
\begin{equation*}
\begin{split}
\frac43 \al_j\log\de_j \int_{\Om_{i,n} }\Pi_i \Psi_{l,i,n}  \frac{\big(\frac{\de_j}{\de_i}\big)^{\al_j} - |y|^{\al_j} }{\big(\frac{\de_j}{\de_i}\big)^{\al_j} + |y|^{\al_j} }=&\, \frac43\log d_j \int_{\Om_{i,n} }\Pi_i \Psi_{l,i,n}  \frac{\big(\frac{\de_j}{\de_i}\big)^{\al_j} - |y|^{\al_j} }{\big(\frac{\de_j}{\de_i}\big)^{\al_j} + |y|^{\al_j} } - \frac43 \be_j\log\rho \int_{\Om_{i,n} }\Pi_i \Psi_{l,i,n}  \\
&\, +\frac83 \be_j \Big(\frac{\de_j}{\de_i}\Big)^{\al_j} \log\rho \int_{\Om_{i,n} }\Pi_i \Psi_{l,i,n}  \frac{ 1 }{\big(\frac{\de_j}{\de_i}\big)^{\al_j} + |y|^{\al_j} }  = o(1),
\end{split}
\end{equation*}
}
\begin{equation*}
\begin{split}
\int_{\Om_{i,n} }\Pi_i \Psi_{l,i,n} \log\Big(\Big(\frac{\de_j}{\de_i}\Big)^{\al_j} + |y|^{\al_j}\Big) \frac{\big(\frac{\de_j}{\de_i}\big)^{\al_j} - |y|^{\al_j} }{\big(\frac{\de_j}{\de_i}\big)^{\al_j} + |y|^{\al_j} }=&\, -\al_j a_i\int_{\R^2 }\Pi_i Y_{0i} \log |y| +o(1),
\end{split}
\end{equation*}
\eqref{mujmff} and dominated convergence. If $l=0$ for $j$ odd and $l=1$ for $j$ even, we get that
{\small
\begin{equation*}
\begin{split}
&\int_{\Om_\e} \left[K_0 \psi_0+ K_1\psi_1 - |x|^{\al_j-2}e^{w_j}\psi_l\right] P_\e \eta_{j}=  \sum_{i<j}  \frac{16\pi(\al_1-2)\be_j}{3(\be_1+1)} a_i +\frac{16\pi(\al_1-2)\be_j}{3(\be_1+1)} a_j \\
&\,+ \sum_{i>j} \left[\frac{16\pi(\al_1-2)\be_j}{3(\be_1+1)} a_i  + \frac{16\pi }{3} \al_j a_i\right] + o(1) = \frac{16\pi(\al_1-2)\be_j}{3(\be_1+1)}  \sum_{i=1\atop i\ne j}^m   a_i  + \frac{16\pi }{3} \al_j  \sum_{i>j}a_i + o(1) .
\end{split}
\end{equation*}
}
Besides, similarly as above we obtain that
{\small
\begin{align*}
&\int_{\Om_\e} |x|^{\al_j-2} e^{w_j} \( Z_{0j}-\eta_j \) 
= \int_{\Om_\e} |x|^{\al_j-2} e^{w_j}  Z_{0j} - \int_{\Om_\e} |x|^{\al_j-2} e^{w_j} \eta_j=  \int_{B_{r\over \de_j}(0)\sm B_{\e_j\over\de_j}(0)} {2\al_j^2|y|^{\al_j-2}\over(1+|y|^{\al_j})^2}  {1-|y|^{\al_j}\over 1+|y|^{\al_j}}  \\
&- \int_{\Om_{j,n } }  {2\al_j^2|y|^{\al_j-2}\over (1+|y|^{\al_j})^2 } \lf[ {4\over 3}  \log\(\de_j^{\al_j} + \de_j^{\al_j}|y|^{\al_j}\)  Y_{0j}(y) + {8\over 3}{1\over 1+|y|^{\al_j} }\rg]dy + O(\de_j^{\al_j}) \\
&= O(\rho^{\ti\sigma} |\log\rho|)- {4\over 3} \al_j \log \de_j  \int_{\Om_{j,n } }  {2\al_j^2|y|^{\al_j-2}\over (1+|y|^{\al_j})^2 } Y_{0j}(y) \, dy - {4\over 3}  \int_{\Om_{j,n } }  {2\al_j^2|y|^{\al_j-2}\over (1+|y|^{\al_j})^2 }  Y_{0j}(y) \log\(1 +|y|^{\al_j}\) \, dy \\
&\, -  {8\over 3} \int_{\Om_{j,n } }  {2\al_j^2|y|^{\al_j-2}\over (1+|y|^{\al_j})^2 } {1\over 1+|y|^{\al_j} }\, dy = - {8\pi\over 3}\al_j +o(1),
\end{align*}
}
in view of
$$\int_{\R^2}  {2\al_j^2|y|^{\al_j-2}\over (1+|y|^{\al_j})^2 }  Y_{0j}(y) \log\(1 +|y|^{\al_j}\) \, dy  =-2\pi\al_j\qquad \text{and}\qquad\int_{\R^2 }  {2\al_j^2|y|^{\al_j-2}\over (1+|y|^{\al_j})^2 } {1\over 1+|y|^{\al_j} }\, dy =2\pi\al_j.$$
Therefore, we conclude that
{\small
$$o(1)= a_j\lf({4\pi\over 3}\al_j+o(1)\rg) + \frac{16\pi(\al_1-2)\be_j}{3(\be_1+1)} a_j + \frac{16\pi(\al_1-2)\be_j}{3(\be_1+1)}  \sum_{i=1\atop i\ne j}^m   a_i  + \frac{16\pi }{3} \al_j  \sum_{i>j}a_i - \ti c_{l,n} \lf( - {8\pi\over 3}\al_j+o(1)\rg),$$}
and the conclusion follows. 
\end{proof}

Now, by using Claims \ref{5.4}-\ref{claim32} and arguing similarly to the proof of Proposition \ref{p2}, we deduce the a-priori estimate \eqref{estphi}. This finishes the proof.
\end{proof}

\begin{appendix}
\section{\hspace{-0.5cm} }

\begin{lem}\label{estint}
There exist $p_0>1$ close to 1 such that all the following integrals are of order $O\lf(\rho^{p\eta_p}\rg)$ for any $1<p\le p_0$ and for some $ \eta_p>0 $: $(i)\   \de_j^{2-p} \int_{A_j\over \de_j } \big | \frac{|y|^{\al_j-1} }{ (1 +|y|^{\al_j})^{2}}\big|^pdy;$  $(ii)\ 
\rho^{\eta}\de_j^{2-2p}\int_{A_j\over \de_j}\big | \frac{|y|^{\al_j-2} }{(1 +|y|^{\al_j})^{2}}\big|^pdy; $  $(iii) \  \de_j^{2-2p}\big({\de_i\over \de_j}\big)^{\al_i p} \int_{A_j\over \de_j}\big | \frac{ 1  }{ |y|^{\al_j+2} }\big|^pdy$, for $ i<j;$ $\ (iv)\  \de_j^{2-2p}\big({\de_j\over \de_i}\big)^{\al_i p} \int_{A_j\over \de_j}\big |  |y|^{\al_j-2} \big|^pdy$, for $\ j<i;$

\noindent $(v)\   \rho^{(1+\tau)p} \de_j^{2+2\tau p}\int_{A_j\over \de_j}\big | \frac{(1 +|y|^{\al_j})^{2\tau}}{ |y|^{(\al_j-2)\tau} }\big|^pdy$, for $j$ odd and $(vi)\  \rho^{(1+1/\tau)p} \de_j^{2+2p/\tau }\int_{A_j\over \de_j}\big | \frac{(1 +|y|^{\al_j})^{2/\tau}}{ |y|^{(\al_j-2)/\tau} }\big|^pdy$ for $j$ even.

\end{lem}

\begin{proof}[\dem]
It is readily checked that for any $1<p$ and any $j=1,\dots,m$
$$\de_j^{2-p} \int_{A_j\over \de_j } \lf | \frac{|y|^{\al_j-1} }{ (1 +|y|^{\al_j})^{2}}\rg|^p\,dy=O(\de_j^{2-p})=O(\rho^{(2-p)\be_j/\al_j})$$
and
$$\rho^\eta\de_j^{2-2p}\int_{A_j\over \de_j}\lf | \frac{|y|^{\al_j-2} }{(1 +|y|^{\al_j})^{2}}\rg|^p\,dy=O(\rho^\eta \de_j^{2-2p})=O(\rho^{\eta+(2-2p)\be_j/\al_j}).$$
Furthermore, fixing $j=1,\dots,m$ for any $i<j$ we have that
$$ \int_{A_j\over \de_j}\lf | \frac{ 1  }{ |y|^{\al_j+2} }\rg|^p\,dy=O\bigg( \Big({\de_{j-1}\over \de_j}\Big)^{1-p-{\al_ip\over 2}}\bigg).$$
Hence, there exist $\eta_p>0$ such that $ \de_j^{2-2p}\big({\de_i\over \de_j}\big)^{\al_ip}\big({\de_{j-1}\over \de_j}\big)^{1-p-{\al_i\over 2}}=O(\rho^{p\eta_p}),$
since for $p=1$ we have that
$$ \Big({\de_i\over \de_j}\Big)^{\al_i}\Big({\de_{j-1}\over \de_j}\Big)^{-{\al_i\over 2}}=\Big({\de_i\over \de_{j-1}}\Big)^{\al_i}\Big({\de_{j-1}\over \de_j}\Big)^{{\al_i\over 2}}=O\big(\rho^{\al_i({\be_i\over \al_i}-{\be_{j-1}\over \al_{j-1}}) + {\al_i\over 2}({\be_{j-1}\over \al_{j-1}} - {\be_j\over \al_j})}\big)$$
and $ \al_i\Big({\be_i\over \al_i}-{\be_{j-1}\over \al_{j-1}}\Big) + {\al_i\over 2}\Big({\be_{j-1}\over \al_{j-1}} - {\be_j\over \al_j}\Big)>0$, in view of $i\le j-1<j$ and ${\be_l\over \al_l}$ is decreasing in $l$. Similarly, for $j<i$ we have that
\begin{equation*}
\begin{split}
\de_j^{2-2p}\Big({\de_j\over \de_i}\Big)^{\al_i p} \int_{A_j\over \de_j}\lf |  |y|^{\al_j-2} \rg|^p\,dy&=O\bigg(\de_j^{2-2p}\Big({\de_j\over \de_i}\Big)^{\al_ip}\Big({\de_{j+1}\over \de_j}\Big)^{1-p+{\al_i\over 2}}\bigg)
\end{split}
\end{equation*}
and for $p=1$ we find that $ \big({\de_j\over \de_i}\big)^{\al_i}\big({\de_{j+1}\over \de_j}\big)^{{\al_i\over 2}}=\big({\de_{j+1}\over \de_i}\big)^{\al_i}\big({\de_{j}\over \de_{j+1}}\big)^{{\al_i\over 2}}=O\big(\rho^{\al_i({\be_{j+1}\over \al_{j+1}}-{\be_{i}\over \al_{i}}) + {\al_i\over 2}({\be_{j}\over \al_{j}} - {\be_{j+1}\over \al_{j+1}})}\big).$ Now, if $j$ odd we have that
$$\int_{A_j\over \de_j}\lf | \frac{(1 +|y|^{\al_j})^{2\tau}}{ |y|^{(\al_j-2)\tau} }\rg|^p\,dy=O\bigg(  \Big({\de_{j-1}\over \de_j}\Big)^{1+{p \tau} -{\al_jp \tau\over 2}} + \Big({\de_{j+1}\over \de_j}\Big)^{1+{p \tau}+{\al_jp\tau\over 2}} \bigg).$$
Then, we get that $\rho^{(1+\tau)p} \de_j^{2+2\tau p} \Big({\de_{j-1}\over \de_j}\Big)^{1+{p \tau} -{\al_jp \tau\over 2}} = O\big( \rho^{(1+\tau)p + {\be_j\over \al_j}(1+\tau p+{\al_jp\tau \over 2} ) + {\be_{j-1}\over \al_{j-1}} (1+p\tau - {\al_jp\tau \over 2})}\big),$
and for $p=1$ we find that
\begin{equation}\label{ap1}
1+\tau + {\be_j\over \al_j}\Big(1+\tau +{\al_j\tau \over 2} \Big) + {\be_{j-1}\over \al_{j-1}} \Big(1+\tau - {\al_j\tau \over 2}\Big)>0.
\end{equation}
Indeed, $\al_j={\al_{j-1}+2\over \tau}+2$ implies ${\al_j\be_{j-1}\tau\over 2\al_{j-1}}={\be_{j-1}\over 2} + {\be_{j-1}\over \al_{j-1}}(1+\tau).$ Also, $\be_{j-1}={\tau\be_j+\tau+1\over \tau}$ and \linebreak${1\over 2}(1+\tau) + {\be_j\over \al_j}(1+\tau)>0$ implies that
\begin{equation*}
\begin{split}
1+\tau+{\be_j\over\al_j}\Big(1+\tau+{\al_j\tau\over 2}\Big)&=1+\tau+{\be_j\over\al_j}(1+\tau)+{\be_j\tau\over 2}>{\be_j\tau\over 2} + {1\over 2}(1+\tau)\\
&\ge{\be_{j-1}\over 2}= {\al_j\be_{j-1}\tau\over 2 \al_{j-1}}-{\be_{j-1}\over \al_{j-1}}(1+\tau)=-{\be_{j-1}\over \al_{j-1}}\Big(1+\tau - {\al_j\tau \over 2}\Big)
\end{split}
\end{equation*}
and \eqref{ap1} follows. Similarly, we get that
$$\rho^{(1+\tau)p} \de_j^{2+2\tau p} \Big({\de_{j+1}\over \de_j}\Big)^{1+{p \tau} +{\al_jp \tau\over 2}} = O\big( \rho^{(1+\tau)p + {\be_j\over \al_j}(1+\tau p-{\al_jp\tau \over 2} ) + {\be_{j-1}\over \al_{j-1}} (1+p\tau + {\al_jp\tau \over 2})}\big),$$
and for $p=1$,  $ 1+\tau + {\be_j\over \al_j}\big(1+\tau -{\al_j\tau \over 2} \big) + {\be_{j-1}\over \al_{j-1}} \big(1+\tau +{\al_j\tau \over 2}\big)>0,$ by using that $ \al_j={\al_{j+1}-2\over \tau} -2$ and $\be_{j+1}=\tau \be_j -\tau -1$. Therefore, there exist $\eta_p>0$ such that
$$\rho^{(1+\tau)p} \de_j^{2+2p\tau }\int_{A_j\over \de_j}\lf | \frac{(1 +|y|^{\al_j})^{2\tau}}{ |y|^{(\al_j-2)\tau} }\rg|^p\,dy=O\big(\rho^{p\eta_p} \big).$$
Next, similar arguments as above lead us to obtain that for $j$ even, 
$$\ds 1+{1\over \tau} + {\be_j\over \al_j}\Big(1+{1\over \tau} +{\al_j\over 2\tau } \Big) + {\be_{j-1}\over \al_{j-1}} \Big(1+{1\over \tau}- {\al_j\over 2\tau}\Big)>0,$$ 
by using $\al_j=(\al_{j-1}+2)\tau +2$ and $\be_{j-1}={\be_j+\tau+1\over \tau}$. Also, by using that $\al_j=(\al_{j+1}-2) \tau -2$ and $\be_{j}=\tau \be_{j+1} +\tau +1$ it follows that
$ 1+{1\over \tau} + {\be_j\over \al_j}\big(1+{1\over \tau} - {\al_j\over 2\tau } \big) + {\be_{j+1}\over \al_{j+1}} \big(1+{1\over \tau}+ {\al_j\over 2\tau}\big)>0.$
Thus, there exist $\eta_p>0$ such that
$$\rho^{(1+1/\tau)p} \de_j^{2+2p/\tau }\int_{A_j\over \de_j}\lf | \frac{(1 +|y|^{\al_j})^{2/\tau}}{ |y|^{(\al_j-2)/\tau} }\rg|^p\,dy=O\big(\rho^{p\eta_p} \big)$$
in view of
$$ \int_{A_j\over \de_j}\lf | \frac{(1 +|y|^{\al_j})^{2/\tau}}{ |y|^{(\al_j-2)/\tau} }\rg|^p\,dy=O\bigg(\Big({\de_{j-1}\over \de_j}\Big)^{1+{p\over \tau} -{\al_jp\over 2\tau}} + \Big({\de_{j+1}\over \de_j}\Big)^{1+{p\over\tau}+{\al_jp\over 2\tau}} \bigg).$$
This completes the proof.
\end{proof}

\end{appendix}

\ms
\begin{center}
{\bf Acknowledgements}
\end{center}

\noindent The author would like to thank to Professor Angela Pistoia (U. Roma ``La Sapienza'', Italy) for pointing out him problem \eqref{mfsh}. Moreover, the author would like to express his gratitude to Professors Angela Pistoia and Pierpaolo Esposito (U. Roma Tre, Italy) for many stimulating discussions about this problem and related ones. This work has been supported by grant Fondecyt Regular N$^\text{o}$ 1201884, Chile.



\small
\begin{thebibliography}{99}
\bibitem{AP} M. Ahmedou, A. Pistoia, \emph{On the supercritical mean field equation on pierced domains}, Proc. Amer. Math. Soc. {\bf 143} (9) (2015), 3969--3984 

\bibitem{BL} D. Bartolucci, C.S. Lin, \emph{Existence and uniqueness for mean field equations on multiply connected domains at the critical parameter}, Math. Ann. {\bf 359} (2014), 1--44. 

\bibitem{BaPi}D. Bartolucci, A. Pistoia, {\it Existence and qualitative properties of concentrating solutions for the sinh-Poisson equation}, IMA J. Appl. Math. {\bf 72} (2007), no. 6, 706-729.

\bibitem{BaPiWe}T. Bartsch, A. Pistoia, T. Weth, {\it N-vortex equilibria for ideal fluids in bounded planar domains and new nodal solutions of the sinh-Poisson and the Lane-Emden-Fowler equations}. Comm. Math. Phys. {\bf 297} (2010), no. 3, 653--686.

\bibitem{bjmr} L. Battaglia, A. Jevnikar, A. Malchiodi, D. Ruiz, \emph{A general existence result for the Toda system on compact surfaces}, Adv. Math. 285 (2015), 937--979



\bibitem{clmp} E. Caglioti, P.L. Lions, C. Marchioro, M. Pulvirenti, \emph{A special class of stationary flows for two-dimensional Euler equations: a statistical mechanics description}, Comm. Math. Phys. {\bf 143} (1992), 501--525

\bibitem{clmp1} E. Caglioti, P.L. Lions, C. Marchioro, M. Pulvirenti, \emph{A special class of stationery flows for two-dimensional Euler equations: A statistical mechanics description, part II}, Comm. Math. Phys. {\bf 174} (1995), 229--260.

\bibitem{CL1} C.C. Chen, C.S. Lin, \emph{Sharp estimates for solutions of multi-bubbles in compact Riemann surfaces}, Comm. Pure Appl. Math. {\bf 55} (2002), 728--771.

\bibitem{CL2} C.C. Chen, C.S. Lin, \emph{Topological Degree for a mean field equation on Riemann surface}, Comm. Pure Appl. Math. {\bf 56} (2003), 1667--1727.

\bibitem{DEM5} M. del Pino, P. Esposito and M. Musso, \emph{Linearized theory for entire solutions of a singular Liouville equation}, Proc. Amer. Math. Soc. {\bf 140} (2) (2012), 581–588. 

\bibitem{DeKM} M. del Pino, M. Kowalczyk, M. Musso, \emph{Singular limits in Liouville-type equations}, Cal. Var. P.D.E., {\bf 24} (2005), 47--81.

\bibitem{Dja} Z. Djadli, \emph{Existence result for the mean field problem on Riemann surfaces of all genuses}, Commun. Contemp. Math. {\bf 10} (2008), 205-220.

\bibitem{EFP} P. Esposito, P. Figueroa, A. Pistoia, \emph{On the mean field equation with variable intensities on pierced domains}, Nonlinear Analysis {\bf 190} (2020) 111597.

\bibitem{EGP} P. Esposito, M. Grossi,A. Pistoia, \emph{On the existence of blowing-up solutions for a mean field equation}. Ann. Inst. H. Poincaré Anal. Non Linéaire {\bf 22} (2005), no. 2, 227--257.



\bibitem{EW} P. Esposito, J. Wei, \emph{Non-simple blow-up
solutions for the Neumann two-dimensional sinh-Gordon equation}.
Calc. Var. Partial Differential Equations {\bf 34} (2009), no. 3,
341--375.

\bibitem{F} P. Figueroa, \emph{Singular limits for Liouville-type equations on the flat two-torus}, Calc. Var. P.D.E {\bf 49} (2014), no- 1--2, 613--647.

\bibitem{F1} P. Figueroa, \emph{A note on sinh-Poisson equation with variable intensities on pierced domains},  Asymptotic Analysis, {\bf 122} (2021) 327--348.

\bibitem{F2} P. Figueroa, \emph{Bubbling solutions for mean field equations with variable intensities on compact Riemann surfaces}, arXiv:2203.09731, accepted for publication in Journal D'Analyse Mathematique.

\bibitem{FIT} P. Figueroa, L. Iturriaga, E. Topp, \emph{Sign-changing solutions for the sinh--Poisson equation with Robin Boundary condition}, in preparation.
	
\bibitem{GP} M. Grossi, A. Pistoia, \emph{Multiple Blow-Up Phenomena for the Sinh-Poisson Equation}, Arch. Rational Mech. Anal. {\bf 209} (2013) 287--320.

\bibitem{j} A. Jevnikar, \emph{An existence result for the mean field equation on compact surfaces in a doubly supercritical regime}, Proc. Roy. Soc. Edinburgh Sect A {\bf 143} (2013), 1021--1045.

\bibitem{j2} A. Jevnikar, \emph{Blow-up analysis and existence results in the supercritical case for an asymmetric mean field equation with variable intensities}, J. Diff. Eq. 263 (2017), no. 2, 972-1008

\bibitem{jwy1} A. Jevnikar, J. Wei, W. Yang, \emph{Classification of blow-up limits for the sinh-Gordon equation}, Differential Integral Equations {\bf 31} (2018), 657--684.

\bibitem{jwy2} A. Jevnikar, J. Wei, W. Yang, \emph{On the topological degree of the mean field equation with two parameters}, Indiana Univ. Math. J. {\bf 67} (2018), no. 1, 29--88.

\bibitem{jwyz} J. Jost, G. Wang, D. Ye, C. Zhou, \emph{The blow up of solutions of the elliptic sinh-Gordon equation}, Calc. Var. Partial Differential Equations {\bf 31} (2008) no. 2, 263--276.


\bibitem{Mal} A. Malchiodi, \emph{Morse theory and a scalar field equation on compact surfaces}, Adv. Differential Equations {\bf 13} (2008), 1109--1129.


\bibitem{os1} H. Ohtsuka, T. Suzuki, \emph{Mean field equation for the equilibrium turbulence and a related functional inequality}, Adv. Differential Equations {\bf 11} (2006), 281--304.

\bibitem{o} L. Onsager, \emph{Statistical hydrodynamics}. Nuovo Cimento (9) {\bf 6} (1949), 279--287.

\bibitem{pr1} A. Pistoia, T. Ricciardi, \emph{Concentrating solutions for a Liouville type equation with variable intensities in 2D-turbulence}, Nonlinearity {\bf 29} (2016), no. 2, 271--297.

\bibitem{pr2} A. Pistoia, T. Ricciardi, \emph{Sign-changing tower of bubbles for a sinh-Poisson equation with asymmetric exponents}, Discrete Contin. Dyn. Syst. {\bf 37} (2017), 5651--5692.

\bibitem{r} T. Ricciardi, \emph{Mountain pass solutions for a mean field equation from two-dimensional turbulence}, Diff. Int. Eqs. {\bf 20} (2007), 561--575.

\bibitem{rt} T. Ricciardi, R. Takahashi, \emph{Blow-up behavior for a degenerate elliptic sinh-Poisson equation with variable intensities}, Calc. Var. Partial Differential Equations {\bf 55} (2016), Paper No. 152, 25 pp.

\bibitem{rtzz}T. Ricciardi, R. Takahashi, G. Zecca, X. Zhang, \emph{On the existence and blow-up of solutions for a mean field equation with variable intensities}, Atti Accad. Naz. Lincei Rend. Lincei Mat. Appl. {\bf 27} (2016), 413--429.

\bibitem{rz} T. Ricciardi, G. Zecca, \emph{Minimal blow-up masses and existence of solutions for an asymmetric sinh- Poisson equation}, Math. Nachr. {\bf 290} (2017), no. 14-15, 2375-2387



\bibitem{ss} K. Sawada, T. Suzuki, \emph{Derivation of the equilibrium mean field equations of point vortex and vortex filament system}, Theoret. Appl. Mech. Japan {\bf 56} (2008), 285--290.

\end{thebibliography}
\end{document}